\documentclass{article}
\usepackage{geometry}
\usepackage{graphicx,float}
\usepackage{amsmath,amssymb,amsthm,amsfonts}
\usepackage{mathrsfs,cleveref}
\usepackage{color}

\title{Stochastic Cahn--Hilliard Equations from One-Dimensional Ising--Kac--Kawasaki Dynamics}
\author{ Qi Zhang\thanks{ Beijing institute of mathematical sciences and applications, Beijing, 101408 China. Email: \texttt{qzhang@bimsa.cn}}}

\date{\today}

\newtheorem{theorem}{Theorem}
\newtheorem{lemma}{Lemma}
\newtheorem{proposition}[theorem]{Proposition}

\newtheorem{corollary}[theorem]{Corollary}

\newtheorem{definition}{Definition}[section]

\begin{document}

\maketitle

\begin{abstract}
This paper investigates the scaling limit of the one-dimensional lattice Ising--Kac--Kawasaki dynamics near the critical temperature. 
Starting from a martingale formulation for the Kac coarse-grained field, we decompose the microscopic dynamics into a discrete conservative drift and a pure-jump Dynkin martingale. 
Under a critical scaling regime, the nonlinear drift is identified via a conservative multiscale replacement scheme based on the second-order Boltzmann--Gibbs Principle, yielding the cubic conservative term. 
For the fluctuation component, a martingale central limit theorem characterizes the predictable quadratic variation as a divergence-form Gaussian noise. 
By establishing uniform $H^{-1}$ energy estimates and utilizing a compactness argument, we prove that the Kac coarse-grained field converges to the unique solution of a one-dimensional conservative stochastic Cahn--Hilliard equation. 
Furthermore, we demonstrate that the associated canonical equilibrium measure induced by the microscopic dynamics converges weakly to the standard $\phi^4_1$ measure on the conserved-mass hyperplane.
\end{abstract}

% Suggested in the preamble:
% \usepackage{hyperref}

\vskip 0.2truein


\vskip 0.2truein

\section{Introduction}

Phase separation with a conserved order parameter is classically described at the continuum level
by the Cahn--Hilliard equation, or equivalently by Model~B in the
Hohenberg--Halperin classification \cite{CahnHilliard1958,HohenbergHalperin1977,Onuki2002}.
Its stochastic counterpart is expected to govern mesoscopic fluctuations around phase-separating
profiles under a conservation law, and therefore naturally takes the form of a conservative SPDE, namely a stochastic Cahn--Hilliard equation driven by divergence-type noise.
In its deterministic form, it can be interpreted as an $H^{-1}$-gradient flow of a
Ginzburg--Landau-type free energy, which naturally enforces mass conservation.
Random perturbations---motivated by thermal fluctuations and mesoscopic coarse-graining---lead to the stochastic Cahn--Hilliard equation, typically driven by a divergence-form noise so that the conservation law is preserved at the stochastic level \cite{GiacominLebowitz1998I}.
A central problem in nonequilibrium statistical mechanics is to derive such a macroscopic conservative SPDE directly from a microscopic conservative interacting particle system \cite{F89, Spohn1991}.

The Ising–Kac model is a mean-field model with long-range ferromagnetic interactions. 
It was introduced in statistical mechanics for its simplicity and because it provides a framework to rigorously recover the van der Waals theory \cite{kac1963van} of phase transitions.
It is well-known that for the standard one-dimensional Ising model with nearest-neighbor interactions, no phase transition occurs at any $T > 0$ \cite{Ising25}. However, in the Kac-limit $\gamma \to 0$, the Ising-Kac model approaches the mean-field regime, where a sharp second-order phase transition emerges at the critical inverse temperature $\beta_c = 1$ \cite{kac1963van, LP1966}.
In this regime, the microscopic details of the underlying lattice become increasingly irrelevant as long-range fluctuations dominate the spatial structure.
This emergence of scale invariance justifies the transition from a discrete lattice gas to a macroscopic description. 
Under the framework of the Universality Class hypothesis, disparate microscopic Ising-Kac models—provided they share the same spatial dimensions, symmetries, and conservation laws—converge toward the same macroscopic evolution equations, serving as a bridge between discrete particle systems and continuous field theories. 

In this work we consider the one-dimensional dynamical Ising--Kac model under Kawasaki exchange \cite{Kawasaki1966I} on $\Lambda_N := \mathbb{Z}/(2N+1)\mathbb{Z}$.
Let \(\Sigma_N := \{-1,1\}^{\Lambda_N}.\)
We consider a kernel $\kappa_\gamma(z)$ of the form $\kappa_\gamma(z) = \gamma\,\mathfrak K(\gamma z)$ with a Kac scaling parameter $\gamma\in(0,\frac{1}{4})$, 
\[
\kappa_\gamma(z) = \gamma\,\mathfrak K(\gamma z), \quad \kappa_\gamma(z)\ge 0,
\quad 
\sum_{z\in\Lambda_N}\kappa_\gamma(z)=1,
\]
and the Ising--Kac Hamiltonian
\[
H_\gamma(\sigma)
=
-\frac12 \sum_{i,j\in\Lambda_N} \kappa_\gamma(i-j)\,\sigma_i\sigma_j,
\quad \sigma\in\Sigma_N.
\]
For each bond $(i,i+1)$, let $\sigma^{i,i+1}$ denote the configuration obtained from
$\sigma$ by exchanging the spins at sites $i$ and $i+1$, namely
\[
(\sigma^{i,i+1})_j
:=
\begin{cases}
\sigma_{i+1}, & j=i,\\
\sigma_i, & j=i+1,\\
\sigma_j, & j\neq i,i+1.
\end{cases}
\]
We also write
\(
d_i(\sigma):=\sigma_i-\sigma_{i+1}\in\{0,\pm 2\}.
\)
Then the exchange energy difference is defined by
\[
\Delta_i H_\gamma(\sigma)
:=
H_\gamma(\sigma^{i,i+1})-H_\gamma(\sigma).
\]
The Ising--Kac--Kawasaki dynamics $\sigma^N(t)$ is a continuous-time Markov chain on $\Sigma_N$ generated by the energy difference from nearest-neighbour exchanges $\Delta_i H_\gamma(\sigma)$.
The exchange rate is defined as
\[
c_{i,i+1}(\sigma) := \frac{1}{1+\exp\!\bigl(\beta\,\Delta_i H_\gamma(\sigma)\bigr)}.
\]
Then the Markov generator is
\[
(\mathscr L_\gamma^{N} f)(\sigma)
=
\sum_{i \in\Lambda_N} c_{i,i+1}(\sigma)\,\bigl[f(\sigma^{i,i+1})-f(\sigma)\bigr].
\]
A key structural feature is magnetization conservation: we define the oriented microscopic current associated with bond $(i,i+1)$ as
\(
j_i(\sigma):=(\sigma_i-\sigma_{i+1})\,c_{i,i+1}(\sigma)\),
then a direct computation yields the discrete continuity equation
\[
(\mathscr L_\gamma^{N}\sigma_i)(\sigma)=j_{i-1}(\sigma)-j_i(\sigma).
\]
Hence the spin variation at site \(i\) is given by the incoming current through the left bond minus the outgoing current through the right bond, which is the natural lattice analogue of a divergence form.

For a fixed total magnetization $M(N):=\sum_{i \in \Lambda_N} \sigma(\cdot,i)$, we define the canonical sector
\[
\Sigma_{N,M}
:=
\left\{
\sigma\in\Sigma_N:\ \sum_{i\in\Lambda_N}\sigma^N(i)=M(N)
\right\}.
\]
The corresponding canonical Gibbs measure is
\[
\mu_{N,\gamma,\beta}(d\sigma)
:=
\frac{1}{Z_{N,\gamma,\beta}^{M}}
\exp\!\bigl(-\beta H_\gamma(\sigma)\bigr)
\mathbf 1_{\Sigma_{N,M}}(\sigma),
\]
where $Z_{N,\gamma,\beta}^{M}
:=
\sum_{\sigma\in\Sigma_{N,M}}\exp\!\bigl(-\beta H_\gamma(\sigma)\bigr)
\mathbf 1_{\Sigma_{N,M}}(\sigma)$ is the normalizing constant.

Since the Kawasaki exchange \(\sigma\mapsto \sigma^{i,i+1}\) preserves the total magnetization, the dynamics leaves each sector \(\Sigma_{N,M}\) invariant. Moreover, the exchange rates $c_{i,i+1}(\sigma)$ satisfy detailed balance with respect to \(\mu_{N,\gamma,\beta}\). 
Hence, the Ising--Kac--Kawasaki dynamics is reversible on each canonical sector. Furthermore, the total magnetization is strictly conserved for all $t\ge0$:
\[
\sum_{i\in\Lambda_N}\sigma^N(t,i)=\sum_{i\in\Lambda_N}\sigma^N(0,i) = M(N)
\qquad\text{for all }t\ge0.
\]

To connect to a macroscopic field, we introduce the Kac-smoothed spin field
\[
h_\gamma(\sigma;i):=\sum_{j\in\Lambda_N}\kappa_\gamma(i-j)\sigma^N(j), \quad i\in \Lambda_N,
\]
and define the Kac coarse-grained field on $\Lambda_{\varepsilon_{\gamma}}$
\[
X_\gamma(t,\varepsilon_\gamma i) = \delta_\gamma^{-1}
h_\gamma\!\left(\sigma^N \left(\frac{t}{\alpha_\gamma}\right);i\right),
\]
where the space scale $\varepsilon_\gamma := \frac{1}{2N+1}$ so that $\varepsilon_\gamma i\in \mathbb T$, and \((\alpha_\gamma,\delta_\gamma)\) are the time, and amplitude scales.
Since the Kac kernel is normalized by
\(
\sum_{z\in\Lambda_N}\kappa_\gamma(z)=1,
\)
we have $\sum_{i\in\Lambda_N}h_\gamma(\sigma;i)
=
\sum_{j\in\Lambda_N} \sigma^N(j) = M$.
Thus, the Kac coarse-grained field $X_\gamma$ also satisfies a conservation law:
\[
\langle X_\gamma(t),1\rangle_\gamma
=
\langle X_\gamma(0),1\rangle_\gamma = \frac{\varepsilon_{\gamma}}{\delta_{\gamma}} M(N)
\qquad\text{for all }t\in[0,T].
\]

For each smooth test function \(\phi\in C^\infty(\mathbb T)\), set
\[
\langle X_\gamma(t),\phi\rangle_\gamma
:=
\varepsilon_\gamma\sum_{i\in\Lambda_N}
X_\gamma(t,\varepsilon_\gamma i)\phi(\varepsilon_\gamma i).
\]
Using the Dynkin martingale decomposition, we have the canonical drift–martingale splitting:
\begin{equation}\label{eq:dynkin}
    \langle X_\gamma(t),\phi\rangle_\gamma
=
\langle X_\gamma(0),\phi\rangle_\gamma
+
\int_0^t
\Bigl(\alpha_\gamma^{-1}\mathscr L_\gamma^{N}\langle X_\gamma,\phi\rangle_\gamma\Bigr)
\Big(\sigma^N \Big(\frac{s}{\alpha_{\gamma}}\Big)\Big)\,ds
+ M_t^\gamma(\phi),
\end{equation}
where \(M_t^\gamma(\phi)\) is a Dynkin martingale.

This formula already contains the full structure of the limit problem:
the first integral term is the discrete conservative drift, while \(M^\gamma\) is the microscopic noise source.
Under the scaling considered in this work, one expects the limiting dynamics is a one-dimensional stochastic Cahn--Hilliard equation of the form
\begin{equation}
\partial_t X
=
-\nu \Delta^2 X
- A \Delta X
+ \chi \Delta(X^3)
+ \sigma_{*}\nabla\cdot\xi, \quad \int_{\mathbb T} X(\cdot,x) dx \equiv M.
\end{equation}
Here $\xi$ is space--time white noise and $\partial_x\xi$ is the conservative noise forcing.
The coefficients $(\nu,A,\chi,\sigma_*)$ are determined by the Kac kernel, the inverse temperature, and the scaling relations between $(\alpha_\gamma,\varepsilon_\gamma,\delta_\gamma)$.
Here, the fourth-order term is generated by the Kac interaction after expanding the coarse-grained field, while the second-order linear term stems from critical tuning. Furthermore, the cubic conservative drift arises from the nonlinear closure of the microscopic current, and the conservative noise $\nabla \cdot\xi$ originates from the martingale part of the microscopic dynamics.

In one space dimension, the equilibrium Gibbs measure of stochastic Cahn--Hilliard equation is the canonical \(\phi^4_1\) measure on the conserved-mass hyperplane \(V_M := \{ \phi \in L^2(\mathbb T): \langle \phi, 1 \rangle = M\}\):
\begin{equation}\label{phi41law}
    \mu (d\phi) :=
\frac{1}{Z'} \mathbf{1}_{\phi \in V_M} \exp\!\left(
-\frac{2}{\sigma_*^2} \int_{\mathbb T}
\Bigl( -\frac{A}{2}\phi(x)^2+\frac{\chi}{4}\phi(x)^4 \Bigr)dx
\right)
\mathcal{N}(M, \frac{\sigma_*^2}{2 \nu}(-\Delta)^{-1}),
\end{equation}
where $\mathcal{N}(M, \frac{\sigma_*^2}{2 \nu}(-\Delta)^{-1})$ is the Gaussian measure on $L^2(\mathbb T)$ with mean \(M\) and covariance \(\frac{\sigma_*^2}{2\nu}(-\Delta)^{-1}\) on $V_M$, and $Z'$ the normalization constant. 
Equivalently, in formal terms,
\[
\mu(d\phi)\propto
\exp\!\left(
-\frac{2}{\sigma_*^2}\mathcal F(\phi) \right)\mathbf 1_{\{ \langle \phi,1\rangle=M\}}\,\mathcal D \phi,
\]
with free energy
\[
\mathcal F(\phi)
=
\int_{\mathbb T}
\left(
\frac12| \nabla \phi(x)|^2 -\frac{A}{2}\phi(x)^2 +\frac{\chi}{4}\phi(x)^4 \right)\,dx.
\]
In one dimension, the Wick renormalization is unnecessary, and canonical \(\phi^4_1\) measure \(\mu\) is the ergodic and time-reversible measure on \(V_M \subset L^2(\mathbb T)\), see e.g. \cite{Fu91}. Moreover, it is supported on functions of spatial regularity strictly below \(C^{1/2}\), has finite polynomial moments, and is the natural invariant (indeed reversible) Gibbs measure associated with the conservative stochastic Cahn--Hilliard/Model B dynamics.

From the PDE/SPDE viewpoint, stochastic Cahn--Hilliard equations provide a rich testing ground where high-order dissipation, conservation constraints, and nonlinear potentials interact.
Early foundational work developed a semigroup/variational framework and energy estimates for
stochastic Cahn--Hilliard-type dynamics \cite{DaPratoZabczyk1992,DaPratoDebussche1995}.
More recent advances address physically relevant features such as conservative (divergence-type)
noise, singular/logarithmic potentials enforcing physical constraints, and less restrictive
growth assumptions on noise coefficients; see, for instance,
\cite{CuiHong2023,DG2024,GX2020}.
On the long-time side, ergodicity under mass conservation and singular potentials has been
investigated via probabilistic tools such as (log-)Harnack inequalities \cite{GX2020}, while
random dynamical systems methods provide an abstract and robust route to random attractors for broad classes of (locally monotone) SPDEs encompassing Cahn--Hilliard-type models \cite{GS2020}.
In particular, conservative stochastic Cahn--Hilliard equations
in two dimensions and related invariant measures are considered in
\cite{RocknerYangZhu2021}.
These results provide a robust macroscopic analytic framework, but they do not by themselves
explain how the conservative SPDE emerges from a microscopic lattice dynamics.

On the interacting-particle side, hydrodynamic and fluctuation theories for conservative systems have developed a powerful multiscale toolbox based on local equilibrium, the Boltzmann--Gibbs
principle, spectral-gap estimates, and one-block/two-block replacements
\cite{FunakiHandaUchiyama1991,KipnisLandim1999,Funaki2018HydrodynamicLimitExclusion,Spohn1991}.
These methods explain how microscopic conservative exchanges generate macroscopic conservation laws,
and they strongly suggest that the right route to stochastic Cahn–Hilliard/Model~B equation is through a careful closure of the microscopic current and of the local mobility.
Beyond weakly interacting settings, quantitative hydrodynamic limits for conservative dynamics
with strong, finite-range interactions have also become accessible \cite{KwonMenzNam2023HydrodynamicKawasakiStrong},
highlighting the reach of modern two-scale and functional inequality approaches.
However, applying this philosophy to the critical Ising--Kac--Kawasaki regime is nontrivial
because both the nonlinear drift and the noise must be extracted from conservative bond observables.

It is also useful to compare the present conservative setting with recent discrete-to-SPDE results for non-conservative (Ising--Kac--Glauber / Model~A) dynamics.
Bertini, Presutti, Rüdiger, and Saada \cite{bertini1994}, and Fritz and  Rüdiger \cite{fritz1995} have shown that the one-dimensional dynamical Ising–Kac–Glauber model converges to the stochastic Ginzburg–Landau equation.
Mourrat and Weber proved that the two-dimensional dynamical Ising--Kac--Glauber model converges to the
renormalized \(\phi^4_2\) equation \cite{MW0217}, and Grazieschi, Matetski and Weber obtained the
corresponding three-dimensional convergence to \(\phi^4_3\) \cite{GMW2025}.
Hairer and Iberti shows that the sequence of Gibbs measures for a 2D Kac-Ising model converges to the $\phi^4_2$ measure on $\mathbb T^2$ near criticality  \cite{Hairer2018}.
Zhu and Zhu established lattice approximation results for the dynamical \(\phi^4_3\) model
\cite{ZZ18}, while Gubinelli and Hofmanov{\'a} constructed the Euclidean \(\phi^4_3\) field by PDE/Dirichlet-form methods \cite{GH21}.
These works show that renormalized singular SPDEs can indeed arise as scaling limits of discrete systems, and that lattice-level counterterms reflect the structure of the continuum limit.
However, for the rescaling limit of the Ising–Kac–Kawasaki dynamics, even in the one‑dimensional case, a complete result is still lacking. We note that valuable progress has been made in \cite{Ibetri2018}, but the problem remains open.
By contrast, the conservative case requires an additional current-level analysis:
one must close a divergence-form nonlinear drift and a divergence-form noise simultaneously. 

Motivated by these developments, a central challenge remains to rigorously connect lattice Ising--Kac--Kawasaki dynamics near criticality with macroscopic SPDEs, in a way that simultaneously identifies the nonlinear conservative drift and the
divergence-form noise from the microscopic generator and its martingale part.
In this paper we aim to advance the lattice-to-SPDE direction and, building on the same multi-scale replacement philosophy, to push beyond the dynamical limit toward the associated
equilibrium and long-time theory: convergence of discrete (canonical) Gibbs measures to the macroscopic constrained $\phi^4$-type measures, and the emergence of reversibility and ergodic behavior in the scaling limit.

However, deriving this stochastic Cahn--Hilliard equation rigorously from an underlying conservative lattice dynamics remains a delicate problem: the nonlinearity in the
microscopic bond current must be closed in terms of the macroscopic field, and the martingale fluctuations must be identified as a divergence-form Gaussian noise.
Both the nonlinear drift term and the noise term must be derived from microscopic currents.

The drift term in \eqref{eq:dynkin} can be rewritten in conservative form using the microscopic bond current.  
A direct computation yields an explicit expression for the energy difference
$\Delta_i H_\gamma$ across a swap, which in turn permits a Taylor expansion of the heat-bath rate.
This produces a leading nonlinear current term containing the fast local observable
$d^2_i := (\sigma_{i+1} - \sigma_i)^2$, which is not itself a function of the coarse-grained field and must be closed by a conservative multiscale replacement argument.
This yields a closed macroscopic current of the form $X^2\nabla X$, and after taking the outer discrete divergence one obtains the cubic conservative drift $\Delta(X^3)$.
Moreover, the quadratic variation of the discrete martingales $M^\gamma(\phi)$ involves the local mobility observable
\(
q_i:= d^2_i c_{i,i+1}(\sigma) =(\sigma_{i+1} - \sigma_i)^2c_{i,i+1}(\sigma),
\)
and this must be closed by the same conservative replacement mechanism in order to produce a deterministic limiting covariance of divergence type.
A martingale central limit theorem then yields convergence to a Gaussian limit with covariance
$\int_{\mathbb T} |\nabla \phi|^2 dx$, which identifies the noise as the conservative forcing $\nabla\xi$.

In this paper we address these issues for the one-dimensional Ising-Kac model evolving under Kawasaki exchange dynamics.  
Under a suitable scaling, our goal is to show that a Kac coarse-grained field converges to a one-dimensional stochastic Cahn–Hilliard equation.

\paragraph{Assumptions and main result.}

Throughout this paper, we make the following assumptions:

\begin{enumerate}
    \item[({\bf A})](Kernel assumptions) The Kac kernel $\kappa_\gamma(z)$ is of the form
\(
\kappa_\gamma(z)=\gamma\,\mathfrak K(\gamma z),
\quad
\kappa_\gamma(z)=\kappa_\gamma(-z), 
\)
where the nonnegative function
\(
\mathfrak K\in W^{1,1}(\mathbb R)\cap W^{1,\infty}(\mathbb R), \mathfrak K(0) = 0,
\)
and
\[
\sum_{z\in\Lambda_N}\kappa_\gamma(z)=1,\quad  \mathfrak m_2:=\int_{\mathbb R}u^2\,\mathfrak K(u)\,du < \infty.
\]
\item[({\bf B})] (Scaling assumptions) Let $\alpha_\gamma$, $\delta_\gamma$, and $\varepsilon_\gamma := \frac{1}{2N+1}$ be the time, amplitude, and microscopic space scales respectively. As $\gamma \downarrow 0$, we assume the macroscopic Kac limit holds:
\[
\alpha_\gamma \sim \varepsilon_\gamma^{5/2},\qquad
\delta_\gamma \sim \varepsilon_\gamma^{1/4},\qquad
\gamma \sim \varepsilon_\gamma^{3/4}.
\]
We also assume the following precise balance limits hold for some constants $\nu, \chi, \sigma_* > 0$:
\[
\nu_\gamma := \frac{\mathfrak m_{2,\gamma}}{4} \frac{\varepsilon_\gamma^4}{\alpha_\gamma\gamma^2} \to \nu, \quad 
\chi_\gamma := \frac{\beta}{6}\frac{\varepsilon_\gamma^2 \delta_\gamma^2}{\alpha_\gamma} \to \chi, \quad 
\widetilde\sigma_\gamma^2 := \frac{\varepsilon_\gamma^3}{\alpha_\gamma\delta_\gamma^2} \to \sigma_*^2.
\]
Furthermore, to prevent the 2nd-order linear drift from blowing up, we assume the inverse temperature $\beta$ is critically tuned near the critical temperature ($\beta_c = 1$) such that:
\[
 A_\gamma := \frac{\varepsilon_\gamma^2}{\alpha_\gamma} \left( \frac12-\beta\kappa_\gamma(1) - \frac{\beta}{2} \right) \to A \in (0,\infty).
\]
This positive limit is the stable-side non-degeneracy used in the static coercivity estimates below.
\item[({\bf C})] (Initial law and entropy bound) 
Let \(\nu_{N,0} := f_{N,0}\mu_{N,\gamma,\beta}\) be the initial law of $\sigma^N(t)$, where \(\mu_{N,\gamma,\beta}\) is the Gibbs measure associated with \(H_\gamma\).
We assume that for each $N \ge 1$,
\[
H_N( \nu_{N,0}\mid \mu_{N,\gamma,\beta})\le C N.
\]
\end{enumerate}

The following main result of the present paper is that the Kac-smoothed field \(X_\gamma\) converges, in the sense of the associated martingale problem, to the one-dimensional stochastic Cahn--Hilliard / Model~B equation with conservative noise. This theorem is a unified conservative replacement scheme that closes the nonlinear current and
the mobility in parallel, leading to a complete martingale-problem identification of the 1D stochastic Cahn--Hilliard limit equations.

\begin{theorem}\label{thm:main-informal}
Assume that assumptions \(({\bf A})\)-\(({\bf C})\) hold, and after the Fourier extension, the initial data $\widetilde X_\gamma(0)\longrightarrow X_0$ in $H^{-1}(\mathbb T)$, 
and
\(
    \langle X_\gamma(0),1\rangle_\gamma \to \langle X(0),1\rangle := M
\)
as $\gamma \to 0$.
Then after Fourier extension and letting \(\gamma\to0\), the Kac coarse-grained field $X_{\gamma}$ converges in law to $X$ in $L^2(0,T;L^2(\mathbb T))\cap C([0,T],H^{-3}(\mathbb T))$, where $X \in V_M$ is the unique solution of the one-dimensional stochastic Cahn–Hilliard equation
\begin{equation}\label{eq:modelB-intro}
    \partial_t X = -\nu \Delta^2 X - A \Delta X + \chi \Delta(X^3) + \sigma_* \nabla \cdot\xi, \quad \int_{\mathbb T} X(\cdot,x)dx \equiv M.
\end{equation}
Moreover, the induced equilibrium measure $\mu_\gamma := (\widetilde X_\gamma)_\#\mu_{N,\gamma,\beta}$ of the Kac coarse-grained field converges weakly to the canonical \(\phi^4_1\) measure \(\mu\) on the conserved-mass hyperplane \( V_M \).
\end{theorem}

The convergence result in Theorem \ref{thm:main-informal} relies on a particular balance among the time, amplitude, and Kac interaction scales.
In the nonlinear drift generated by the local fluctuation observable \(d_i^2\), the effective size is proportional to \(\frac{\varepsilon_\gamma^3}{\alpha_\gamma}|\Lambda_N|\delta_\gamma^2\). Since \(|\Lambda_N|\sim\varepsilon_\gamma^{-1}\), this gives the order \(\frac{\varepsilon_\gamma^2\delta_\gamma^2}{\alpha_\gamma}\), which is \(O(1)\) under Assumption \(({\bf B})\).
Thus the chosen scaling is one consistent regime in which the cubic conservative drift remains visible in the limit.
At the same time, the Kac range \(\gamma\) is chosen so that the residual terms from the second-order Boltzmann--Gibbs replacement vanish. In particular, the integrated residual error estimated below is of order
\(\frac{\varepsilon_\gamma\gamma^3}{\alpha_\gamma\delta_\gamma}\), which tends to zero under the present assumptions.

The critical \(5/2\) scaling limits \( \alpha_\gamma \sim \varepsilon_\gamma^{5/2}, \delta_\gamma \sim \varepsilon_\gamma^{1/4}, \gamma \sim \varepsilon_\gamma^{3/4}\) in Assumption \(({\bf B})\) should therefore be read as a convenient representative scaling regime satisfying
\[
\varepsilon_\gamma\to0,\qquad
\frac{\varepsilon_\gamma}{\gamma}\to0,\qquad
\frac{\alpha_\gamma}{\gamma}\to0, \quad \text{as} \quad \gamma\to 0.
\]
Heuristically, the temporal scaling \(\alpha_\gamma\sim\varepsilon_\gamma^{5/2}\) corresponds to a dynamical exponent \(z=5/2\). This is slower than the diffusive \(z=2\) scale and is consistent with the additional constraint imposed by Kawasaki conservation near criticality.

It is instructive to compare this choice with the classical results for Ising--Kac--Glauber dynamics by Bertini, Presutti, Rüdiger, and Saada \cite{bertini1994}, and Fritz and Rüdiger \cite{fritz1995}. While the spatial scaling \(\gamma \sim \varepsilon^{3/4}\) and the fluctuation amplitude \(\delta\sim\varepsilon^{1/4}\) are the same as in that non-conservative setting, the time scale changes from \(\alpha\sim\varepsilon^{1/2}\) for Glauber dynamics to \(\alpha\sim\varepsilon^{5/2}\) here. This difference reflects the slower relaxation caused by the $H^{-1}$ type gradient flow structure in the Cahn-Hilliard equitation.

Throughout the paper, all asymptotic notation is understood as \(\gamma\downarrow0\), equivalently along \(\varepsilon_\gamma\downarrow0\).
We write \(a_\gamma\sim b_\gamma\) if
\(a_\gamma/b_\gamma\to1\); for instance \(\alpha_\gamma\sim\varepsilon_\gamma^{5/2}\) means \(\alpha_\gamma/\varepsilon_\gamma^{5/2}\to1\). 
We use the notation $a\lesssim b$ if there exists a constant $C>0$, independent of the variables under consideration, so that $a \leq C\cdot b$, and we denote $a\simeq b$ if $a\lesssim b$ and $b\lesssim a$. 
We write \(a_\gamma=o(b_\gamma)\) if
\(a_\gamma/b_\gamma\to0\), and in particular \(o(1)\) denotes a quantity which vanishes as \(\gamma\downarrow0\).
We write \(\mathbb T=\mathbb R/\mathbb Z\), \(x_i=\varepsilon_\gamma i\), and, for lattice functions on \(\Lambda_{\varepsilon_\gamma}\),
\[
\langle f,g\rangle_\gamma:=\varepsilon_\gamma\sum_{i\in\Lambda_N}f(x_i)g(x_i),
\qquad
\|f\|_{p,\gamma}:=\bigl(\varepsilon_\gamma\sum_{i\in\Lambda_N}|f(x_i)|^p\bigr)^{1/p}
\]
with the usual sup-norm convention when \(p=\infty\).  We use \(\nabla_{\varepsilon_\gamma}\) and \(\Delta_{\varepsilon_\gamma}:=\nabla_{\varepsilon_\gamma}^*\nabla_{\varepsilon_\gamma}\) for the discrete gradient and Laplacian, and \(\widetilde f_\gamma:=\mathrm{Ext}_\gamma f\) for the Fourier extension to \(\mathbb T\).
We denote $\Delta = \partial_x^2$ and $\nabla = \partial_x$.
The Fourier transform on the torus $\mathbb{T}$ is defined with $\hat{f}(k):=(\mathscr{F}_{\mathbb{T}}f)(k)= \int_{\mathbb{T}}e^{-2\pi i k x}f(x)dx$, and the inverse Fourier transform on the torus $\mathbb{T}$ is given by $(\mathscr{F}^{-1}_{\mathbb{T}}\hat{f})(x) = \sum_{k \in \mathbb{Z}}e^{2\pi i kx}\hat{f}(k)$. 
For any $u \in \mathcal{S}'(\mathbb{T})$ and $j\geq -1$, we denote $\Delta_{j}u$ as the $j$th Littlewood-Paley block of $u$. For Besov spaces and Sobolev spaces with $\alpha \in \mathbb R$, we write $C^{\alpha}:=B^{\alpha}_{\infty,\infty}(\mathbb{T})$ as the H\"{o}lder-Besov space, and $H^{\alpha} := B^{\alpha}_{2,2}(\mathbb{T})$ as the Sobolev space. We also use the convention
\(
d\eta
:=
\prod_{u\in\Lambda_\ell}
\bigl(\delta_{-1}(d\eta_u)+\delta_1(d\eta_u)\bigr)
\)
for the counting measure on \(\{-1,1\}^{\Lambda_\ell}\).

The rest of this paper is organized as follows: 
Section~\ref{sec:prel} introduces the Ising--Kac--Kawasaki dynamics, the Dirichlet forms, and basic well-posedness results of one-dimensional stochastic Cahn–Hilliard equation.
Section~\ref{sec:current} rewrites the drift in current form and performs the Taylor expansion leading to the nonlinear current term.
Section~\ref{sec:clean-second-order-bg} establish the second-order Boltzmann--Gibbs Principle and derive the cubic conservative term.
Section~\ref{sec:noise} computes the quadratic variation of the Dynkin martingale, proves the mobility replacement, and identifies the conservative Gaussian noise. We establish the $H^{-1}$ energy estimates for $Y_{\gamma}$ in Section~\ref{subsec:energy-estimates}.
Finally, Section~\ref{sec:limit} combines tightness and identification to conclude the proof of the main theorem.

\section{Preliminary}\label{sec:prel}

\subsection{Fourier extension from the lattice to the torus}
\label{subsec:fourier-extension}

To place the whole argument in a fixed function space over \(\mathbb T\), we need to identify lattice fields
with trigonometric polynomials via discrete Fourier interpolation, following the Fourier-model viewpoint
used in weakly nonlinear lattice-to-SPDE arguments.

Let
\(
\Lambda_{\varepsilon_\gamma}:=\varepsilon_\gamma \mathbb Z/\mathbb Z
=
\Bigl\{-\frac{1}{2}, -\frac{1}{2} + \varepsilon_\gamma, \dots, -\varepsilon_\gamma, 0,\varepsilon_\gamma,\dots, \frac{1}{2}\Bigr\}.
\)
We also define the finite-dimensional trigonometric polynomial space
\[
\mathcal T_\gamma
:=
\mathrm{span}\{e_k(x):=e^{2\pi i kx},\ k\in\mathbb Z_\gamma\}
\subset C^\infty(\mathbb T).
\]
For a lattice function \(f:\Lambda_{\varepsilon_\gamma}\to\mathbb C\), define its discrete Fourier coefficients by
\[
\widehat f_\gamma(k):=
\varepsilon_\gamma\sum_{j\in\Lambda_{\varepsilon_\gamma}}f(j)e^{-2\pi i kj},
\quad k\in\mathcal B_\gamma.
\]
where
\(
\mathcal B_\gamma:=\{k\in\mathbb Z:\ |k|\le c\,\gamma^{-1}\}
\)
is the Brillouin zone.
Its Fourier extension to \(\mathbb T\) is
\[
(\mathrm{Ext}_\gamma f)(x)
:=
\sum_{k\in\mathcal B_\gamma}\widehat f_\gamma(k)e_k(x),\quad
e_k(x):=e^{2\pi i kx}.
\]

Let \(x_i:=\varepsilon_\gamma i \in \mathbb T \), and define the discrete gradient and Laplacian as
\[
\nabla_{\varepsilon_\gamma} f(x_i)
:=
\frac{f(x_{i+1})-f(x_i)}{\varepsilon_\gamma},
\quad
\Delta_{\varepsilon_\gamma}:=\nabla_{\varepsilon_\gamma}^*\nabla_{\varepsilon_\gamma}, \quad f \in \mathcal T_\gamma.
\]
We also set
\(
A_\gamma:=-\Delta_{\varepsilon_\gamma},
\)
Define the extended discrete Laplacian \(\Delta_{\varepsilon_\gamma}\) on \(\mathcal T_\gamma\) by Fourier multiplier:
\[
 \Delta_{\varepsilon_\gamma} e_k =  - \frac{4}{\varepsilon_\gamma^2}\sin^2(\pi k\varepsilon_\gamma)e_k,
\qquad k\in \mathcal{B}_\gamma.
\]
This is the torus realization of \(A_\gamma=-\Delta_{\varepsilon_\gamma}\).

Now we show the symbol comparison of \(A_\gamma=-\Delta_{\varepsilon_\gamma}\) and second-order consistency.

\begin{lemma}\label{lem:symbol_bilaplacian}
Let
\(
\lambda(k):=(2\pi k)^4,
\)
and \(\lambda_\gamma(k):=\frac{16}{\varepsilon_\gamma^4}\sin^4(\pi k\varepsilon_\gamma)\).
Then for every $k \in \mathcal{B}_{\gamma}$ one has
\[
c\,|k|^4\le \lambda_\gamma(k)\le C\,|k|^4, \quad \text{and} \quad |\lambda_\gamma(k)-\lambda(k)|
\le
C\,\varepsilon_{\gamma}^2 |k|^6.
\]
where the constants $c,C>0$ are independent of $\varepsilon_{\gamma}$ and $k$.
\end{lemma}

\begin{proof}
For $|k|\le 2N+1 =(\varepsilon_{\gamma})^{-1}$ we have
$|\pi\varepsilon k|\le \pi/2$. Hence
\(
\frac{2}{\pi}|y|\le |\sin y|\le |y|, |y|\le \frac{\pi}{2}.
\)

With $y=\pi\varepsilon_{\gamma} k$, this yields
\[
16\varepsilon_{\gamma}^{-4}\Bigl(\frac{2}{\pi}\pi\varepsilon_{\gamma} |k|\Bigr)^4
\le
\lambda_{\varepsilon_{\gamma}}(k)
\le
16\varepsilon^{-4}(\pi \varepsilon_{\gamma} |k|)^4.
\]
For the consistency estimate, note that
\( \sin y=y-\frac{y^3}{6}+O(y^5)\), as \( y\to 0\).
Therefore
\[
\lambda_\gamma(k)
=
16\varepsilon_{\gamma}^{-4}\bigl((\pi\varepsilon_{\gamma} k)^4+O((\pi\varepsilon_{\gamma} k)^6)\bigr)
=
(2\pi k)^4 + O(\varepsilon_{\gamma}^2 |k|^6).
\]
This proves the claim.
\end{proof}

Now we show the compatibility of the extension with the discrete Laplacian.

\begin{lemma}
\label{lem:fourier-extension-operator}
For every lattice field \(u\) and every polynomial \(P\),
\(
 \mathrm{Ext}_\gamma(P(A_\gamma)u)=P( A_\gamma) \mathrm{Ext}_\gamma u.
\)
\end{lemma}

\begin{proof}
The operator \(A_\gamma\) is diagonal in the discrete Fourier basis, with symbol \(\lambda_\gamma(k)\).
Since \(\mathrm{Ext}_\gamma\) preserves discrete Fourier coefficients, the equality
\(
 \mathrm{Ext}_\gamma(A_\gamma u)=  A_\gamma \mathrm{Ext}_\gamma u
\)
holds modewise. The polynomial functional calculus then follows immediately.
\end{proof}

For \(U=\sum_{k\in\mathbb Z_\gamma}\widehat U(k)e_k\in\mathcal T_\gamma\), we define the Sobolev norms 
\[
\|U\|_{H^s(\mathbb T)}^2
:=
\sum_{k\in\mathbb Z_\gamma}(1+|k|^2)^s|\widehat U(k)|^2.
\]
On the lattice side, define the discrete Sobolev norms by
\[
\|u\|_{H^s_\gamma}^2
:=
\sum_{k\in\mathbb Z_\gamma}(1+\lambda_\gamma(k))^s|\widehat u_\gamma(k)|^2.
\]

Let \((\varrho_q)_{q\ge -1}\) be a standard smooth dyadic partition of unity with support \(\mathcal A_j \subset \mathbb Z\), and let
\(
\Delta_j u
:=
\sum_{k\in\mathbb Z}\varrho_q(k)\widehat u(k)e_k
\)
be the Littlewood--Paley blocks on \(\mathbb T\).
We define the Besov norm on \(\mathbb T\) by
\[
\|u\|_{B^\beta_{p,q}}
:= \left( \sum_{j \ge -1}2^{\beta q j }\|\Delta_j u\|_{L^p(\mathbb T)}^q \right)^{1/q}.
\]
For \(\beta\in\mathbb R\) and \(1\le p,q<\infty\), define the discrete Besov norm by
\[
\|f\|_{B^\beta_{p,q,\gamma}}
:=
\|\mathrm{Ext}_\gamma f\|_{B^\beta_{p,q}(\mathbb T)}
=
\left(
\sum_{j \ge -1}2^{\beta qj}\|\Delta_j \mathrm{Ext}_\gamma f\|_{L^p(\mathbb T)}^q
\right)^{1/q}.
\]

\begin{lemma}[Equivalence of discrete and continuum Sobolev norms]
\label{lem:fourier-norm-equivalence}
For every fixed \(s\in[-2,2]\), there exist constants \(c_s,C_s>0\), independent of \(\gamma\), such that
for every lattice field \(u\),
\[
c_s\,\|u\|_{H^s_\gamma}
\le
\| \mathrm{Ext}_\gamma u\|_{H^s(\mathbb T)}
\le
C_s\,\|u\|_{H^s_\gamma}.
\]
In particular,
\[
\|\mathrm{Ext}_\gamma u\|_{H^{-1}(\mathbb T)}\asymp \|u\|_{-1,\gamma},
\qquad
\|\mathrm{Ext}_\gamma u\|_{H^1(\mathbb T)}\asymp \|u\|_{1,\gamma}.
\]
\end{lemma}

Now we show the approximation and block-wise estimates of semigroup.

\begin{lemma}\label{ass:semigroup}
Let $\lim_{\gamma \to 0}\nu_{\gamma} = \nu$, 
\(
S_\gamma(t):=e^{-t\nu_{\gamma}\Delta_{\varepsilon_\gamma}^2}\), and \(S(t):=e^{-t\nu\Delta^2}.\)
Under Assumption \( ({\bf A}) \), the following hold.

\begin{enumerate}
\item[(i)] For every fixed \(K>0\),
\[
\sup_{|k|\le K}\sup_{0\le t\le T}
\bigl|e^{-t\lambda_\gamma(k)}-e^{-t(2\pi k)^4}\bigr|
\longrightarrow0.
\]

\item[(ii)] For every dyadic block \(q\ge -1\),
\[
\sup_{\gamma}\sup_{k\in \mathrm{supp}\,\varrho_q\cap\mathcal B_\gamma}
\int_0^T e^{-2(T-r)\lambda_\gamma(k)}|k|^2\,dr
\lesssim 2^{-q}.
\]

\item[(iii)] There exists \(\vartheta\in(0,\frac14)\) such that for every \(q\ge -1\),
\[
\sup_{\gamma}
\sup_{k\in\mathrm{supp}\,\varrho_q\cap\mathcal B_\gamma}
\int_0^s
\bigl|e^{-(t-r)\lambda_\gamma(k)}-e^{-(s-r)\lambda_\gamma(k)}\bigr|^2|k|^2\,dr
\le
C|t-s|^\vartheta 2^{-q(2-4\vartheta)},
\]
and
\[
\sup_{\gamma}
\sup_{k\in\mathrm{supp}\,\varrho_q\cap\mathcal B_\gamma}
\int_s^t
e^{-2(t-r)\lambda_\gamma(k)}|k|^2\,dr
\le
C|t-s|^\vartheta 2^{-q(2-4\vartheta)}.
\]
\end{enumerate}
\end{lemma}

\begin{proof}
Item (i) follows immediately from the convergence of the symbols \(\lambda_\gamma(k)\to (2\pi k)^4\) on compact frequency sets. Item (ii) follows from
\[
\int_0^T e^{-2(T-r)\lambda_\gamma(k)}|k|^2\,dr
\lesssim \frac{|k|^2}{\lambda_\gamma(k)}
\lesssim |k|^{-2},
\]
and the fact that \(|k|\sim 2^q\) on the support of \(\varrho_q\).
Item (iii) is the standard semigroup increment estimate:
for the first integral, one writes
\[
e^{-(t-r)\lambda}-e^{-(s-r)\lambda}
=
e^{-(s-r)\lambda}(e^{-(t-s)\lambda}-1)
\]
and uses \(|e^{-a}-1|\lesssim a^\vartheta\) for \(a\ge0\); the second follows by direct integration.
\end{proof}

\subsection{Well-posedness of the $1$-d stochastic Cahn--Hilliard equation}\label{subsec:1d-sch-background}

We briefly recall the well-posedness result for the $1$d stochastic Cahn--Hilliard equation
\begin{equation}\label{eq:intro-1d-sch}
\partial_t X = -\nu \Delta^2 X - A \Delta X + \chi \Delta(X^3) + \sigma_* \nabla \cdot\xi.
\end{equation}
where \(\xi\) is a space-time white noise on $\mathbb R^{+} \times \mathbb T$, \(\nu>0\), \(\chi>0\), \(\sigma_*>0\), and \(A\in\mathbb R\).

In one space dimension, the stochastic Cahn--Hilliard equation \eqref{eq:intro-1d-sch}driven by additive conservative noise is substantially less singular than its higher-dimensional counterparts: the linear stochastic convolution already has positive spatial regularity, so the cubic term can be interpreted by classical multiplication and no renormalization is needed. 
For our purposes, we also apply the Da Prato--Debussche decomposition $X=Y+Z$.
Given an initial datum \(X_0\), we split the solution as
\(
X=Y+Z,
\)
where \(Z\) is the solution of the linear stochastic equation
\begin{equation}
\label{eq:intro-Z}
dZ= -\nu \Delta^2 Z dt+\sigma_*\nabla\cdot\xi(t)dt,
\qquad
Z(0)=0,
\end{equation}
and \(Y\) solves the random deterministic equation
\begin{equation}
\label{eq:intro-Y}
\partial_t Y
=-\nu \Delta^2  Y - A\Delta (Y+Z)
+\chi\,\partial_x^2\bigl((Y+Z)^3\bigr),
\qquad
Y(0)=X(0).
\end{equation}
In one dimension, the stochastic convolution in \eqref{eq:intro-Z} is spatially regular enough
to be treated pathwise as a function.
More precisely, for every \(\alpha<\frac12\) and every \(p\ge 1\), one expects
\begin{equation}
\label{eq:intro-Z-est}
\mathbb E\Bigl[\sup_{t\in[0,T]}\|Z(t)\|_{C^\alpha(\mathbb T)}^p\Bigr]<\infty,
\end{equation}
and 
\(
Z\in C([0,T], C^\alpha(\mathbb T)), \text{ for every } \alpha<\frac12,
\)
almost surely.
In particular, \(Z^2\) and \(Z^3\) are classical pointwise products, which is the key simplification specific to the one-dimensional case.

The shifted equation for \(Y\) becomes a random but classical semilinear fourth-order parabolic equation, and standard \(H^{-1}\)-energy estimates in \(H^{-1}\), \(H^{1}\), and \(L^4\) become available. 
Testing \eqref{eq:intro-Y} against \((-\Delta)^{-1}Y\) yields the formal coercive identity
\begin{align}
\frac12\frac{d}{dt}\|Y(t)\|_{H^{-1}}^2
+\nu \|Y(t)\|_{H^1}^2
&=
- A\,\|Y(t)\|_{L^2}^2
- \chi\int_{\mathbb T} Y(t,x)\,(Y(t,x)+Z(t,x))^3\,dx.
\label{eq:intro-energy}
\end{align}
Using Young's inequality,
\[
-\int_{\mathbb T}Y(Y+Z)^3\,dx
\le
-c\|Y\|_{L^4}^4
+
C\bigl(1+\|Z\|_{L^\infty}^4\bigr),
\]
so that one obtains the a priori estimate
\begin{equation}
\label{eq:intro-Y-est}
\sup_{t\in[0,T]}\|Y(t)\|_{H^{-1}}^2
+
\int_0^T \|Y(s)\|_{H^1}^2\,ds
+
\int_0^T \|Y(s)\|_{L^4}^4\,ds
\le
C_T\Bigl(1+\|X_0\|_{H^{-1}}^2+\sup_{s\in[0,T]}\|Z(s)\|_{L^\infty}^4\Bigr).
\end{equation}
We summarize the well-posedness of the one-dimensional stochastic Cahn--Hilliard equation as follows.

\begin{proposition}\label{thm:wellposedness-SCH-1d}
Let \(T>0\), $M \in \mathbb R$, \(X_0\in H^{-1}(\mathbb T)\) with $\int_{\mathbb T} X(0) dx = M$. Let $\nu>0$, $A\in\mathbb R$, $\chi >0$, and $\sigma_*\ge0$.
Then the stochastic Cahn--Hilliard equation has a unique solution \(X=Y+Z\) in the class
\[
X\in
C([0,T],H^{-1}(\mathbb T))
\cap
L^4((0,T)\times\mathbb T)
\quad\text{a.s.},
\]
with conservative total magnetization: $\int_{\mathbb T} X(t) dx = M, \forall t >0$.

Moreover, $X$ admits the Da-Prato–Debussche decomposition $X=Y+Z$, where the stochastic convolution
\[
Z\in C([0,T],\mathcal C^\alpha(\mathbb T))\cap C([0,T],H^{-1}(\mathbb T))
\quad\text{a.s.}, \quad\alpha<\frac12,
\]
and for almost every realization of \(Z\), the shifted term
\[
Y\in
C([0,T],H^{-1}(\mathbb T))
\cap
L^2(0,T;H^1(\mathbb T))
\cap
L^4((0,T)\times\mathbb T)
\quad\text{a.s.}.
\]
\end{proposition}

For the discrete Kawasaki--Kac dynamics, the full field \(X_\gamma\) contains a martingale part
whose \(H^{-1}\)-It\^o correction is not uniformly bounded in \(\gamma\).
Accordingly, the direct energy method on \(X_\gamma\) is not the correct route.
Instead, one proves convergence of the linear discrete stochastic part \(Z_\gamma\),
establishes uniform energy estimates and tightness for the deterministic remainder \(Y_\gamma\),
and finally reconstructs the limit through
\[
X_\gamma=Y_\gamma+Z_\gamma \rightarrow Y+Z=X.
\]
This is the conservative analogue of the shifted-equation strategy used throughout the modern analysis of weakly nonlinear stochastic evolution equations.

\subsection{Dirichlet forms, entropy inequality, and the Kipnis–Varadhan bound}

On each invariant sector \(\Sigma_{N,M}:=
\left\{
\sigma\in\Sigma_N:\ \sum_{i\in\Lambda_N}\sigma^N(i)=M(N)
\right\}\), the canonical Gibbs measure \(\mu_{N,\gamma,\beta}\) is the natural reversible reference measure.
In this subsection we formalize the canonical Gibbs measures naturally associated with the conservative Kawasaki dynamics, both on the full torus and on finite blocks, together with the corresponding Dirichlet forms and Fisher information.
Standard references are Guo--Papanicolaou--Varadhan \cite{GPV88}, Yau \cite{Yau91},
Kipnis--Varadhan \cite{KipnisVaradhan1986}, and Kipnis--Landim \cite{KipnisLandim1999}.

\begin{definition}[Dirichlet form on the canonical sector]
\label{def:global-canonical-dirichlet}
For \(f:\Sigma_{N,M}\to\mathbb R\), define the global Dirichlet form
\[
D^{N,\gamma}(f)
:=
-\bigl\langle f, \mathscr L_\gamma^{N} f\bigr\rangle_{\mu_{N,\gamma,\beta}}
= \sum_{i\in\Lambda_N}D_i^{N,\gamma}(f),
\]
where the bond Dirichlet form is
\[
 D_i^{N,\gamma}(f)
:=
\frac12
\int_{\Sigma_{N,M}}
c_{i,i+1}(\sigma)\,
\bigl(f(\sigma^{i,i+1})-f(\sigma)\bigr)^2\,
\mu_{N,\gamma,\beta}(d\sigma).
\]
\end{definition}

Using the Dirichlet form, we can define the Fisher information of the Ising–Kac–Kawasaki dynamics.
\begin{definition}
For a density \(f\) with respect to \(\mu_{N,\gamma,\beta}\), we define the bond Fisher information by
\begin{equation}
    I_{i,i+1}^{N,\gamma}(f) := D_{i}^{N,\gamma}(\sqrt f)
=
\frac12 \int_{\Sigma_{N,M}} c_{i,i+1}(\sigma)
\bigl(\sqrt{f(\sigma^{i,i+1})}-\sqrt{f(\sigma)}\bigr)^2\,
\mu_{N,\gamma,\beta}(d\sigma).
\end{equation}
The corresponding global Fisher information is
\[
I^{N,\gamma}(f):= \sum_{i\in\Lambda_N} I_{i,i+1}^{N,\gamma}(f) = \sum_{i \in \Lambda_N} D_i^{N,\gamma}(\sqrt f),
\]
for every density \(f\ge0\) with \(\int_{\Sigma_{N,M}} f\,d\mu_{N,\gamma,\beta}=1\).
\end{definition}

Fix a block center \(i\in\Lambda_N\) and a block size
\(
\Lambda_\ell:=\{-\ell,\dots,\ell\},
\Lambda_{i,\ell}:=i+\Lambda_\ell.
\)
We use local coordinates
\(
\eta_u:=\sigma_{i+u}, u\in\Lambda_\ell,
\)
and write
\(
\sigma=\eta\oplus_i\xi,
\)
where
\(
\xi\in\{-1,1\}^{\Lambda_N\setminus\Lambda_{i,\ell}}
\) 
is the exterior configuration.
The Ising--Kac Hamiltonian restricted to the block can be written as
\begin{equation}
\label{eq:frozen-exterior-Hamiltonian}
H_{i,\ell}^{\gamma,\xi}(\eta)
=
H_{\gamma,\ell}^{\mathrm{blk}}(\eta)
-
\sum_{u\in\Lambda_\ell}
b_{i,\xi}^{N,\gamma}(u)\eta_u
+
\mathrm{const}(\xi),
\end{equation}
where
\[
H_{\gamma,\ell}^{\mathrm{blk}}(\eta)
:=
-\frac12
\sum_{u,v\in\Lambda_\ell}
\kappa_\gamma(u-v)\eta_u\eta_v, \qquad b_{i,\xi}^{N,\gamma}(u) := \sum_{z\notin\Lambda_{i,\ell}}
\kappa_\gamma(i+u-z)\xi_z,
\]
and the term \(\mathrm{const}(\xi)\) is independent of \(\eta\) and will be absorbed into the normalizing constants.

For
\[
m\in\mathcal M_\ell
:=
\left\{
-1,-1+\frac{2}{2\ell+1},\dots,1
\right\},
\]
define the canonical sector
\[
\Sigma_{\ell,m}^{\mathrm{can}}
:=
\left\{
\eta\in\{-1,1\}^{\Lambda_\ell}:
\frac{1}{2\ell+1}
\sum_{u\in\Lambda_\ell}\eta_u=m
\right\}.
\]
The frozen-exterior local canonical Gibbs measure is
\begin{equation}
\label{eq:frozen-local-canonical}
\mu_{i,\ell,m}^{\mathrm{can},\gamma,\xi}(d\eta)
:=
\frac{1}{Z_{i,\ell,m}^{\gamma,\xi}}
\mathbf 1_{\Sigma_{\ell,m}^{\mathrm{can}}}(\eta)
\exp\left\{
-\beta H_{\gamma,\ell}^{\mathrm{blk}}(\eta)
+
\beta\sum_{u\in\Lambda_\ell}
b_{i,\xi}^{N,\gamma}(u)\eta_u
\right\}
d\eta.
\end{equation}
This is precisely the conditional Gibbs measure on the block after fixing the exterior configuration \(\xi\) and the block magnetization \(m\).

We introduce the simplified block canonical Gibbs measure without the frozen exterior field:
\begin{equation}\label{eq:simplified-block-canonical}
\mu_{\gamma,\ell,m}^{\mathrm{can}}(d\eta)
:=
\frac{1}{Z_{\gamma,\ell,m}^{\mathrm{can}}}
\mathbf 1_{\Sigma_{\ell,m}^{\mathrm{can}}}(\eta)
\exp\left\{
-\beta H_{\gamma,\ell}^{\mathrm{blk}}(\eta)
\right\}
d\eta.
\end{equation}

Now we show that as $\gamma\ell^2 \to 0$, we can remove the frozen exterior field.
\begin{lemma}\label{lem:remove-frozen-field}
For fixed \(\ell\), we have
\[
\sup_{u\in\Lambda_\ell}
\left|
b_{i,\xi}^{N,\gamma}(u)
-
\bar b_{i,\xi}^{N,\gamma}
\right|
\le
C\gamma\ell,
\qquad
\bar b_{i,\xi}^{N,\gamma}
:=
\frac{1}{2\ell+1}
\sum_{u\in\Lambda_\ell}
b_{i,\xi}^{N,\gamma}(u).
\]
Consequently, for every \(m\in\mathcal M_\ell\),
\begin{equation}
\label{eq:local-can-comparison}
e^{-C\gamma\ell^2}
\mu_{\gamma,\ell,m}^{\mathrm{can}}(d\eta)
\le
\mu_{i,\ell,m}^{\mathrm{can},\gamma,\xi}(d\eta)
\le
e^{C\gamma\ell^2}
\mu_{\gamma,\ell,m}^{\mathrm{can}}(d\eta).
\end{equation}
In particular,
\[
\sup_{i,\xi,m}
\left\|
\mu_{i,\ell,m}^{\mathrm{can},\gamma,\xi}
-
\mu_{\gamma,\ell,m}^{\mathrm{can}}
\right\|_{\mathrm{TV}}
\le
C\gamma\ell^2.
\]
\end{lemma}

\begin{proof}
For \(u,v\in\Lambda_\ell\),
\[
\left|
b_{i,\xi}^{N,\gamma}(u)
-
b_{i,\xi}^{N,\gamma}(v)
\right|
\le
\sum_{z\notin\Lambda_{i,\ell}}
\left|
\kappa_\gamma(i+u-z)
-
\kappa_\gamma(i+v-z)
\right|.
\]
Since
\(
\kappa_\gamma(z)=\gamma\mathfrak K(\gamma z),
\)
we have
\(
|\kappa_\gamma(a)-\kappa_\gamma(b)|
\le
C\gamma^2|a-b|.
\)
Then summing in \(z\) and using \(\mathfrak K\in W^{1,1}\cap W^{1,\infty}\), we obtain
\[
\sup_{u\in\Lambda_\ell}
\left|
b_{i,\xi}^{N,\gamma}(u)
-
\bar b_{i,\xi}^{N,\gamma}
\right|
\le C\gamma\ell.
\]
On the fixed-magnetization sector \(\Sigma_{\ell,m}^{\mathrm{can}}\), the constant field term satisfies
\[
\bar b_{i,\xi}^{N,\gamma}
\sum_{u\in\Lambda_\ell}\eta_u
=
\bar b_{i,\xi}^{N,\gamma}(2\ell+1)m,
\]
and is therefore absorbed into the normalizing constant. Hence
\[
\left| \sum_{u\in\Lambda_\ell}
\left(
b_{i,\xi}^{N,\gamma}(u)
-
\bar b_{i,\xi}^{N,\gamma}
\right)\eta_u \right| \le C\gamma\ell^2
\]
This gives
\eqref{eq:local-can-comparison}.
\end{proof}

For an internal bond \((j,j+1)\), \(|j|\le\ell-1\), define the exact frozen-exterior local exchange rate by
\[
c_{j,j+1}^{(\ell),\gamma,\xi}(\eta)
:=
c_{i+j,i+j+1}(\eta\oplus_i\xi).
\]
Now we show the projection inequality for the local Fisher information.

\begin{lemma}[One-block projection of Fisher information]
\label{lem:clean-one-block-FI-projection}
Fix \(t\in[0,T]\), and disintegrate
\[
\mu_{N,\gamma,\beta}(d\sigma)
=
\nu_{i,\ell}^{N,\gamma}(d\xi)\,
\mu_{i,\ell,m(\xi)}^{\mathrm{can},\gamma,\xi}(d\eta),
\qquad
\sigma=\eta\oplus_i\xi.
\]
Write
\[
f_{N,t}(\eta\oplus_i\xi)
=
f_{i,\ell,t}^{\mathrm{out}}(\xi)\,
q_{i,\ell,t}^{\gamma,\xi}(\eta),
\]
where \(q_{i,\ell,t}^{\gamma,\xi}\) is a density with respect to
\(\mu_{i,\ell,m(\xi)}^{\mathrm{can},\gamma,\xi}\).  
Define
\[
\mathcal I_{i,\ell}^{\gamma,\xi}(q)
:=
\sum_{|u|\le\ell-1}
\frac12
\int
c_{u,u+1}^{(\ell),\gamma,\xi}(\eta)
\left(
\sqrt{q(\eta^{u,u+1})}-\sqrt{q(\eta)}
\right)^2
\mu_{i,\ell,m(\xi)}^{\mathrm{can},\gamma,\xi}(d\eta).
\]
Then
\begin{equation}\label{eq:projection-ineq-summed}
    \varepsilon_\gamma\sum_i
\int
f_{i,\ell,t}^{\mathrm{out}}(\xi)
\mathcal I_{i,\ell}^{\gamma,\xi}
\bigl(q_{i,\ell,t}^{\gamma,\xi}\bigr)
\nu_{i,\ell}^{N,\gamma}(d\xi) \le C\varepsilon_\gamma\ell\,I^{N,\gamma}(f_{N,t}).
\end{equation}
\end{lemma}

\begin{proof}
For a fixed block center \(i\), exterior configuration \(\xi\), and internal
bond \((u,u+1)\), the local exchange \(\eta\mapsto\eta^{u,u+1}\) is exactly the
global exchange \(\sigma\mapsto\sigma^{i+u,i+u+1}\).  The frozen local exchange
rate agrees with the corresponding global rate:
\[
c_{u,u+1}^{(\ell),\gamma,\xi}(\eta)
=
c_{i+u,i+u+1}(\eta\oplus_i\xi).
\]
Since
\(
f_{N,t}(\eta\oplus_i\xi)
=
f_{i,\ell,t}^{\mathrm{out}}(\xi)
q_{i,\ell,t}^{\gamma,\xi}(\eta),
\)
we get
\[
\int
f_{i,\ell,t}^{\mathrm{out}}(\xi)
I_{u,u+1}^{(\ell),\gamma,\xi}
\bigl(q_{i,\ell,t}^{\gamma,\xi}\bigr)
\nu_{i,\ell}^{N,\gamma}(d\xi)
\le
I_{i+u,i+u+1}^{N,\gamma}(f_{N,t}).
\]
Summing over \(i\) and \(|u|\le\ell-1\), each global bond is counted at most
\(C\ell\) times.  This proves the estimate.
\end{proof}

We define the local Dirichlet form and the local Fisher information associated with the canonical block Gibbs measures $\mu_{\gamma,\ell,m}^{\mathrm{can}}$.

\begin{definition}\label{def:local-dirichlet}
For an internal bond \((j,j+1)\), \(|j|\le \ell-1\), define the local exchange operator
\[
(L_{j,j+1}^{(\ell),\gamma}g)(\eta)
:=
c_{j,j+1}^{(\ell),\gamma}(\eta)\,
\bigl(g(\eta^{j,j+1})-g(\eta)\bigr),
\]
where
\[
c_{j,j+1}^{(\ell),\gamma}(\eta)
:=
\frac{1}{1+\exp\!\bigl(\beta\,\Delta_{j}^{(\ell)}H_{i,\gamma,\ell}^{\mathrm{blk}}(\eta)\bigr)},
\qquad
\Delta_{j}^{(\ell)}H_{i,\gamma,\ell}^{\mathrm{blk}}(\eta)
:=
H_{i,\gamma,\ell}^{\mathrm{blk}}(\eta^{j,j+1})-H_{i,\gamma,\ell}^{\mathrm{blk}}(\eta).
\]
The corresponding local Dirichlet form is given by
\[
D_{j,j+1}^{(\ell),\gamma}(g)
:=
-\bigl\langle g,L_{j,j+1}^{(\ell),\gamma}g\bigr\rangle_{\mu_{\gamma,\ell,m}^{\mathrm{can}}}
=
\frac12 \int_{\Sigma_{\ell,m}^{\mathrm{can}}}
c_{j,j+1}^{(\ell),\gamma}(\eta)\,
\bigl(g(\eta^{j,j+1})-g(\eta)\bigr)^2\, \mu_{\gamma,\ell,m}^{\mathrm{can}}(d\eta).
\]
Accordingly, the local Fisher information on the bond \((j,j+1)\) is
\begin{equation}\label{def:localF}
    I_{j,j+1}^{(\ell),\gamma}(g) := D_{j,j+1}^{(\ell),\gamma}(\sqrt g).
\end{equation}
\end{definition}

Now we show the uniform ellipticity of the local exchange rates.

\begin{lemma}\label{lem:local-rates-elliptic}
For every fixed \(\ell\ge1\) there exist constants
\(
0<c_-(\ell)\le c_+(\ell)<\infty
\)
such that for all \(\gamma\in(0,1)\), all \(|j|\le \ell-1\), and all \(\eta\in\Sigma_{\ell,m}^{\mathrm{can}}\),
\[
c_-(\ell)\le c_{j,j+1}^{(\ell),\gamma}(\eta)\le c_+(\ell).
\]
Similarly, there exist constants \(0<\underline c (\ell) \le \overline c (\ell) <\infty\), independent of \(N\) and \(\gamma\), such that
\[
\underline c (\ell) \le c_{i,i+1}(\sigma)\le \overline c(\ell)
\]
for all \(i\in\Lambda_N\) and all \(\sigma\in\Sigma_{N,M}\).
\end{lemma}

\begin{proof}
Since the exchange \(\eta\mapsto\eta^{j,j+1}\) only modifies two spins, the corresponding energy increment
\(
\Delta_j^{(\ell)}H_{i,\gamma,\ell}^{\mathrm{blk}}(\eta)
\)
is uniformly bounded for fixed \(\ell\), because \(\sum_{z}\kappa_\gamma(z)=1\) and \(|\eta_u|\le1\). The logistic form of the rate therefore yields uniform upper and lower bounds on \(c_{j,j+1}^{(\ell),\gamma}(\eta)\). The same argument applies to the full-torus rates \(c_i^\gamma(\sigma)\), since \(|\Delta_i H_\gamma(\sigma)|\) is uniformly bounded by the \(L^1\)-normalization of the kernel.
\end{proof}

\begin{lemma}[Entropy inequality]
\label{lem:entropy-inequality}
Let \(\nu\) be a probability measure, and let \(\mu\ll\nu\) with relative entropy
\[
H(\mu\mid \nu):=\int \log\!\left(\frac{d\mu}{d\nu}\right)d\mu<\infty.
\]
Then for every bounded measurable random variable \(F\) and every \(a>0\),
\[
\mathbb E_\mu[F]
\le
\frac{1}{a}H(\mu\mid \nu)
+
\frac{1}{a}\log \mathbb E_\nu[e^{aF}].
\]
In particular, for any random variable \(Y\),
\[
\mathbb E_\mu[|Y|]
\le
\frac{H(\mu\mid \nu)+\log 2}{a}
+
\frac{1}{a}
\max_{\pm}
\log \mathbb E_\nu[e^{\pm aY}].
\]
\end{lemma}

\begin{proof}
The first inequality is the standard entropy inequality; see
\cite[Lemma~6.1.4]{KipnisLandim1999} or \cite[Section~2]{GPV88}.
Applying it to \(F=|Y|\) and using
\(
e^{a|Y|}\le e^{aY}+e^{-aY}
\)
gives
\[
\log \mathbb E_\nu[e^{a|Y|}]
\le
\log 2+\max_{\pm}\log \mathbb E_\nu[e^{\pm aY}],
\]
which yields the second claim.
\end{proof}

Note that the Markov generator \(\mathscr L_\gamma^{N}\) is a self-adjoint non-positive operator on \(L^2(\mu_{N,\gamma,\beta})\).
We recall the standard Kipnis--Varadhan estimate for additive functionals of reversible Markov processes in the following lemma; see \cite{KipnisVaradhan1986} or \cite[Chapter~2, Section~1]{KipnisLandim1999}. 

\begin{lemma}[Kipnis--Varadhan \(H^{-1}\) estimate]
\label{lem:KV-invariant}
Let \(W:\Sigma_{N,M}\to\mathbb R\) satisfy
\(
\int W\,d\mu_{N,\gamma,\beta}=0.
\)
Define the negative Sobolev norm
\begin{equation}\label{VCH-1}
    \|W\|_{\mathcal{H}^{-1}_N}^2 :=
\sup_f
\Bigl\{
2\langle W,f\rangle_{\mu_{N,\gamma,\beta}}
-
\langle f,(-\mathscr L_\gamma^{N})f\rangle_{\mu_{N,\gamma,\beta}}
\Bigr\}.
\end{equation}
Then
\[
\mathbb E_{\mu_{N,\gamma,\beta}}
\left[
\left(\int_0^T W(\sigma(t))\,dt\right)^2
\right]
\le
2T\,\|W\|_{\mathcal{H}^{-1}_N}^2.
\]
\end{lemma}

Now we consider the fast process \((\sigma(\alpha_\gamma^{-1}t))_{t\ge0}\) and its associated accelerated generator \(
\mathscr A_N:=\alpha_\gamma^{-1}\mathscr L_\gamma^N
\) with total magnetization $M$. 
Note that the canonical Gibbs measure \(\mu_{N,\gamma,\beta}\) is also reversible for \(\mathscr A_N\). 
Let \(\nu_{N,t}\) be the law of the process \(\sigma^N(\alpha^{-1}_{\gamma}t)\), so that
\[
\nu_{N,t}(d\sigma)=f_{N,t}(\sigma)\,\mu_{N,\gamma,\beta}(d\sigma),
\quad
f_{N,t}\ge0,
\quad
\int_{\Sigma_{N,M}} f_{N,t}\,d\mu_{N,\gamma,\beta}=1.
\]
We define the relative entropy with respect to the canonical Gibbs equilibrium by
\[
H_N(f_{N,t}\mid \mu_{N,\gamma,\beta})
:=
\int_{\Sigma_{N,M}}  f_{N,t}(\sigma)\log f_{N,t}(\sigma)\,\mu_{N,\gamma,\beta}(d\sigma).
\]
Now we recall the standard results for entropy dissipation and entropy production rate of the Ising--Kac--Kawasaki dynamics, see e.g. \cite{Funaki2018HydrodynamicLimitExclusion}.

\begin{lemma}\label{lem:entropy-dissipation-kawasaki}
Let \((f_{N,t})_{t\in[0,T]}\) be the density of the Ising--Kac--Kawasaki dynamics $\sigma^N(\alpha^{-1}_{\gamma}t)$ with respect to \(\mu_{N,\gamma,\beta}\). 
Then the relative entropy is absolutely continuous in time and satisfies
\[
\frac{d}{dt}H_N(f_{N,t}\mid \mu_{N,\gamma,\beta})
\le
-4\alpha^{-1}_{\gamma}\sum_{i\in\Lambda_N} I_{i,i+1}^{N,\gamma}(f_{N,t})
=
-4\alpha^{-1}_{\gamma} I^{N,\gamma}(f_{N,t})
\]
for almost every \(t\in[0,T]\). Therefore
\[
4\alpha^{-1}_{\gamma}\int_0^T \sum_{i\in\Lambda_N} I_{i,i+1}^{N,\gamma}(f_{N,t})\,dt
\le
H_N(f_{N,0}\mid \mu_{N,\gamma,\beta}).
\]
\end{lemma}

Since the canonical Gibbs measure \(\mu_{N,\gamma,\beta}\) is also reversible for \(\mathscr A_N\), we have the following relation:
\begin{equation}\label{eq:varH-1}
    \sup_{f \in L^2(\mu_{N,\gamma,\beta})}
\Bigl\{
2\langle W,f\rangle_{\mu_{N,\gamma,\beta}}
-
\langle f,(-\mathscr A_N)f\rangle_{\mu_{N,\gamma,\beta}}
\Bigr\} = \alpha_{\gamma}\|W\|_{\mathcal{H}^{-1}_N}^2.
\end{equation}
Due to $0 <\alpha_{\gamma} \ll 1$, and the ergodicity of $\mu_{N,\gamma,\beta}$, we show the Entropy--Kipnis--Varadhan estimate for additive functionals with the accelerated process \((\sigma^N(\alpha_\gamma^{-1}t))_{t\ge0}\) with non-equilibrium initial laws.

\begin{lemma}\label{lem:entropy-KV}
Let \((\sigma^N(\alpha_\gamma^{-1}t))_{t\ge0}\) be the Markov process generated by \(\mathscr A_N\), started from the initial law \(f_{N,0}\mu_{N,\gamma,\beta}\), and assume
\(
H_N(f_{N,0}\mid \mu_{N,\gamma,\beta})\le C_H N.
\)
For a bounded observable \(W:\Sigma_{N,M}\to\mathbb R\) with
\(
\int W\,d\mu_{N,\gamma,\beta}=0,
\)
assume that there is a constant \(C_W <\infty\) such that
\begin{equation}\label{eq:entropy-KV-density-Hminus1}
    |\langle W,g^2\rangle_{\mu_{N,\gamma,\beta}}|
\le C_W \langle g,(-\mathscr L_\gamma^N)g\rangle_{\mu_{N,\gamma,\beta}}^{1/2}, \quad \|g\|_{L^2(\mu_{N,\gamma,\beta})}=1.
\end{equation}
Then
\begin{equation}\label{E:KLnonq}
    \mathbb E_{f_{N,0}}
\left|
\int_0^T W(\sigma^N(\alpha_\gamma^{-1}t))\,dt
\right| 
\le
C_W\sqrt{\alpha_{\gamma}NT}\,\|W\|_{\mathcal{H}^{-1}_N}. 
\end{equation}
\end{lemma}

\begin{proof}
Let \(\mathcal P_{f_{N,0}}\) be the law on paths of the accelerated process
started from \(f_{N,0}\mu_{N,\gamma,\beta}\), and let
\(\mathcal P_{\mu}\) be the stationary path law started from
\(\mu_{N,\gamma,\beta}\).  Since both path laws have the same transition
kernel, their relative entropy on \(D([0,T],\Sigma_{N,M})\) is given by
\[
H(\mathcal P_{f_{N,0}}\mid\mathcal P_{\mu})
=
H_N(f_{N,0}\mid\mu_{N,\gamma,\beta}).
\]
Applying Lemma~\ref{lem:entropy-inequality} on path space, with
\(\mu=\mathcal P_{f_{N,0}}\) and \(\nu=\mathcal P_{\mu}\), yields that for every
\(a>0\),
\begin{align}\label{E:entropyB}
    & \mathbb E_{f_{N,0}}
\left| \int_0^T W(\sigma^N(\alpha_\gamma^{-1}t))\,dt \right| \nonumber \\
\le & \frac{H_N(f_{N,0}\mid \mu_{N,\gamma,\beta})+\log 2}{a}
+
\frac{1}{a}
\max_{\pm}
\log \mathbb E_{\mu_{N,\gamma,\beta}}
\left[
\exp\!\left(
\pm a\int_0^T W(\sigma^N(\alpha_\gamma^{-1}t))\,dt
\right)
\right].
\end{align}
For the exponential moment under the invariant law, by Feynman--Kac formula,
\[
\mathbb E_{\mu_{N,\gamma,\beta}}
\exp\left\{
\int_0^T V(\sigma^N(\alpha_\gamma^{-1}t))\,dt
\right\}
=
\left\langle
1,e^{T(\mathscr A_N+V)}1
\right\rangle_{\mu_{N,\gamma,\beta}},
\]
where \(V=\pm aW\) acts by multiplication.
The Rayleigh--Ritz variational
formula for the largest eigenvalue $\lambda_{\max} $ of the self-adjoint operator \(\mathscr A_N+V\) satisfies
\[
\lambda_{\max} = \sup_{\|g\|_2=1}
\left\{ \langle V,g^2\rangle_\mu -
\alpha_\gamma^{-1} \langle g,(-\mathscr L_\gamma^N)g\rangle_\mu \right\}.
\]
Then by \eqref{eq:entropy-KV-density-Hminus1}, we have 
\begin{align}
&\log
\mathbb E_{\mu_{N,\gamma,\beta}} \left[ \exp\!\left( \pm a\int_0^T W(\sigma^N(\alpha_\gamma^{-1}t))\,dt \right) \right]
\\
\le & T \sup_{\substack{g\in L^2(\mu_{N,\gamma,\beta})\\
\|g\|_{L^2(\mu_{N,\gamma,\beta})}=1}}
\left\{ \langle V,g^2\rangle_{\mu_{N,\gamma,\beta}}
-
\alpha_\gamma^{-1} \left\langle g,(-\mathscr L_\gamma^N)g \right\rangle_{\mu_{N,\gamma,\beta}} \right\} \nonumber \\
\le & C\,a^2T\alpha_{\gamma} C_{W}^2.
\end{align}
Substituting this estimate into the entropy bound \eqref{E:entropyB} yields
\begin{align*}
    \mathbb E_{f_{N,0}}|Y|
\le &
\frac{H_N(f_{N,0}\mid \mu_{N,\gamma,\beta})+\log 2}{a}
+
C^2_{W}aT\alpha_{\gamma} \\
\leq & C
\sqrt{
\bigl(H_N(f_{N,0}\mid \mu_{N,\gamma,\beta})+1\bigr)\, T  \alpha_{\gamma}}.
\end{align*}
Under the entropy bound \(H_N(f_{N,0}\mid \mu_{N,\gamma,\beta})\le C_HN\), this becomes estimate \eqref{E:KLnonq}
\end{proof}

\subsection{Centered weak \(L^2\)-energy estimate}
\label{subsec:centered-weak-l2-energy}

This subsection collects the static and dynamic inputs needed for the
centered weak \(L^2\)-control of the Kac field.  The argument is self-contained
within the canonical sector: we first record the conserved mass normalization,
then derive the critical static mass scale from Assumption \(({\bf B})\), and
next isolate the static Kac--Ising inputs used for the infrared estimate, and
finally combine the equilibrium Fourier fluctuation bound with the initial
layer estimate.

On the canonical sector we denote
\[
m_N:=\frac1{|\Lambda_N|}\sum_{i\in\Lambda_N}\sigma_i
=\varepsilon_\gamma M(N),
\]
which is deterministic on \(\Sigma_{N,M}\), and is conserved by the
Kawasaki dynamics.
Since fluctuation field has conserved spatial average
\begin{equation}\label{eq:mass-sector-implies-mN-small}
    \langle X_\gamma(t),1\rangle_\gamma =\delta_\gamma^{-1}m_N \to M,
\end{equation}
we have $|m_N| \le C\delta_\gamma $.
Let
\[
\theta_\gamma
=
\frac{1}{\beta(1-m_N^2)}-1
=
\frac{1-\beta}{\beta}
+\frac{m_N^2}{\beta(1-m_N^2)}.
\]
Then Assumption \(({\bf B})\) implies
\begin{equation}\label{eq:critical-static-mass-rate}
c\delta_\gamma^2
\le
\theta_\gamma
\le
C\delta_\gamma^2 .
\end{equation}

Now we give the Fourier lower bound for the Kac kernel. 
\begin{lemma}\label{lem:kac-kernel-fourier-lower}
Under Assumption \(({\bf A})\) and the macroscopic Kac discretization
\(\varepsilon_\gamma/\gamma\to0\), there exists \(c>0\) such that, for all
sufficiently small \(\gamma\) and every \(k\in\mathcal B_\gamma\setminus\{0\}\),
\begin{equation}
\label{eq:kac-kernel-fourier-lower}
1-\widehat\kappa_\gamma(k)
\ge
c\,\min\left\{1,
\left(\frac{\varepsilon_\gamma |k|}{\gamma}\right)^2
\right\}.
\end{equation}
\end{lemma}

\begin{proof}
Set
\(
p_{\gamma,k}:=\frac{\varepsilon_\gamma k}{\gamma},
\widehat{\mathfrak K}(p)
:=
\int_{\mathbb R}\mathfrak K(u)e^{-2\pi i p u}\,du .
\)
Since \(\mathfrak K\in W^{1,1}(\mathbb R)\) is even and nonnegative and \(\int_{\mathbb R}\mathfrak K(u)\,du=1\), we have
\[
1-\widehat{\mathfrak K}(p)
=
\int_{\mathbb R}\mathfrak K(u)\bigl(1-\cos(2\pi p u)\bigr)\,du .
\]
As \(p\to0\), the finite second moment gives
\[
\frac{1-\widehat{\mathfrak K}(p)}{p^2}
\longrightarrow
2\pi^2\int_{\mathbb R}u^2\mathfrak K(u)\,du>0.
\]
Thus \(1-\widehat{\mathfrak K}(p)\ge c_0p^2\) for \(|p|\le p_0\).
For \(|p|\ge p_0\), the strict positivity of \(\mathfrak K\) and the
Riemann--Lebesgue lemma imply
\[
\inf_{|p|\ge p_0}\bigl(1-\widehat{\mathfrak K}(p)\bigr)>0.
\]
Indeed, \(|\widehat{\mathfrak K}(p)|<1\) for \(p\ne0\), while
\(\widehat{\mathfrak K}(p)\to0\) as \(|p|\to\infty\).

It remains to transfer this continuum bound to the lattice kernel.  By the
\(W^{1,1}\)-regularity of \(\mathfrak K\), the Riemann-sum comparison is
uniform for \(k\in\mathcal B_\gamma\):
\[
\left|
\widehat\kappa_\gamma(k)-\widehat{\mathfrak K}(p_{\gamma,k})
\right|
\le
C\gamma\bigl(1+|p_{\gamma,k}|\bigr)+o_\gamma(1)
\le
C\left(\gamma+\frac{\varepsilon_\gamma}{\gamma}\right)+o_\gamma(1),
\]
because \(|k|\lesssim\gamma^{-1}\) on \(\mathcal B_\gamma\).  The right-hand
side tends to zero.  On the region \(|p_{\gamma,k}|\ge p_0\), this proves a
uniform positive lower bound for \(1-\widehat\kappa_\gamma(k)\).

On the low-frequency region \(0<|p_{\gamma,k}|\le p_0\), choose a compact
interval \(J\subset\mathbb R\setminus\{0\}\) with
\(\int_J u^2\mathfrak K(u)\,du>0\), and take \(p_0\) small enough so that
\(|2\pi p u|\le\pi/2\) for \(u\in J\), \(|p|\le p_0\).  Then
\[
1-\cos(2\pi p_{\gamma,k}\gamma z)
\ge
c\,p_{\gamma,k}^2(\gamma z)^2
\qquad\text{for }\gamma z\in J.
\]
Consequently,
\[
1-\widehat\kappa_\gamma(k)
\ge
c\,p_{\gamma,k}^2
\sum_{\gamma z\in J}
\gamma\,\mathfrak K(\gamma z)(\gamma z)^2
\ge
c\,p_{\gamma,k}^2,
\]
for all sufficiently small \(\gamma\).  Combining the low- and high-frequency
regions proves \eqref{eq:kac-kernel-fourier-lower}.
\end{proof}

We shall use the following finite-volume form of the Lebowitz--Penrose coarse-graining principle.  
It follows by conditioning on a coarse-grained magnetization profile \(m\), applying Stirling's formula, and then optimizing the resulting non-local free-energy functional
\[
\mathcal F_{N,\gamma}(m) := \varepsilon_\gamma\sum_{i\in\Lambda_N} I(m_i) - \frac{\beta}{2}
\left\langle m,\mathcal K_\gamma m\right\rangle_\gamma =  \varepsilon_\gamma\sum_{i\in\Lambda_N} I(m_i) - \frac{\beta}{2}\varepsilon_\gamma
\sum_{i,j\in\Lambda_N}
\kappa_\gamma(i-j)m_i m_j.
\]
Here, \(\mathcal K_\gamma\) denotes the discrete Kac convolution operator
$(\mathcal K_\gamma m)_i := \sum_{j\in\Lambda_N}\kappa_\gamma(i-j)m_j$,
and $I(m)$ is the Lebowitz--Penrose functional (see \cite{LP1966}), which is given by
\[
I(m) := \frac{1+m}{2}\log\frac{1+m}{2} + \frac{1-m}{2}\log\frac{1-m}{2}, \qquad -1<m<1,
\]
with the usual continuous extension to \(m=\pm1\).
Consider the relevant perturbations $r$ in profiles $\mathcal P_0:=\left\{r:\varepsilon_\gamma\sum_i r_i=0\right\}$, 
and the free-energy cost is measured by the centered increment
\[
\mathcal Q_\gamma(r)
=
\mathcal F_{N,\gamma}(m_N+r)
-
\mathcal F_{N,\gamma}(m_N)
-
D\mathcal F_{N,\gamma}(m_N)[r].
\]
Now we state the finite-volume coarse-grained Laplace principle for Kac-Ising systems.  In the present paper we only use this upper-bound form, together with a quadratic lower bound on
\(\mathcal Q_\gamma\).

\begin{lemma}\label{lem:finite-volume-LP-variational}
For every deterministic field \(G\in\mathcal P_0\), we have the canonical-sector variational upper bound
\begin{equation}\label{eq:finite-volume-LP}
\log \mathbb E_{\mu_{N,\gamma,\beta}}
\exp\left\{
\varepsilon_\gamma^{-1}
\left\langle G,\sigma-m_N\right\rangle_\gamma
\right\}
\le \varepsilon_\gamma^{-1}
\sup_{r\in\mathcal P_0(m_N)}
\left\{
\left\langle G,r\right\rangle_\gamma - \mathcal Q_\gamma(r)
\right\},
\end{equation}
where the admissible canonical profile set
\[
\mathcal P_0(m_N)
:=
\left\{
r\in\mathcal P_0:\ m_N+r_i\in[-1,1]\ \text{for every }i
\right\}.
\]
Moreover, there exists a positive operator \(B_\gamma\) on \(\mathcal P_0\),
so that for every \(u\in\mathcal P_0\) and \(s\in\mathbb R\), and for all sufficiently small $\gamma$
\begin{equation}
\log
\mathbb E_{\mu_{N,\gamma,\beta}}
\exp\left\{
s\left\langle u,\sigma-m_N\right\rangle_\gamma
\right\}
\le
C\varepsilon_\gamma s^2
\left\langle u,B_\gamma^{-1}u\right\rangle_\gamma .
\label{eq:finite-volume-LP-laplace-bound}
\end{equation}
\end{lemma}

\begin{proof}
By finite-volume coarse-grained version of the Lebowitz--Penrose variational principle(see Presutti \cite{Presutti2009Scaling}, Benois--Bodineau--Presutti \cite{BenoisBodineauPresutti1998}), the exponential weight of a profile is bounded above by
\[
\exp\left\{-\varepsilon_\gamma^{-1}\mathcal F_{N,\gamma}(m)
\right\}.
\]
Now we apply this upper bound to the tilted canonical expectation
\[
\mathbb E_{\mu_{N,\gamma,\beta}}
\exp\left\{
\varepsilon_\gamma^{-1}
\left\langle G,\sigma-m_N\right\rangle_\gamma
\right\}, \quad G\in\mathcal P_0.
\]
Since the canonical constraint fixes the zero mode, admissible perturbations satisfies the form \(m=m_N+r, r\in\mathcal P_0\).  By \(G\in\mathcal P_0\), we have
\(
\left\langle G,m_N\right\rangle_\gamma=0.
\)
Then the untilted canonical partition function is normalized by the same coarse-grained free energy at \(m_N\).  Subtracting this reference value gives
\[
\log
\mathbb E_{\mu_{N,\gamma,\beta}}
\exp\left\{
\varepsilon_\gamma^{-1}
\left\langle G,\sigma-m_N\right\rangle_\gamma
\right\}
\le
\varepsilon_\gamma^{-1}
\sup_{r\in\mathcal P_0}
\left\{
\left\langle G,r\right\rangle_\gamma
-
\mathcal F_{N,\gamma}(m_N+r)
+
\mathcal F_{N,\gamma}(m_N)
\right\}.
\]
Since \(G\) is zero-average and \(r\in\mathcal P_0\), the linear term
\(D\mathcal F_{N,\gamma}(m_N)[r]\) may be inserted and subtracted.  This gives
upper bound \eqref{eq:finite-volume-LP}.

It remains to prove the quadratic consequence.  
Estimate~\ref{eq:critical-static-mass-rate} yields that
\[
c\delta_\gamma^2\le \frac{1-\beta}{\beta} \le C\delta_\gamma^2,
\qquad
\theta_\gamma\le C\delta_\gamma^2,
\]
and hence \(\theta_\gamma\le C \frac{1-\beta}{\beta}\).  In particular
\(\beta<1\) for small \(\gamma\).
Since
\[
I''(m)=\frac1{1-m^2}\ge1,\qquad -1<m<1,
\]
the increment satisfies
\[
\varepsilon_\gamma\sum_i
\Bigl[
I(m_N+r_i)-I(m_N)-I'(m_N)r_i
\Bigr]
\ge
\frac12\,\varepsilon_\gamma\sum_i r_i^2 .
\]
Notice that interaction part is exactly quadratic. The Parseval's identity implies
\[
\mathcal Q_\gamma(r)
\ge
\frac12
\sum_{q\in\mathcal B_\gamma\setminus\{0\}}
\bigl(1-\beta\widehat\kappa_\gamma(q)\bigr)
|\widehat r_\gamma(q)|^2 .
\]
Since \(\theta_\gamma\le C  \frac{1-\beta}{\beta}  \) and
\(1-\widehat\kappa_\gamma(q)\ge0\), we have
\[
1-\beta\widehat\kappa_\gamma(q)
=
\beta\bigl( \frac{1-\beta}{\beta} +1-\widehat\kappa_\gamma(q)\bigr)
\ge
c\bigl(\theta_\gamma+1-\widehat\kappa_\gamma(q)\bigr),
\]
Then there exists a positive operator \(B_\gamma\) on \(\mathcal P_0\) with the Fourier multiplier
\[
B_\gamma e_q
:=
\bigl(\theta_\gamma+1-\widehat\kappa_\gamma(q)\bigr)e_q,
\qquad q\ne0,
\]
so that
\begin{equation}\label{eq:finite-volume-LP-coercive-input}
\mathcal Q_\gamma(r)
\ge c_0\left\langle r,B_\gamma r\right\rangle_\gamma,
\qquad r\in\mathcal P_0(m_N).
\end{equation}
Applying \eqref{eq:finite-volume-LP} with \(G=\varepsilon_\gamma s u\),
and using the coercivity estimate \eqref{eq:finite-volume-LP-coercive-input}, we have 
\[
\begin{aligned}
\log
\mathbb E_{\mu_{N,\gamma,\beta}}
\exp\left\{
s\left\langle u,\sigma-m_N\right\rangle_\gamma
\right\}
&\le
\varepsilon_\gamma^{-1}
\sup_{r\in\mathcal P_0}
\left\{
\varepsilon_\gamma s\left\langle u,r\right\rangle_\gamma - c_0\left\langle r,B_\gamma r\right\rangle_\gamma
\right\}.
\end{aligned}
\]
The supremum of the quadratic functional is bounded by $C\varepsilon_\gamma^2s^2 \left\langle u, B_\gamma^{-1}u\right\rangle_\gamma$.
After multiplying by \(\varepsilon_\gamma^{-1}\), we obtain \eqref{eq:finite-volume-LP-laplace-bound}.
\end{proof}

In order to give the centered weak $L^2$-energy estimate, we show the following equilibrium centered Kac fluctuation bound for $h_\gamma(\sigma;i)-m_N$.

\begin{lemma}[Equilibrium centered Kac fluctuation bound]
\label{lem:equilibrium-centered-kac-fluctuation}
Set $h_\gamma(\sigma;i) := \sum_{z\in\Lambda_N}\kappa_\gamma(i-z)\sigma_z$.
Then
\begin{equation}\label{eq:equilibrium-centered-h-fluctuation-rate}
    \int_{\Sigma_{N,M}}
\varepsilon_\gamma\sum_{i\in\Lambda_N}\bigl(h_\gamma(\sigma;i)-m_N\bigr)^2
\mu_{N,\gamma,\beta}(d\sigma)
\le C\left(
\frac{\gamma}{\delta_\gamma}+
\frac{\varepsilon_\gamma}{\delta_\gamma^2}+
\gamma\right) \to 0 \quad \text{as} \quad \gamma \to 0.
\end{equation}
\end{lemma}

\begin{proof}
Since \(\sum_z\kappa_\gamma(z)=1\), convolution by \(\kappa_\gamma\) leaves
the conserved zero mode unchanged.  
Thus the zero Fourier mode of
\(h_\gamma\) is exactly \(m_N\), while every nonzero centered mode is
multiplied by \(\widehat\kappa_\gamma(k)\).  Hence Parseval's identity gives
\begin{equation}\label{eq:Parsevalidentity}
    \varepsilon_\gamma
\sum_{i\in\Lambda_N}
\bigl(h_\gamma(\sigma;i)-m_N\bigr)^2 = \sum_{k\in\mathcal B_\gamma\setminus\{0\}}
\left|\widehat\kappa_\gamma(k)\right|^2\left|\widehat\sigma_\gamma(k)\right|^2,
\end{equation}
where
\[
\widehat\sigma_\gamma(k)
:=
\varepsilon_\gamma
\sum_{i\in\Lambda_N}
\bigl(\sigma_i-m_N\bigr)e^{-2\pi i k\varepsilon_\gamma i},
\qquad
k\in\mathcal B_\gamma\setminus\{0\}.
\]
Since
\(\kappa_\gamma\) is even, nonnegative, and normalized, we have
\[
\widehat\kappa_\gamma(k) :=
\sum_{z\in\Lambda_N}\kappa_\gamma(z)
e^{-2\pi i k\varepsilon_\gamma z}
= \sum_{z\in\Lambda_N}
\kappa_\gamma(z)\cos(2\pi k\varepsilon_\gamma z),
\]
and therefore
\[
1-\widehat\kappa_\gamma(k)
=
\sum_{z\in\Lambda_N}
\kappa_\gamma(z)
\bigl(1-\cos(2\pi k\varepsilon_\gamma z)\bigr)
\ge0.
\]
For \(k\ne0\), this term is strictly positive whenever the support of
\(\kappa_\gamma\) is not contained in one period of the mode.  
Then estimate~\ref{eq:critical-static-mass-rate} gives
\(\theta_\gamma\ge c\delta_\gamma^2>0\), and 
\(
\theta_\gamma+1-\widehat\kappa_\gamma(k)>0.
\)
We now estimate the covariance from the canonical Gibbs measure $\mu_{N,\gamma,\beta}(\sigma)$ on the canonical sector $\Sigma_{N,M}$.
For a real deterministic zero-average lattice field \(u\), we set
\[
\Lambda_\gamma(s;u)
:=
\log
\mathbb E_{\mu_{N,\gamma,\beta}}
\exp\left\{
s\langle u,\sigma-m_N\rangle_\gamma
\right\}.
\]
Applying the finite-volume coarse-grained Laplace principle of
Lemma~\ref{lem:finite-volume-LP-variational} with \(G=\varepsilon_\gamma s u\)
gives the local normalized bound
\[
\Lambda_\gamma(s;u)
\le
\varepsilon_\gamma^{-1}
\sup_{r\in\mathcal P_0(m_N)}
\left\{
\varepsilon_\gamma s\langle u,r\rangle_\gamma
-
\mathcal Q_\gamma(r)
\right\},
\]
where
\[
\mathcal Q_\gamma(r)
:=
\mathcal F_{N,\gamma}(m_N+r)
-
\mathcal F_{N,\gamma}(m_N)
-
D\mathcal F_{N,\gamma}(m_N)[r].
\]
The normalization by \(\mu_{N,\gamma,\beta}\), whose normalizing constant is
\(Z_{N,\gamma,\beta}^{M}\), removes the additive finite-volume Stirling
constant before the second derivative at \(s=0\) is taken.
By Lemma~\ref{lem:finite-volume-LP-variational}, we have
\[
\Lambda_\gamma(s;u)
\le
C\varepsilon_\gamma s^2
\sum_{q\ne0}
\frac{|\widehat u_\gamma(q)|^2}
{\theta_\gamma+1-\widehat\kappa_\gamma(q)} .
\]
Differentiating at \(s=0\) gives
\[
\operatorname{Var}_{\mu_{N,\gamma,\beta}}
\bigl(\langle u,\sigma-m_N\rangle_\gamma\bigr)
\le
C\varepsilon_\gamma
\sum_{q\ne0}
\frac{|\widehat u_\gamma(q)|^2}
{\theta_\gamma+1-\widehat\kappa_\gamma(q)} .
\]
Apply this estimate to the real and imaginary parts of
\(u_i=e^{-2\pi i k\varepsilon_\gamma i}\).  With the Fourier normalization,
these test fields have only the modes \(\pm k\), so that
\begin{equation}\label{eq:static-infrared-bound}
    \mathbb E_{\mu_{N,\gamma,\beta}}
\left|\widehat\sigma_\gamma(k)\right|^2
\le C\,\frac{\varepsilon_\gamma}
{\theta_\gamma+1-\widehat\kappa_\gamma(k)},
\qquad k\in\mathcal B_\gamma\setminus\{0\}.
\end{equation}
Plugging this estimate to \eqref{eq:Parsevalidentity}, we have
\[
\int \varepsilon_\gamma
\sum_{i\in\Lambda_N}
\bigl(h_\gamma(\sigma;i)-m_N\bigr)^2\,d\mu_{N,\gamma,\beta}
\le
C\varepsilon_\gamma
\sum_{k\ne0}
\frac{|\widehat\kappa_\gamma(k)|^2}
{\theta_\gamma+1-\widehat\kappa_\gamma(k)}.
\]
Then by Lemma~\ref{lem:kac-kernel-fourier-lower},
\(
1-\widehat\kappa_\gamma(k)
\ge
c\,\min\left\{1,
\left(\varepsilon_\gamma |k|/\gamma\right)^2
\right\}.
\)
Using the \(W^{1,1}\)-regularity of \(\mathfrak K\), we have the estimate for low-frequency sum 
\[
\varepsilon_\gamma
\sum_{0<|k|\le c\gamma/\varepsilon_\gamma}
\frac{|\widehat\kappa_\gamma(k)|^2}
{\theta_\gamma+1-\widehat\kappa_\gamma(k)}
\le
C
\int_{\mathbb R}
\frac{\gamma\,|\widehat{\mathfrak K}(p)|^2}
{\theta_\gamma+c\,p^2}\,dp +C\frac{\varepsilon_\gamma}{\theta_\gamma} \le C \frac{\gamma}{\sqrt{\theta_\gamma}} + C\frac{\varepsilon_\gamma}{\theta_\gamma}. 
\]
On the complementary modes \(1-\widehat\kappa_\gamma(k)\ge c\), and Parseval's identity gives
\begin{equation}\label{eq:highkerfre}
    \varepsilon_\gamma
\sum_{|k|>c\gamma/\varepsilon_\gamma}
|\widehat\kappa_\gamma(k)|^2
\le \sum_z\kappa_\gamma(z)^2
\le C\gamma .
\end{equation}
Thus
\[
\int_{\Sigma_{N,M}} \varepsilon_\gamma \sum_{i\in\Lambda_N}
\bigl(h_\gamma(\sigma;i)-m_N\bigr)^2
\mu_{N,\gamma,\beta}(d\sigma) \le C\left( \frac{\gamma}{\sqrt{\theta_\gamma}} +\frac{\varepsilon_\gamma}{\theta_\gamma} +\gamma \right).
\]
Since $c\delta_\gamma^2
\le
\theta_\gamma \le
C\delta_\gamma^2$, we obtain \eqref{eq:equilibrium-centered-h-fluctuation-rate}.
\end{proof}

Now we show the centered weak \(L^2\)-energy estimate with nonequilibrium initial value.

\begin{lemma}\label{lem:weak-L2-h-from-microscopic}
Under Assumption \(({\bf A})-({\bf C})\), we have
\begin{equation}\label{eq:weak-L2-h-energy-critical-rate}
\mathbb E_{f_{N,0}}\int_0^T\varepsilon_\gamma\sum_{i\in\Lambda_N} \bigl(h_{\gamma,i}(t)-m_N\bigr)^2\,dt
\le C_T\delta_\gamma^2\log(\varepsilon_\gamma^{-1}).
\end{equation}
Moreover,
\begin{equation}
\label{eq:clean-critical-centered-energy}
\mathbb E_{f_{N,0}}
\left|\int_0^T\varepsilon_\gamma\sum_j h^2_{\gamma,j}(t)dt\right| \le  C\left( \delta_\gamma^2\log(\varepsilon_\gamma^{-1}) \right).
\end{equation}
\end{lemma}

\begin{proof}
Let
\[
\tau_\gamma
:=
\frac{\alpha_\gamma}{2\varepsilon_{\gamma}^2}
\left[
\log\!\left(H_N(f_{N,0}\mid\mu_{N,\gamma,\beta})+2\right) + 4\log(\delta_\gamma^{-1})
\right].
\]
Assumption \(({\bf C})\) gives
\[\log(H_N(f_{N,0}\mid\mu_{N,\gamma,\beta})+2)
\le C\log(\varepsilon_\gamma^{-1})\]. Together with Assumption \(({\bf B})\), this implies
\(
\tau_\gamma \le C\frac{\alpha_\gamma}{\varepsilon_\gamma^{2}}\log(\varepsilon_\gamma^{-1}).
\)
Then by \(|h_\gamma(\sigma;i)-m_N|\le2\), we have
\begin{equation}\label{eq:squr}
    \mathbb E_{f_{N,0}}\int_0^{T\wedge\tau_\gamma} \varepsilon_\gamma \sum_{i\in\Lambda_N} \bigl(h_\gamma(\sigma^N(t/\alpha_\gamma);i)-m_N\bigr)^2 dt \le 4\tau_\gamma \le C \frac{\alpha_\gamma}{\varepsilon_\gamma^{2}}\log(\varepsilon_\gamma^{-1}).
\end{equation}
Then Lemma \ref{lem:local-rates-elliptic} and Lemma \ref{lem:entropy-dissipation-kawasaki}, the density of accelerated process $\sigma^N(t/\alpha_\gamma)$ satisfies 
\begin{equation}\label{eq:microscopic-entropy-mixing}
    H_N(f_{N,t}\mid\mu_{N,\gamma,\beta})
\le e^{- C\varepsilon_{\gamma}^2t/\alpha_\gamma} H_N(f_{N,0}\mid\mu_{N,\gamma,\beta}),
\qquad 0\le t\le T.
\end{equation}
For \(t\ge\tau_\gamma\), by \eqref{eq:microscopic-entropy-mixing} and the
definition of \(\tau_\gamma\), we have
\(
H_N(f_{N,t}\mid\mu_{N,\gamma,\beta}) \le \delta_\gamma^4.
\)
Using the Pinsker's inequality, we get
\begin{align}\label{eq:pins}
    & \left| \mathbb E_{f_{N,0}} \left(\varepsilon_\gamma \sum_{i\in\Lambda_N}
\bigl(h_\gamma(\sigma^N(t/\alpha_\gamma);i)-m_N\bigr)^2 \right) - \int \varepsilon_\gamma \sum_{i\in\Lambda_N}
\bigl(h_\gamma(\sigma;i)-m_N\bigr)^2  d\mu_{N,\gamma,\beta}\right| \nonumber \\
\le & 4\|f_{N,t}\mu_{N,\gamma,\beta}-\mu_{N,\gamma,\beta}\|_{\mathrm{TV}} \nonumber \\
\le & 4\sqrt2\, \delta_\gamma^2 .
\end{align}
Combining estimate \eqref{eq:equilibrium-centered-h-fluctuation-rate}, \eqref{eq:squr}, and \eqref{eq:pins}, we have
\[
\mathbb E_{f_{N,0}}\int_0^T\varepsilon_\gamma\sum_{i\in\Lambda_N} \bigl(h_{\gamma,i}(t)-m_N\bigr)^2\,dt
\le
C_T\left[
\frac{\alpha_\gamma}{\varepsilon_\gamma^{2}}\log(\varepsilon_\gamma^{-1}) +
\frac{\gamma}{\delta_\gamma} +
\frac{\varepsilon_\gamma}{\delta_\gamma^2} +
\gamma + \delta_\gamma^2
\right].
\]
By Assumption \(({\bf B})\), we obtain \eqref{eq:weak-L2-h-energy-critical-rate}.
Combining with \eqref{eq:mass-sector-implies-mN-small}, we further have estimate \eqref{eq:clean-critical-centered-energy}.
\end{proof}

We also need the following mean-square Kac-gradient estimate.

\begin{lemma}\label{lem:mean-square-kac-gradient}
Under Assumptions \(({\bf A})\)--\(({\bf C})\), for every \(T>0\),
\begin{equation}\label{eq:kac-gradient-defect-scale}
\frac{\varepsilon_\gamma}{\alpha_\gamma\delta_\gamma}
\mathbb E_{f_{N,0}}
\int_0^T
\varepsilon_\gamma\sum_{i\in\Lambda_N}
\bigl(h_{\gamma,i+1}(t)-h_{\gamma,i}(t)\bigr)^2\,dt
\le
C_T\left[
\gamma^3
+\gamma^2\delta_\gamma^2
+\gamma^2\frac{\alpha_\gamma}{\varepsilon_\gamma^2}
\log(\varepsilon_\gamma^{-1})
\right].
\end{equation}
\end{lemma}

\begin{proof}
Set \(\eta_\gamma(z):=\kappa_\gamma(z-1)-\kappa_\gamma(z).\)
Since \(\sum_{z \in \Lambda_N}\eta_\gamma(z)=0\), we have
\[
h_\gamma(\sigma;i+1)-h_\gamma(\sigma;i)
=\sum_z\eta_\gamma(i-z)(\sigma_z-m_N).
\]
The estimates in Assumption \(({\bf A})\) give
\[
\sum_z|\eta_\gamma(z)|\le C\gamma,\qquad
\sum_z|\eta_\gamma(z)|^2\le C\gamma^3,\qquad
\|\eta_\gamma\|_{\ell^\infty}\le C\gamma^2.
\]
We first prove the equilibrium estimate.  By Parseval's identity and the bound \eqref{eq:static-infrared-bound},
\[
\int \varepsilon_\gamma\sum_{i \in \Lambda_N} |h_\gamma(\sigma;i+1)-h_\gamma(\sigma;i)|^2\,
d\mu_{N,\gamma,\beta}
\le
C\varepsilon_\gamma
\sum_{k\ne0}
\frac{|\widehat\eta_\gamma(k)|^2}
{\theta_\gamma+1-\widehat\kappa_\gamma(k)} .
\]
On the low-frequency region
\(|k|\le c\gamma/\varepsilon_\gamma\), writing
\(p=\varepsilon_\gamma |k|/\gamma\), we have
\[
|\widehat\eta_\gamma(k)|
\le C\varepsilon_\gamma |k|\,|\widehat\kappa_\gamma(k)|
\le C\gamma p,
\qquad
1-\widehat\kappa_\gamma(k)\ge cp^2.
\]
Thus the low-frequency part is bounded by
\[
C\varepsilon_\gamma
\sum_{0<|k|\le c\gamma/\varepsilon_\gamma}
\frac{\gamma^2p^2}{\theta_\gamma+cp^2}
\le C\gamma^3 .
\]
On the complementary region \(1-\widehat\kappa_\gamma(k)\ge c\), Parseval's
identity for the kernel \(\eta_\gamma\) gives
\[
\varepsilon_\gamma
\sum_{|k|>c\gamma/\varepsilon_\gamma}
|\widehat\eta_\gamma(k)|^2
\le
C\sum_z|\eta_\gamma(z)|^2
\le C\gamma^3 .
\]
Hence
\begin{equation}\label{eq:equilibrium-kac-gradient-square}
\int \varepsilon_\gamma\sum_{i \in \Lambda_N} |h_\gamma(\sigma;i+1)-h_\gamma(\sigma;i)|^2\,
d\mu_{N,\gamma,\beta}
\le C\gamma^3 .
\end{equation}
We pass from equilibrium to the nonequilibrium initial law as in Lemma~\ref{lem:weak-L2-h-from-microscopic}.  Let \(\tau_\gamma\) be the time
defined in that proof.  The deterministic bound
\(\sup_i|g_{\gamma,i}|\le C\gamma\) implies
\[
\mathbb E_{f_{N,0}}
\int_0^{T\wedge\tau_\gamma}
\varepsilon_\gamma\sum_{i \in \Lambda_N} |h_\gamma(\sigma;i+1)-h_\gamma(\sigma;i)|^2\,dt
\le
C\gamma^2\tau_\gamma
\le
C\gamma^2\frac{\alpha_\gamma}{\varepsilon_\gamma^2}
\log(\varepsilon_\gamma^{-1}).
\]
For \(t\ge\tau_\gamma\), the entropy decay
\eqref{eq:microscopic-entropy-mixing} and Pinsker's inequality give
\[
\left|
\mathbb E_{f_{N,0}}
\varepsilon_\gamma\sum_i |h_\gamma(\sigma;i+1)-h_\gamma(\sigma;i)|^2
-
\int\varepsilon_\gamma\sum_i |h_\gamma(\sigma;i+1)-h_\gamma(\sigma;i)|^2\,
d\mu_{N,\gamma,\beta}
\right|
\le C\gamma^2\delta_\gamma^2 .
\]
Combining this estimate with
\eqref{eq:equilibrium-kac-gradient-square} proves the stated bound.  The
convergence \eqref{eq:kac-gradient-defect-scale} follows immediately from
Assumption \(({\bf B})\).
\end{proof}
\section{Expansion for the discrete drift}\label{sec:current}

In this section we give an expansion for the discrete drift for the rescaled field $X_\gamma$ in the \(H^{-1}\)-testing form.
For a smooth test function \(\phi\in C^\infty(\mathbb T)\) with zero mean, we set
\[
\psi_\gamma:=\Delta_{\varepsilon_\gamma}^{-1}\phi,
\qquad
\langle\psi_\gamma,1\rangle_\gamma=0,
\qquad
\Delta_{\varepsilon_\gamma}\psi_\gamma=\phi
\quad\text{on }\Lambda_{\varepsilon_\gamma}.
\]
so that
\(
\langle X_\gamma,\phi\rangle_{-1,\gamma}
:=
\langle X_\gamma,\Delta_{\varepsilon_\gamma}^{-1}\phi\rangle_\gamma
=
\langle X_\gamma,\psi_\gamma\rangle_\gamma .
\)

With the Ising–Kac Hamiltonian $H_\gamma(\sigma)= -\frac12\sum_{j \in \Lambda_N } h(\sigma, j)\sigma_j$,
the exchange energy difference admits the exact representation
\[
\Delta_i H_\gamma(\sigma)
=
d_i\bigl(h_\gamma(\sigma;i)-h_\gamma(\sigma;i+1)\bigr) + \kappa_\gamma(1)d^2_i,
\]
where $d_i:=\sigma_i-\sigma_{i+1}\in\{0,\pm2\}$.
Thus the nonlocal dependence of \(\Delta_i H_\gamma\) enters only through the discrete gradient of the coarse-grained Kac field.

Note that the increment of \(h_\gamma\) under a nearest-neighbour exchange:
\begin{equation}\label{Ex:h}
    h_\gamma(\sigma^{i,i+1};k)-h_\gamma(\sigma;k)
=
d_i\Bigl[\kappa_\gamma(k-i-1)-\kappa_\gamma(k-i)\Bigr].
\end{equation}
Therefore
\[
(\mathscr L_\gamma^{N} h_\gamma)(\sigma;k)
=
\sum_{i\in\Lambda_N}
j_i(\sigma)\Bigl[\kappa_\gamma(k-i-1)-\kappa_\gamma(k-i)\Bigr],
\qquad j_i(\sigma)= \frac{d_i}{1 + \exp{(\beta\Delta_i H_\gamma (\sigma))}}.
\]
We emphasize that the microscopic current $j_i$ is indexed by the bond $(i,i+1)$, not by sites $i$.
Thus the left bond for each $j_i(\sigma)$ is not omitted; it is simply represented by the shifted index \(i-1\).

Define the coarse-grained current
\[
\mathcal J_\gamma(t,x)
:=
\frac{\varepsilon_\gamma}{\alpha_\gamma\delta_\gamma}
\sum_{i\in\Lambda_N} \kappa_\gamma\!\left(\frac{x}{\varepsilon_\gamma}-i\right)
j_i\!\left(\sigma\!\left(\frac{t}{\alpha_\gamma}\right)\right), \quad x \in \Lambda_{\varepsilon_{\gamma}}.
\]
Then Dynkin's formula yields the martingale decomposition
\[
X_\gamma(t,x) = X_\gamma(0,x) - \int_0^t \nabla^*_{\varepsilon_\gamma} \mathcal J_\gamma(s,x)\,ds + M_\gamma(t,x),
\]
where $\nabla_{\varepsilon_\gamma}^*$ is the discrete divergence, and $M_\gamma$ is a pure-jump martingale.
For each test function $\phi \in C^{\infty}(\mathbb T)$, the weak drift of \(\langle X_\gamma, \phi\rangle_{-1,\gamma}=\langle X_\gamma,\psi_\gamma\rangle_\gamma\) can be written in conservative form as
\[
\Bigl(\alpha_\gamma^{-1}\mathscr L_\gamma^{N}\langle X_\gamma,\psi_\gamma\rangle_\gamma\Bigr)(\sigma) =  \varepsilon_\gamma\sum_{i \in \Lambda_N }\mathcal J_\gamma(\sigma;\varepsilon_{\gamma} i) \nabla_{\varepsilon_\gamma}\psi_\gamma(\varepsilon_{\gamma} i)
=
\frac{\varepsilon_\gamma^2}{\alpha_\gamma\delta_\gamma}
\sum_{i\in\Lambda_N}j_i(\sigma)\,
\nabla_{\varepsilon_\gamma}\bigl(\mathcal K_\gamma^\varepsilon\psi_\gamma\bigr)(\varepsilon_\gamma i) 
\]
in the discrete Riemann-sum sense, where the discrete convolution operator
\[
(\mathcal K_\gamma^\varepsilon f)(\varepsilon_\gamma i):= \sum_{j \in\Lambda_N}\kappa_\gamma(i-j)f(\varepsilon_\gamma j).
\]
Note that the operators \(\Delta_{\varepsilon_\gamma}\) and \(\mathcal K_\gamma^\varepsilon\)  commute on zero-mean lattice functions. Thus
\( \Delta_{\varepsilon_\gamma}\mathcal K_\gamma^\varepsilon\psi_\gamma = \mathcal K_\gamma^\varepsilon\phi\).
Since \( X_\gamma(t,\varepsilon_\gamma i)
=
\delta_\gamma^{-1}
h_\gamma\!\left(\sigma^N(t/\alpha_\gamma);i\right), \varepsilon_\gamma i\in \mathbb T\), we have
\[
h_\gamma\!\left(\sigma^N\left(\frac{t}{\alpha_\gamma}\right);i\right)
-
h_\gamma\!\left(\sigma^N\left(\frac{t}{\alpha_\gamma}\right);i+1\right)
=
-\delta_\gamma\varepsilon_\gamma
\nabla_{\varepsilon_\gamma}X_\gamma(t, \varepsilon_{\gamma} i),
\]
and
\[
\Delta_i H_\gamma = -\,d_i\,\delta_\gamma\varepsilon_\gamma
\nabla_{\varepsilon_\gamma}X_\gamma(t,\varepsilon_{\gamma} i)
+d_i^2\,\kappa_\gamma(1).
\]
Let \(F(z):=\frac{1}{1+e^z}\), so that
\(
j_i=d_i\,F\!\bigl(\beta\Delta_i H_\gamma\bigr).
\)
Set
\(a_i:=\beta d_i^2\kappa_\gamma(1), b_i:=-\beta d_i\,\delta_\gamma\varepsilon_\gamma
\nabla_{\varepsilon_\gamma}X_\gamma(t,\varepsilon_{\gamma} i)\), then
\(
\beta\Delta_i H_\gamma=a_i+b_i.
\)
Since $ \kappa_{\gamma}(0) = 0$ and the Lipschitz regularity of $\mathfrak R$, one has \(a_i=O(\gamma^2)\).
Applying Taylor expansion to \(F(a_i+b_i)\) around \(a_i\), we have
\[
F(a_i+b_i) = F(a_i)+F'(a_i)b_i+\frac12F''(a_i)b_i^2+ \widetilde R_i,
\]
with
\(
\widetilde R_i = \frac16 F^{(3)}(a_i+\theta_i b_i)b_i^3, \ \theta_i\in(0,1).
\)
Therefore
\[
j_i = d_iF(a_i)
-\beta F'(a_i)d_i^2\,\delta_\gamma\varepsilon_\gamma
\nabla_{\varepsilon_\gamma}X_\gamma
+\frac{\beta^2}{2}F''(a_i)d_i^3
(\delta_\gamma\varepsilon_\gamma\nabla_{\varepsilon_\gamma}X_\gamma)^2
+ d_i \widetilde R_i.
\]
From the linear and nonlinear expansions $j_i$, we have decomposition
\(
j_i =j_i^{\mathrm{lin}}+j_i^{\mathrm{nl}}+j_i^{\mathrm{rem}}.
\)
Here,

\medskip
\noindent\textbf{(i) Linear part.}
Since \(d_i^2\in\{0,4\}\), we have \(a_i\in\{0,4\beta \kappa_\gamma(1)\}\), hence \(F(a_i)\) takes only two values.
Using the Taylor expansion \(F(u)=\frac12-\frac{u}{4}+O(u^2)\) at \(u=0\), one obtains
\[
j_i^{\mathrm{lin}}
=
\Bigl(\tfrac12-\beta\kappa_\gamma(1)\Bigr)d_i
+\widetilde r_{i,\gamma}^{\,\mathrm{lin}},
\qquad
|\widetilde r_{i,\gamma}^{\,\mathrm{lin}}|\lesssim \kappa_\gamma(1)^2\,|d_i|,
\]
so the leading linear current coefficient is \(\tfrac12-\beta\kappa_\gamma(1)\).

\medskip
\noindent\textbf{(ii) Nonlinear part.}
Since
\(
F'(u)
=
-\frac14+\frac{u^2}{16}+O(u^4)
\)
and \(d_i b_i=-\beta d_i^2\,\delta_\gamma\varepsilon_\gamma\nabla_{\varepsilon_\gamma}X_\gamma\),
we obtain the more precise nonlinear contribution
\[
-\beta F'(a_i)d_i^2\delta_\gamma\varepsilon_\gamma
\nabla_{\varepsilon_\gamma}X_\gamma
=
\frac{\beta}{4}d_i^2\delta_\gamma\varepsilon_\gamma
\nabla_{\varepsilon_\gamma}X_\gamma
-
\frac{\beta^3}{16}d_i^6\kappa_\gamma(1)^2
\delta_\gamma\varepsilon_\gamma
\nabla_{\varepsilon_\gamma}X_\gamma
+
\widehat r_{i,\gamma}^{\,\mathrm{nl}},
\]
where
\(
|\widehat r_{i,\gamma}^{\,\mathrm{nl}}|
\lesssim
\gamma^4\,
\delta_\gamma\varepsilon_\gamma
\bigl|\nabla_{\varepsilon_\gamma}X_\gamma\bigr|.
\)
Here the term of order \(\kappa_\gamma(1)^2 = O(\gamma^2)\) is one higher Kac order than the
leading nonlinear current.  Thus
\[
j_i^{\mathrm{nl}}(\sigma) = \frac{\beta}{4}\,d_i^2(\sigma)^2\delta_\gamma\varepsilon_\gamma\,\nabla_{\varepsilon_\gamma}X_\gamma(\varepsilon_\gamma i)
+\widetilde r_{i,\gamma}^{\,\mathrm{nl}} 
=
-\frac{\beta}{4}
d_i(\sigma)^2
\sum_{j\in\Lambda_N}\kappa_\gamma(i-j)d_j(\sigma)
+
\widetilde r_{i,\gamma}^{\,\mathrm{nl}}(\sigma),
\]
where $|\widetilde r_{i,\gamma}^{\,\mathrm{nl}}| \le \gamma^2\delta_\gamma\varepsilon_\gamma
\bigl|\nabla_{\varepsilon_\gamma}X_\gamma\bigr|$.

\medskip
\noindent
\textbf{(iii) Remainder.}
By construction,
\(
j_i^{\mathrm{rem}}= \frac12\big(d_iF''(a_i)(b_i)^2\big)(t) + d_i \widetilde R_i(t) \), where $F''(a_i) \le C \gamma^2$.

We now lift the microscopic decomposition
\(
j_i =j_i^{\mathrm{lin}}+j_i^{\mathrm{nl}}+j_i^{\mathrm{rem}}
\)
to the corresponding decomposition of the Kac coarse-grained current \(
\mathcal J_\gamma
=
\mathcal J_\gamma^{\mathrm{lin}}
+
\mathcal J_\gamma^{\mathrm{nl}}
+
\mathcal J_\gamma^{\mathrm{rem}}.
\)
where
\begin{align}
\label{eq:cg-lin-leading}
\mathcal J_\gamma^{\mathrm{lin}}(t,x)
& :=
\frac{\varepsilon_\gamma}{\alpha_\gamma\delta_\gamma} \sum_{i\in\Lambda_N} \kappa_\gamma\!\left(\frac{x}{\varepsilon_\gamma}-i\right)j_i^{\mathrm{lin}}\left(\sigma\!\left(\frac{t}{\alpha_\gamma}\right)\right)
\\
\label{eq:cg-nl-leading}
\mathcal J_\gamma^{\mathrm{nl}}(t,x)
 & := \frac{\varepsilon_\gamma}{\alpha_\gamma\delta_\gamma} \sum_{i\in\Lambda_N} \kappa_\gamma\!\left(\frac{x}{\varepsilon_\gamma}-i\right)
\,j_i^{\mathrm{nl}}\left(\sigma\!\left(\frac{t}{\alpha_\gamma}\right)\right),
\\
\label{eq:cg-rem-leading}
\mathcal J_\gamma^{\mathrm{rem}}(t,x)
&:= \frac{\varepsilon_\gamma}{\alpha_\gamma\delta_\gamma}\sum_{i\in\Lambda_N} \kappa_\gamma\!\left(\frac{x}{\varepsilon_\gamma}-i\right)j_{i}^{\mathrm{rem}}\left(\sigma\!\left(\frac{t}{\alpha_\gamma}\right)\right).
\end{align}

We estimate the linear term and the remainder term in this section.
In order to finish these estimates, we need the following deterministic estimates for the discrete gradient of the coarse field $X_{\gamma}$.

\begin{lemma}\label{lem:eta-gamma-to-zero}
There exists a deterministic constant \(C>0\), independent of \(\gamma\) and \(T\),
such that
\begin{equation}\label{E:gardX:inf}
    \sup_{0\le t\le T}\sup_{x\in\Lambda_{\varepsilon_\gamma}}
\delta_\gamma\varepsilon_\gamma
\bigl| \nabla_{\varepsilon_\gamma}X_\gamma(t,x)
\bigr| \le C\gamma \rightarrow 0
\quad\text{as} \quad \gamma\downarrow0.
\end{equation}
\end{lemma}

\begin{proof}
Let \(x=\varepsilon_\gamma k\in\Lambda_{\varepsilon_\gamma}\). By definition of the
discrete gradient and \(X_\gamma\),
\[
\delta_\gamma\varepsilon_\gamma
\nabla_{\varepsilon_\gamma}X_\gamma(t,\varepsilon_\gamma k)
=
h_\gamma\!\left(\sigma^N\left(\frac{t}{\alpha_\gamma}\right);k+1\right)
-
h_\gamma\!\left(\sigma^N\left(\frac{t}{\alpha_\gamma}\right);k\right).
\]
Hence
\[
\begin{aligned}
\bigl|
\delta_\gamma\varepsilon_\gamma
\nabla_{\varepsilon_\gamma}X_\gamma(t,\varepsilon_\gamma k)
\bigr|
&=
\left|
\sum_{j\in\Lambda_N}
\bigl[\kappa_\gamma(k+1-j)-\kappa_\gamma(k-j)\bigr]
\sigma_j\!\left(\frac{t}{\alpha_\gamma}\right)
\right| \\
&\le
\sum_{j\in\Lambda_N}
\bigl|
\kappa_\gamma(k+1-j)-\kappa_\gamma(k-j)
\bigr|,
\end{aligned}
\]
since \(|\sigma_j|=1\). By translation invariance of the torus, the right-hand side does
not depend on \(k\), so
\[
\bigl|
\delta_\gamma\varepsilon_\gamma
\nabla_{\varepsilon_\gamma}X_\gamma(t,\varepsilon_\gamma k)
\bigr|
\le
\sum_{z\in\mathbb Z}
|\kappa_\gamma(z+1)-\kappa_\gamma(z)|.
\]
It remains to estimate the discrete variation of \(\kappa_\gamma\).
By Assumption \(({\bf A})\), we have
\[
\sum_{z\in\mathbb Z}|\kappa_\gamma(z+1)-\kappa_\gamma(z)|
\le C\gamma,
\]
which implies \eqref{E:gardX:inf}.
\end{proof}

Now we show the remainder term satisfies uniform estimate $|\nabla_{\varepsilon_{\gamma}}\mathcal J_\gamma^{\mathrm{rem}}| = o_\gamma(1)$.

\begin{theorem}\label{thm:remainder-expansion}
Assume that
\(
\frac{\gamma^4}{\alpha_\gamma\delta_\gamma}
\rightarrow 0 
\)
as \(\gamma\downarrow0\).
Then for every \(T>0\),
\[
\sup_{0\le t\le T}\sup_{x\in\Lambda_{\varepsilon_\gamma}}
\bigl|
\nabla_{\varepsilon_\gamma}\mathcal J_\gamma^{\mathrm{rem}}
\bigr|
\rightarrow 0 \quad \text{as} \quad \gamma\downarrow0.
\]
\end{theorem}

\begin{proof}
Writing \(x=\varepsilon_\gamma k\), we get
\[
 \nabla_{\varepsilon_\gamma}\mathcal J_\gamma^{\mathrm{rem}}(t,x)
=
-\frac{1}{\alpha_\gamma\delta_\gamma}
\sum_{i\in\Lambda_N}
\bigl[\kappa_\gamma(k+1-i)-\kappa_\gamma(k-i)\bigr]
j_i^{\mathrm{rem}}(t),
\]
Since
\(
F(r)=\frac{1}{1+e^r}
\)
is smooth and satisfies
\(
F''(z) \to 0,
\) as $z \to 0$,
the mean value theorem yields
\[
|F''(a)|\le \|F^{(3)}\|_{L^\infty(\mathbb R)}\,|a|
\qquad\text{for all }a\in\mathbb R.
\]
Hence
\(
|F''(a_i)|\le C\,\kappa_\gamma(1).
\)
Then by Lemma \ref{lem:eta-gamma-to-zero}, we have
\[
\big|\frac12\big(d_iF''(a_i)(b_i)^2\big)(t) \big|
\le C\,\kappa_\gamma(1) \gamma ^2,
\]
uniformly in \(i,t,\gamma\).
By Taylor's expansion, there exists \(\theta_i^\gamma\in(0,1)\) such that
\(
\widetilde R_i = \frac16\,d_i\,F^{(3)}(a_i+\theta_i^\gamma b_i)\,(b_i)^3.
\)
Since \(F^{(3)}\) is bounded and \(|d_i|\le2\), Lemma \ref{lem:eta-gamma-to-zero} yields
\(
|d_i \widetilde R_i|
\le
C\,|b_i|^3
\le
C\gamma^3.
\)
Since \(\mathfrak K\in W^{1,1}(\mathbb R)\) and \(\kappa_\gamma(z)= \gamma\mathfrak K(\gamma z)\),
we have
\(
\sum_{j\in\mathbb Z}
|\kappa_\gamma(j+1)-\kappa_\gamma(j)|
\le C\gamma.
\)
Thus 
\begin{align*}
    &\sup_{0\le t\le T}\sup_{x\in\Lambda_{\varepsilon_\gamma}}\bigl| \nabla_{\varepsilon_\gamma}\mathcal J_\gamma^{\mathrm{rem}}(t,x)\bigr| \\
\le & \frac{1}{\alpha_\gamma\delta_\gamma}
\sup_i \Big|\frac12\big(d_iF''(a_i)(b_i)^2\big)(t) \Big|
\sum_{j\in\mathbb Z}
|\kappa_\gamma(j+1)-\kappa_\gamma(j)| + \frac{1}{\alpha_\gamma\delta_\gamma}
\sup_i |\widetilde R_i^\gamma(t)|
\sum_{j\in\mathbb Z}
|\kappa_\gamma(j+1)-\kappa_\gamma(j)| \\
\le & C\,
\frac{\gamma^3\kappa_\gamma(1)}{\alpha_\gamma\delta_\gamma} + C\,
\frac{\gamma^4}{\alpha_\gamma\delta_\gamma} \to 0
\end{align*}
This proves the claim.
\end{proof}

For every test function \(\phi\in C^\infty(\mathbb T)\), with \(\psi_\gamma=\Delta_{\varepsilon_\gamma}^{-1}\phi\), we define the \(H^{-1}\)-weak drift components by
\begin{align*}
\mathcal D_\gamma^{\mathrm{lin}}(\phi;\sigma)
&:=
\bigl\langle \mathcal J_\gamma^{\mathrm{lin}}(\sigma),\nabla_{\varepsilon_\gamma}\psi_\gamma\bigr\rangle_\gamma
=
\varepsilon_\gamma\sum_{x\in\Lambda_{\varepsilon_\gamma}}
\mathcal J_\gamma^{\mathrm{lin}}(\sigma;x)\,
\nabla_{\varepsilon_\gamma}\psi_\gamma(x),\\
\mathcal D_\gamma^{\mathrm{nl}}(\phi;\sigma)
&:=
\bigl\langle \mathcal J_\gamma^{\mathrm{nl}}(\sigma),\nabla_{\varepsilon_\gamma}\psi_\gamma\bigr\rangle_\gamma
=
\varepsilon_\gamma\sum_{x\in\Lambda_{\varepsilon_\gamma}}
\mathcal J_\gamma^{\mathrm{nl}}(\sigma;x)\,
\nabla_{\varepsilon_\gamma}\psi_\gamma(x),\\
\mathcal D_\gamma^{\mathrm{rem}}(\phi;\sigma)
&:=
\bigl\langle \mathcal J_\gamma^{\mathrm{rem}}(\sigma),\nabla_{\varepsilon_\gamma}\psi_\gamma\bigr\rangle_\gamma
=
\varepsilon_\gamma\sum_{x\in\Lambda_{\varepsilon_\gamma}}
\mathcal J_\gamma^{\mathrm{rem}}(\sigma;x)\,
\nabla_{\varepsilon_\gamma}\psi_\gamma(x).
\end{align*}
Therefore,
\[
\Bigl(\alpha_\gamma^{-1}\mathscr L_\gamma^N\langle X_\gamma,\phi\rangle_{-1,\gamma}\Bigr)(\sigma)
=
\mathcal D_\gamma^{\mathrm{lin}}(\phi;\sigma)
+
\mathcal D_\gamma^{\mathrm{nl}}(\phi;\sigma)
+
\mathcal D_\gamma^{\mathrm{rem}}(\phi;\sigma).
\]
Here, the linear term
\begin{align}\label{eq:Dlin-correct-start}
\mathcal D_\gamma^{\mathrm{lin}}(\phi;\sigma)
&=
\varepsilon_\gamma\sum_{x} \Bigl(\frac12-\beta\kappa_\gamma(1)\Bigr)
\frac{\varepsilon_\gamma}{\alpha_\gamma\delta_\gamma}
\sum_{i\in\Lambda_N}
\kappa_\gamma\!\left(\frac{x}{\varepsilon_\gamma}-i\right)d_i(\sigma) \nabla_{\varepsilon_\gamma}\psi_\gamma(x)\notag\\
&=
\Bigl(\frac12-\beta\kappa_\gamma(1)\Bigr)
\frac{\varepsilon_\gamma^2}{\alpha_\gamma\delta_\gamma}
\sum_{i\in\Lambda_N}
d_i(\sigma)\,
\nabla_{\varepsilon_\gamma}(\mathcal K_\gamma^\varepsilon\psi_\gamma)(\varepsilon_\gamma i)  \nonumber \\
&= -\Bigl(\frac12-\beta\kappa_\gamma(1)\Bigr)
\frac{\varepsilon_\gamma^2}{\alpha_\gamma}
\bigl\langle X_\gamma,\mathcal K_\gamma^\varepsilon\phi\bigr\rangle_\gamma .
\end{align}
Now we consider the expansion for the \(H^{-1}\)-linear part $\mathcal D_\gamma^{\mathrm{lin}}(\phi)$.

\begin{theorem}\label{thm:expand linear part}
For every zero-mean \(\phi\in C^\infty(\mathbb T)\),
\[
\int_0^t \mathcal D_\gamma^{\mathrm{lin}}(\phi)\Big(\sigma^N\big(\frac{s}{\alpha_\gamma}\big)\Big)\,ds
=
-A'_\gamma\int_0^t\langle X_\gamma(s),\phi\rangle_\gamma\,ds
-\nu_\gamma\int_0^t\langle X_\gamma(s),\Delta_{\varepsilon_\gamma}\phi\rangle_\gamma\,ds
+o_\gamma(1)
\]
where
\[
 \nu_{\gamma} :=
\frac{\mathfrak m_{2}}{4}
\frac{\varepsilon_\gamma^4}{\alpha_\gamma\gamma^2} \to  \nu, \quad \text{and} \quad A'_\gamma := \Bigl(\frac12-\beta\kappa_\gamma(1)\Bigr)\frac{\varepsilon_\gamma^2}{\alpha_\gamma}.
\]
\end{theorem}

\begin{proof}
Using the discrete integration by parts and the commutation relation
\(\Delta_{\varepsilon_\gamma}\mathcal K_\gamma^\varepsilon\psi_\gamma
=\mathcal K_\gamma^\varepsilon\phi\), the linear part $\mathcal D_\gamma^{\mathrm{lin}}(\phi)$ can be rewritten as
\begin{equation}\label{eq:linear-drift-exact}
\mathcal D_\gamma^{\mathrm{lin}}(\phi)
=
-\Bigl(\frac12-\beta\kappa_\gamma(1)\Bigr)
\frac{\varepsilon_\gamma^2}{\alpha_\gamma}
\bigl\langle X_\gamma,\mathcal K_\gamma^\varepsilon\phi\bigr\rangle_\gamma.
\end{equation}
Here the factor \(\delta_\gamma^{-1}\) is exactly absorbed into \(X_\gamma\).

We now expand the operator \(\mathcal K_\gamma^\varepsilon\) on smooth functions.
Since \(\kappa_\gamma\) is even and has finite second and fourth moments, a standard discrete moment expansion gives
\begin{equation}
\label{eq:Kac-test-expansion}
\mathcal K_\gamma^\varepsilon\phi =\phi
+\frac{\mathfrak m_{2,\gamma}}{2}\Bigl(\frac{\varepsilon_\gamma}{\gamma}\Bigr)^2
\Delta_{\varepsilon_\gamma}\phi
+r_\gamma(\phi),
\end{equation}
where 
\[
\mathfrak m_{2,\gamma} :=
\gamma^2\sum_{z\in\mathbb Z} z^2\,\kappa_\gamma(z) = \gamma\sum_{z\in\mathbb Z}(\gamma z)^2\,\mathfrak K(\gamma z) \to \int_{\mathbb R} z^2 \mathfrak K(z) dz = \mathfrak m_2,
\]
and the remainder satisfies $\|r_\gamma(\phi)\|_\infty
\le
C_\phi\Bigl(\frac{\varepsilon_\gamma}{\gamma}\Bigr)^4.$
Substituting \eqref{eq:Kac-test-expansion} into \eqref{eq:linear-drift-exact} yields
\begin{equation}\label{eq:linear-drift-three-pieces}
    \mathcal D_\gamma^{\mathrm{lin}}(\phi) = -\Bigl(\frac12-\beta\kappa_\gamma(1)\Bigr)\frac{\varepsilon_\gamma^2}{\alpha_\gamma}
\langle X_\gamma,\phi\rangle_\gamma
\notag 
-\Bigl(\frac12-\beta\kappa_\gamma(1)\Bigr)\frac{\mathfrak m_{2,\gamma}}{2}\frac{\varepsilon_\gamma^4}{\alpha_\gamma\gamma^2}
\langle X_\gamma,\Delta_{\varepsilon_\gamma}\phi\rangle_\gamma
\notag - \frac{\varepsilon_\gamma^2}{\alpha_\gamma}
\langle X_\gamma,r_\gamma(\phi)\rangle_\gamma .
\end{equation}
The first two terms in the identity above evaluate to
\(
-A'_\gamma\langle X_\gamma,\phi\rangle_\gamma
-\nu_\gamma\langle X_\gamma,\Delta_{\varepsilon_\gamma}\phi\rangle_\gamma.
\).
It remains to prove that the last term is negligible.

By the weak \(L^2\)-bound on \(X_\gamma\) and the uniform bound
\(\|r_\gamma(\phi)\|_\infty\le C_\phi(\varepsilon_\gamma/\gamma)^4\), we have
\begin{align*}
\mathbb E\Bigg[
\Bigg|
(\frac12-\beta\kappa_\gamma(1))\frac{\varepsilon_\gamma^2}{\alpha_\gamma}
\int_0^t
\langle X_\gamma(s),r_\gamma(\phi)\rangle_\gamma\,ds
\Bigg|^2
\Bigg]
\le &
C_t
\Bigl((\frac12-\beta\kappa_\gamma(1))\frac{\varepsilon_\gamma^2}{\alpha_\gamma}\Bigr)^2
\|r_\gamma(\phi)\|_\infty^2  \\
\le &
C_{t,\phi}
\Bigl(\frac{\varepsilon_\gamma^2}{\alpha_\gamma}\Bigr)^2\Bigl(\frac{\varepsilon_\gamma}{\gamma}\Bigr)^8 .
\end{align*}
Using the boundedness of \(A'_\gamma\) and the fact that
\(\frac12-\beta\kappa_\gamma(1)\) is bounded away from zero for small \(\gamma\), we obtain the needed bound on \(\frac{\varepsilon_\gamma^2}{\alpha_\gamma}\). In particular, as $\gamma \to 0$,
\[
 \frac{\varepsilon_\gamma^2}{\alpha_\gamma}
\int_0^t
\langle X_\gamma(s),r_\gamma(\phi)\rangle_\gamma\,ds
\rightarrow 0
\quad\text{in } L^2.
\]
This proves the claim.
\end{proof}

Combining Theorem \ref{thm:remainder-expansion} and Theorem \ref{thm:expand linear part}, the \(H^{-1}\)-tested drift of the rescaled field is
\begin{align}
\Bigl(\alpha_\gamma^{-1}\mathscr L_\gamma^N
\langle X_\gamma,\phi\rangle_{-1,\gamma}\Bigr)(\sigma)
=&
-A'_\gamma\langle X_\gamma,\phi\rangle_\gamma
-\nu_\gamma\langle X_\gamma,\Delta_{\varepsilon_\gamma}\phi\rangle_\gamma
\nonumber\\
&\quad
+\mathcal D_\gamma^{\mathrm{nl}}(\phi;\sigma)
+\mathcal D_\gamma^{\mathrm{rem}}(\phi;\sigma)
+o_\gamma(1).
\label{eq:Hminus-one-drift-summary}
\end{align}
Here the nonlinear current $\nabla_{\varepsilon_{\gamma}} \mathcal J_\gamma^{\mathrm{nl}}$ is defined in \eqref{eq:cg-nl-leading} and the \(H^{-1}\)-tested nonlinear drift has the form
\begin{align}\label{eq:Hminus-one-nonlinear-drift}
    \mathcal D_\gamma^{\mathrm{nl}}(\phi;\sigma)
    = & -\frac{\beta}{4}
\frac{\varepsilon_\gamma^2}{\alpha_\gamma\delta_\gamma}
\sum_{i,j\in\Lambda_N}
\kappa_\gamma(i-j)
d_i(\sigma)^2d_j(\sigma) G_{\gamma,i}^{\phi}
+ r_{\gamma,-1}^{(0)}(\phi;t)  \nonumber\\
    = & \frac{\beta\varepsilon_\gamma^3}{4\alpha_\gamma}
    \sum_{i\in\Lambda_N}
    d_i(\sigma)^2\,\nabla_{\varepsilon_\gamma}X_\gamma(\varepsilon_\gamma i)\,
    \nabla_{\varepsilon_\gamma}
    \bigl(\mathcal K_\gamma^\varepsilon\Delta_{\varepsilon_\gamma}^{-1}\phi\bigr)(\varepsilon_\gamma i)
    + r_{\gamma,-1}^{(0)}(\phi;t).
\end{align}
where $G_{\gamma,i}^{\phi}:=
\nabla_{\varepsilon_\gamma}
\bigl(\mathcal K_\gamma^\varepsilon\Delta_{\varepsilon_\gamma}^{-1}\phi\bigr)
(\varepsilon_\gamma i)$.
Moreover, by \(\kappa_\gamma(1)\le C\gamma\),
Assumption \(({\bf B})\), and the deterministic Kac-gradient bound \eqref{E:gardX:inf}, the error term
\begin{align}
|r_{\gamma,-1}^{(0)}(\phi;t)|
&\le
C\frac{\varepsilon_\gamma^2}{\alpha_\gamma\delta_\gamma}
\sum_{i\in\Lambda_N}
\kappa_\gamma(1)^2\,
\delta_\gamma\varepsilon_\gamma
\bigl|\nabla_{\varepsilon_\gamma}X_\gamma(\varepsilon_\gamma i)\bigr|\,
\bigl|G_{\gamma,i}^{\phi}\bigr|
\nonumber\\
&\le
C_\phi
\frac{\varepsilon_\gamma^2}{\alpha_\gamma\delta_\gamma}
\kappa_\gamma(1)^2
\sum_{i\in\Lambda_N}
\delta_\gamma\varepsilon_\gamma
\bigl|\nabla_{\varepsilon_\gamma}X_\gamma(\varepsilon_\gamma i)\bigr|
\nonumber\\
& \le C_\phi \frac{\varepsilon_\gamma\gamma^3}{\alpha_\gamma\delta_\gamma} \to 0.
\label{eq:Hminus-one-nonlinear-error}
\end{align}
In the last line, we used \(|\Lambda_N|\simeq\varepsilon_\gamma^{-1}\),
\(\kappa_\gamma(1)\le C\gamma\).
\eqref{E:gardX:inf}.  
Assumption \(({\bf B})\), the Taylor-linearization error 
\(
\frac{\varepsilon_\gamma\gamma^3}{\alpha_\gamma\delta_\gamma} = O(\varepsilon_\gamma^{1/2})\rightarrow0 .
\)
The remainder part is 
\begin{equation}\label{eq:Hminus-one-remainder-drift} \mathcal D_\gamma^{\mathrm{rem}}(\phi;\sigma)
=
\bigl\langle
\mathcal J_\gamma^{\mathrm{rem}}(\sigma),
\nabla_{\varepsilon_\gamma}\Delta_{\varepsilon_\gamma}^{-1}\phi
\bigr\rangle_\gamma,
\end{equation}
and Theorem \ref{thm:remainder-expansion} gives
\[
\sup_{0\le t\le T}
\bigl|
\mathcal D_\gamma^{\mathrm{rem}}(\phi;\sigma^N(t/\alpha_\gamma))
\bigr|
\le
C_\phi
\sup_{0\le t\le T}\sup_{x\in\Lambda_{\varepsilon_\gamma}}
\bigl|
\nabla_{\varepsilon_\gamma}\mathcal J_\gamma^{\mathrm{rem}}(t,x)
\bigr|
\longrightarrow0 .
\]
If the second-order Boltzmann--Gibbs replacement of the next section is applied to \(d_i^2\), then the leading nonlinear term in \eqref{eq:Hminus-one-nonlinear-drift} formally becomes
\[
\frac{\beta}{4}\frac{\varepsilon_\gamma^3}{\alpha_\gamma}
\sum_{i\in\Lambda_N}
\Bigl(a_{0,\gamma}+a_{2,\gamma}\delta_\gamma^2X_\gamma(\varepsilon_\gamma i)^2\Bigr)
\nabla_{\varepsilon_\gamma}X_\gamma(\varepsilon_\gamma i)\,
\nabla_{\varepsilon_\gamma}
\bigl(\mathcal K_\gamma^\varepsilon\Delta_{\varepsilon_\gamma}^{-1}\phi\bigr)(\varepsilon_\gamma i),
\]

It is useful to keep the microscopic gradient hidden in the Kac increment rather
than to treat \(d_i^2(h_{\gamma,i+1}-h_{\gamma,i})\) as an ordinary product of a
local observable and a mesoscopic increment.  Since
\[
h_{\gamma,i}(\sigma)=\sum_{z\in\Lambda_N}\kappa_\gamma(i-z)\sigma_z,
\qquad
d_j(\sigma):=\sigma_j-\sigma_{j+1},
\]
translation on the torus and the evenness of \(\kappa_\gamma\) give the exact identity
\begin{equation}\label{eq:kac-increment-as-bond-average}
    h_{\gamma,i+1}(\sigma)-h_{\gamma,i}(\sigma) = \sum_{z\in\Lambda_N}\kappa_\gamma(i-z)
\bigl(\sigma_{z+1}-\sigma_z\bigr) = -\sum_{z\in\Lambda_N}\kappa_\gamma(i-z)d_z(\sigma).
\end{equation}
Thus the leading nonlinear current can be written in terms of
microscopic bond variables:
\[
j_i^{\mathrm{nl}}(\sigma)
=
-\frac{\beta}{4}d_i(\sigma)^2
\sum_{j\in\Lambda_N}\kappa_\gamma(i-j)d_j(\sigma)
+\widetilde r_{i,\gamma}^{\,\mathrm{nl}}(\sigma).
\]
Consequently the \(H^{-1}\)-tested nonlinear drift has the microscopic-bond
form
\begin{equation}\label{eq:clean-microscopic-bond-drift}
\mathcal D_\gamma^{\mathrm{nl}}(\phi;\sigma)
=
-\frac{\beta}{4}
\frac{\varepsilon_\gamma^2}{\alpha_\gamma\delta_\gamma}
\sum_{i,j\in\Lambda_N}
\kappa_\gamma(i-j)d_i(\sigma)^2d_j(\sigma)G_{\gamma,i}^{\phi}
+r_{\gamma,-1}^{(0)}(\phi;t).
\end{equation}
This is the form used in the next section.  

\section{Second-order Boltzmann--Gibbs Principle}
\label{sec:clean-second-order-bg}

This section extracts the effective second-order Boltzmann--Gibbs replacement from the preceding multiscale discussion in a self-contained form.  It gives the
streamlined argument used for the nonlinear drift.  We keep only the route that closes the estimate: one first rewrites the nonlinear current with the microscopic bond gradient \( d_i(t):=\sigma^N(t/\alpha_\gamma)_{i}-\sigma^N(t/\alpha_\gamma)_{i+1}\), then performs the
replacement at this microscopic-gradient level by a discrete chain rule.

We use the block assumptions
\begin{equation}\label{ass:clean-bg-scales}
    \ell = \ell_\gamma=\varepsilon_\gamma^{-1/4},\qquad
L=L_\gamma=\varepsilon_\gamma^{-1/3}.
\end{equation}
so that under Assumption \(({\bf B})\)
\begin{equation}
\ell_\gamma\to\infty,\qquad L_\gamma\to\infty,\qquad
\frac{\ell_\gamma}{L_\gamma}\to0,\qquad
\gamma L_\gamma^2\to0.
\end{equation}

\subsection{Quantitative one-block estimate}
\label{subsec:quantitative-one-block}

We write
\(
\Lambda_{i,\ell}:=\{i-\ell,\dots,i+\ell\}.
\)
For \(r\in\{\ell, L\}\), define
\[
m_i^r(\sigma):=\frac{1}{2r+1}\sum_{|u|\le r}\sigma_{i+u}.
\]
For a bounded local observable \(\Psi\), let \(\Phi_{\Psi,r}^{\gamma}(m)\) be its expectation under the canonical Gibbs
measure on the block of radius \(r\), with block magnetization \(m\).
we also define the one-block fluctuation
\[
V_{i,\ell}^{\Psi,\gamma}(\sigma)
:=
\tau_i\Psi(\sigma)
-
\Phi_{\Psi,\ell}^{\gamma}\bigl(m_i^\ell(\sigma)\bigr).
\]
We also use the exact frozen-exterior canonical average
\[
\Phi_{\Psi,i,\ell}^{\gamma,\xi}(m)
:=
\int \Psi(\eta)\,
\mu_{i,\ell,m}^{\mathrm{can},\gamma,\xi}(d\eta),
\]
where \(\xi=\sigma_{\Lambda_N\setminus\Lambda_{i,\ell}}\) denotes the exterior
configuration and \(\mu_{i,\ell,m}^{\mathrm{can},\gamma,\xi}\) is the Gibbs
measure on the block \(\Lambda_{i,\ell}\), conditioned on the exterior
\(\xi\) and on the block magnetization \(m\).  The simplified canonical
average \(\Phi_{\Psi,\ell}^{\gamma}\) is obtained by replacing the frozen
exterior field by its local constant part.

\begin{lemma}[Uniform local spectral gap]\label{lem:clean-one-block-spectral-gap}
Assume \(\gamma\ell^2\le 1\).  There exists a constant \(C_{\mathrm{sg}}>0\), independent of \(i,\xi,m,\gamma\), and \(\ell\), such that for every function \(G:\Sigma_{\ell,m}^{\mathrm{can}}\to\mathbb R\),
\[
\mathrm{Var}_{\mu_{i,\ell,m}^{\mathrm{can},\gamma,\xi}}(G)
\le
C\ell^2
\sum_{|j|\le \ell-1}
D_{j,j+1}^{(\ell),\gamma}(G).
\]
\end{lemma}

\begin{proof}
Let \(\nu_{\ell,m}^{\mathrm{aux}}\) be the uniform measure on the canonical sector \(\Sigma_{\ell,m}^{\mathrm{can}}\).  The standard spectral gap estimate for nearest-neighbor simple exclusion on an interval of length \(2\ell+1\) gives
\begin{equation}\label{eq:aux-spectral-gap}
\mathrm{Var}_{\nu_{\ell,m}^{\mathrm{aux}}}(G)
\le
C\ell^2
\sum_{|j|\le\ell-1}
\frac12
\int
\bigl(G(\eta^{j,j+1})-G(\eta)\bigr)^2
\nu_{\ell,m}^{\mathrm{aux}}(d\eta).
\end{equation}
We next compare the frozen-exterior Gibbs measure with \(\nu_{\ell,m}^{\mathrm{aux}}\).  On the fixed-magnetization sector, the constant part
\(
-\frac12\gamma\mathfrak K(0)\left(\sum_{u\in\Lambda_\ell}\eta_u\right)^2
\)
of the block Hamiltonian is absorbed into the normalizing constant.  The remaining part satisfies
\begin{equation}\label{eq:oneOscH}
    \left|
\frac12\sum_{u,v\in\Lambda_\ell}
\bigl(\kappa_\gamma(u-v)-\gamma\mathfrak K(0)\bigr)\eta_u\eta_v
\right|
\le C\gamma^2\ell^3\le C,
\end{equation}
because \(\mathfrak K\in W^{1,\infty}\) and \(\gamma\ell^2\le1\).  The nonconstant part of the frozen exterior field is bounded by \(C\gamma\ell^2\le C\) by Lemma~\ref{lem:remove-frozen-field}.  Hence
\begin{equation}\label{eq:frozen-uniform-comparison}
C^{-1}\nu_{\ell,m}^{\mathrm{aux}}(d\eta)
\le
\mu_{i,\ell,m}^{\mathrm{can},\gamma,\xi}(d\eta)
\le
C\nu_{\ell,m}^{\mathrm{aux}}(d\eta),
\end{equation}
with \(C\) independent of \(i,\xi,m,\ell,\gamma\).  Moreover, the exchange rates are uniformly elliptic: the same estimate as in Lemma~\ref{lem:local-rates-elliptic}, using the \(L^1\)-normalization of \(\kappa_\gamma\), gives
\[
0<c_-\le c_{j,j+1}^{(\ell),\gamma,\xi}(\eta)\le c_+<\infty
\]
uniformly in all parameters.  Combining \eqref{eq:aux-spectral-gap} with \eqref{eq:frozen-uniform-comparison} and the lower bound \(c_-\) gives the claimed spectral gap.
\end{proof}

\begin{theorem}[Quantitative one-block estimates]
\label{thm:clean-one-block}
Let \(\Psi\) be a bounded local observable with
\(\operatorname{supp}\Psi\subset\Lambda_r\), and let \(J_i\) be a bounded
deterministic lattice function.  Assume \(\ell>r\) and
\(\gamma\ell^2\le1\).  Then
\[
\begin{aligned}
&\mathbb E_{f_{N,0}}
\left|
\int_0^T
\varepsilon_\gamma\sum_i
J_i
\left[
\tau_i\Psi(\sigma^N(t/\alpha_\gamma))
-
\Phi_{\Psi,\ell}^{\gamma}
\bigl(m_i^\ell(\sigma^N(t/\alpha_\gamma))\bigr)
\right]dt
\right|
\\
&\qquad\le
C_\Psi\|J\|_{\ell^\infty}
\left(
\sqrt{T\alpha_\gamma\ell^3}
+T\gamma\ell^2
\right).
\end{aligned}
\]
The same estimate holds when \(J_i=J_i(t,\sigma)\) is bounded and measurable
with respect to the exterior of the block \(\{i-\ell,\dots,i+\ell\}\).
\end{theorem}

\begin{proof}
By Lemma~\ref{lem:remove-frozen-field}, equivalently by the density
comparison \eqref{eq:local-can-comparison}, the nonconstant frozen exterior
field produces a block Hamiltonian oscillation of order \(C\gamma\ell^2\). Therefore 
\begin{equation}\label{eq:frozen-field co}
    \mathbb E_{f_{N,0}} \left| \int_0^T
\varepsilon_\gamma\sum_i
J_i(t,\sigma^N(t/\alpha_\gamma))
\left[ \tau_i \Phi_{\Psi,\ell}^{\gamma}(\sigma^N(t/\alpha_\gamma))
-
\Phi_{\Psi,i,\ell}^{\gamma,\xi}(m_i^\ell(\sigma^N(t/\alpha_\gamma)))
\right]dt \right|
\end{equation}
It remains to estimate the centered frozen-exterior part.

Fix \(t \in [0,T]\). Now we write the Gibbs measure relative to the exterior of
\(\Lambda_{i,\ell}\) as
\[
\mu_{N,\gamma,\beta}(d\sigma)
=
\nu_{i,\ell}^{N,\gamma}(d\xi)\,
\mu_{i,\ell,m(\xi)}^{\mathrm{can},\gamma,\xi}(d\eta),
\qquad
\sigma=\eta\oplus_i\xi,
\]
and write the density \(f_{N,t}\) as
\[
f_{N,t}(\eta\oplus_i\xi)
= f_{i,\ell,t}^{\mathrm{out}}(\xi)\,
q_{i,\ell,t}^{\gamma,\xi}(\eta),
\]
where \(q_{i,\ell,t}^{\gamma,\xi}\) is a density with respect to
\(\mu_{i,\ell,m(\xi)}^{\mathrm{can},\gamma,\xi}\).  Define the one-block local
Fisher information
\[
\mathcal I_{i,\ell}^{\gamma,\xi}(q) :=
\sum_{|u|\le\ell-1}
\frac12
\int
c_{u,u+1}^{(\ell),\gamma,\xi}(\eta)
\left(
\sqrt{q(\eta^{u,u+1})}-\sqrt{q(\eta)}
\right)^2
\mu_{i,\ell,m(\xi)}^{\mathrm{can},\gamma,\xi}(d\eta).
\]
For
\[
H_{i,\ell}^{\Psi,\xi}(\eta) :=
\Psi(\eta) - \Phi_{\Psi,i,\ell}^{\gamma,\xi}(m^\ell(\eta)),
\]
we have \(\int H_{i,\ell}^{\Psi,\xi}d\mu_{i,\ell,m}^{\mathrm{can},\gamma,\xi}=0\)
on every canonical sector.  
If
\(\bar q:=\int\sqrt q\,d\mu_{i,\ell,m}^{\mathrm{can},\gamma,\xi}\), then
\[
\int H_{i,\ell}^{\Psi,\xi}q\,d\mu
=
\int
H_{i,\ell}^{\Psi,\xi}
\bigl(\sqrt q-\bar q\bigr)
\bigl(\sqrt q+\bar q\bigr)d\mu,
\qquad
\mu=\mu_{i,\ell,m}^{\mathrm{can},\gamma,\xi}.
\]
Since \(H_{i,\ell}^{\Psi,\xi}\) is bounded and
\(\int q\,d\mu=1\), the last factor has a uniform \(L^2(\mu)\)-bound.  The
Cauchy--Schwarz inequality and Lemma~\ref{lem:clean-one-block-spectral-gap}
applied to \(\sqrt q\) give
\[
\left|
\int H_{i,\ell}^{\Psi,\xi}q\,d\mu
\right|
\le
C_\Psi\ell\,
\left(\mathcal I_{i,\ell}^{\gamma,\xi}(q)\right)^{1/2}.
\]
Since \(J_i(t,\sigma)\) is exterior-measurable, it is frozen under the conditioning.  Hence
\[
\begin{aligned}
&\mathbb E_{f_{N,0}}
\left|
\varepsilon_\gamma\sum_i
J_i(t,\sigma^N(t/\alpha_\gamma))
\left[
\tau_i\Psi(\sigma^N(t/\alpha_\gamma))
-
\Phi_{\Psi,i,\ell}^{\gamma,\xi}(m_i^\ell(\sigma^N(t/\alpha_\gamma)))
\right]
\right|
\\
&\qquad\le
C_\Psi\|J\|_{\ell^\infty}\ell\,
\varepsilon_\gamma\sum_i
\int
f_{i,\ell,t}^{\mathrm{out}}(\xi)
\left(
\mathcal I_{i,\ell}^{\gamma,\xi}
\bigl(q_{i,\ell,t}^{\gamma,\xi}\bigr)
\right)^{1/2}
\nu_{i,\ell}^{N,\gamma}(d\xi).
\end{aligned}
\]
Note that each internal local exchange is the corresponding global exchange, and after summing over all block centers each global bond is counted at most \(C\ell\) times. 
Using Lemma~\ref{lem:clean-one-block-FI-projection}, and applying the Cauchy--Schwarz inequality yields
\[
\begin{aligned}
&\mathbb E_{f_{N,0}}
\left|
\varepsilon_\gamma\sum_i
J_i(t,\sigma^N(t/\alpha_\gamma))
\left[
\tau_i\Psi(\sigma^N(t/\alpha_\gamma))
-
\Phi_{\Psi,i,\ell}^{\gamma,\xi}(m_i^\ell(\sigma^N(t/\alpha_\gamma)))
\right]
\right|
\\
&\qquad\le
C_\Psi\|J\|_{\ell^\infty}\ell
\left(
\varepsilon_\gamma\sum_i
\int
f_{i,\ell,t}^{\mathrm{out}}(\xi)
\mathcal I_{i,\ell}^{\gamma,\xi}
\bigl(q_{i,\ell,t}^{\gamma,\xi}\bigr)
\nu_{i,\ell}^{N,\gamma}(d\xi)
\right)^{1/2} \\
&\qquad\le
C_\Psi\|J\|_{\ell^\infty}
\left(
T\varepsilon_\gamma\ell^3
\int_0^T I^{N,\gamma}(f_{N,t})dt
\right)^{1/2}.
\end{aligned}
\]
By Lemma~\ref{lem:entropy-dissipation-kawasaki} and
Assumption \(({\bf C})\), we have $\int_0^T I^{N,\gamma}(f_{N,t})dt \le C\alpha_\gamma\varepsilon_\gamma^{-1}$.
Therefore
\begin{equation*}
    \mathbb E_{f_{N,0}}\left|\int_0^T
\varepsilon_\gamma\sum_i
J_i(t,\sigma^N(t/\alpha_\gamma))
\left[
\tau_i\Psi(\sigma^N(t/\alpha_\gamma)) - \Phi_{\Psi,i,\ell}^{\gamma,\xi}(m_i^\ell(\sigma^N(t/\alpha_\gamma)))
\right]dt \right| \le 
C_\Psi\|J\|_{\ell^\infty}
\sqrt{T\alpha_\gamma\ell^3}.
\end{equation*}
Combining this estimate with the frozen-field comparison \eqref{eq:frozen-field co} gives the quantitative one-block estimates.
\end{proof}

\subsection{Kac block estimates}

In this section, we consider the observable \(\Psi=d_i^2\) in the drift term, and give the local equilibrium average for the local observable \(d_i^2\).

\begin{lemma}\label{lem:clean-explicit-dsquare}
The local equilibrium average for local observable \(\Psi=d_i^2\) is given by
\[
\Phi_{\Psi,L}^{\gamma}(m)
=
a_{0,L}+a_{2,L}m^2+\rho_{\gamma,L}(m),
\qquad
a_{0,L}=2+\frac1L,\qquad
a_{2,L}=-2-\frac1L,
\]
where
\(
\sup_{|m|\le1}|\rho_{\gamma,L}(m)|
\le C\gamma^2L^3 .
\)
\end{lemma}

\begin{proof}
Under the auxiliary uniform canonical measure \(\nu_{\ell,m}^{\mathrm{aux}}\) on a block of size \(2L+1\), exchangeability implies that for every \(i\neq j\),
\[
\mathbb E_{\nu_{L,m}^{\mathrm{aux}}}[\eta_i\eta_j]=\mathbb E_{\nu_{\ell,m}^{\mathrm{aux}}}(\sigma_0\sigma_{1})
=
\frac{(2L+1)^2m^2-(2L+1)}{(2L+1)(2L)}.
\]
Since \(d_i^2=2-2\sigma_i\sigma_{i+1}\), this yields
\begin{equation}\label{eq:exaux}
    \mathbb E_{\nu_{L,m}^{\mathrm{aux}}}(d_i^2) = \left(2+\frac1L\right) -\left(2+\frac1L\right)m^2.
\end{equation}
By estimate \eqref{eq:oneOscH}, the Kac interaction restricted to a block $\Lambda_L$ has oscillation bounded by \(C\gamma^2L^3\) on each canonical sector.  
Comparing the density of canonical Gibbs measure $\mu_{\gamma,L,m}^{\mathrm{can}}$ with respect to the auxiliary uniform canonical measure $\nu_{L,m}^{\mathrm{aux}}$, we have
\[
\left\|
\mu_{\gamma,L,m}^{\mathrm{can}}-\nu_{L,m}^{\mathrm{aux}}
\right\|_{\mathrm{TV}}
\le
C_\beta\,\gamma^2L^3.
\]
Since \(0\le d_i^2\le 4\), we obtain
\[
\left|
\Phi_{\Psi,L}^{\gamma}(m)
-
\int d_i^2\,\nu_{L,m}^{\mathrm{aux}}(d\eta)
\right|
\le
4\left\|
\mu_{\gamma,L,m}^{\mathrm{can}}-\nu_{L,m}^{\mathrm{aux}}
\right\|_{\mathrm{TV}}
\le
C \gamma^2L^3.
\]
uniformly in \(m\).
Combining this with the exact formula \eqref{eq:exaux} proves the lemma.
\end{proof}

As in Lemma~\ref{lem:weak-L2-h-from-microscopic}, we have the following \(L^2\) estimate for $|m_i^\ell(\sigma^N(t/\alpha_\gamma))-h_\gamma(\sigma^N(t/\alpha_\gamma);i)|^2$.

\begin{lemma}\label{lem:clean-strong-kac-defect}
Under Assumptions \(({\bf A})\)--\(({\bf C})\), and the block assumption \eqref{ass:clean-bg-scales}, we have
\begin{align}\label{eq:centered-block-to-kac-L2-rate-main}
    \mathbb E_{f_{N,0}}\int_0^T
\varepsilon_\gamma\sum_i
\left|
 m_i^\ell(\sigma^N(t/\alpha_\gamma))
-h_\gamma(\sigma^N(t/\alpha_\gamma);i)
\right|^2dt 
\le C\varepsilon^{1/4}_{\gamma} \to 0 \quad \text{as} \quad \gamma \to 0.
\end{align}
For every smooth mean-zero \(\phi\), define
\[
W_{\gamma,j}^{\ell,\phi}(t)
:=
\sum_i
\kappa_\gamma(i-j)G_{\gamma,i}^{\phi}
\left[
 m_i^\ell(\sigma^N(t/\alpha_\gamma))
-h_\gamma(\sigma^N(t/\alpha_\gamma);i)
\right].
\]
Then
\begin{equation}\label{eq:kac-covariance}
    \mathbb E_{f_{N,0}}\int_0^T \varepsilon_\gamma\sum_j
\left|W_{\gamma,j}^{\ell,\phi}(t)\right|^2dt
\le C\varepsilon^{1/4}_{\gamma}  \to 0 \quad \text{as} \quad \gamma \to 0.
\end{equation}
\end{lemma}

\begin{proof}
Since both \(m_i^\ell\) and \(h_{\gamma,i}\) preserve the zero Fourier mode, we have
\[
m_i^\ell-m_N=(a_\ell*(\sigma-m_N))_i,
\qquad
h_{\gamma,i}-m_N=(\kappa_\gamma*(\sigma-m_N))_i,
\]
where
\[
a_\ell(r):=\frac{\mathbf 1_{\{|r|\le \ell\}}}{2\ell+1},
\qquad
\widehat a_\ell(k)
:=
\sum_{|r|\le\ell}a_\ell(r)e^{-2\pi i k\varepsilon_\gamma r}.
\]
Thus their difference has Fourier multiplier \(\widehat a_\ell(k)-\widehat\kappa_\gamma(k)\) on the nonzero modes.  
By Parseval's identity and \eqref{eq:static-infrared-bound}, we get
\begin{equation}\label{eq:block-to-kac-fourier-bound-main}
    \int_{\Sigma_{N,M}}
\varepsilon_\gamma\sum_i \left|  m_i^\ell(\sigma)-h_\gamma(\sigma;i) \right|^2 \mu_{N,\gamma,\beta}(d\sigma) \le
C\varepsilon_\gamma
\sum_{k\ne0} \frac{\left|\widehat a_\ell(k)-\widehat\kappa_\gamma(k)\right|^2}{\theta_\gamma+1-\widehat\kappa_\gamma(k)}.
\end{equation}
Now we split the above summation of modes into low-frequency region \(|k|\le c\gamma/\varepsilon_\gamma\) and high-frequency region \(|k| > c\gamma/\varepsilon_\gamma\).  

For low-frequency region \(|k|\le c\gamma/\varepsilon_\gamma\), put
\(q=\varepsilon_\gamma|k|/\gamma\).  Assumption \(({\bf A})\) implies
\[
1-\widehat\kappa_\gamma(k)\asymp q^2,
\qquad
|1-\widehat\kappa_\gamma(k)|\le Cq^2,
\qquad
|1-\widehat a_\ell(k)|\le C(\gamma\ell)^2q^2.
\]
Since
\[
\widehat a_\ell(k)-\widehat\kappa_\gamma(k)
=
\bigl(1-\widehat\kappa_\gamma(k)\bigr)
-
\bigl(1-\widehat a_\ell(k)\bigr),
\]
we have $\left|\widehat a_\ell(k)-\widehat\kappa_\gamma(k)\right| \le Cq^2$ on this low-frequency region.  
Therefore, comparing the sum with the Riemann sum in the variable \(q=\varepsilon_\gamma |k|/\gamma\), we have
\begin{equation}\label{eq:lowfre}
    \varepsilon_\gamma
\sum_{0<|k|\le c\gamma/\varepsilon_\gamma}
\frac{\left|\widehat a_\ell(k)-\widehat\kappa_\gamma(k)\right|^2}{\theta_\gamma+1-\widehat\kappa_\gamma(k)}
\le
C\varepsilon_\gamma
\sum_{0<|k|\le c\gamma/\varepsilon_\gamma}
\frac{q^4}{\theta_\gamma+cq^2}
\le C\gamma.
\end{equation}
Since
\[
\frac{q^4}{\theta_\gamma+cq^2}\le Cq^2,
\qquad
\varepsilon_\gamma
\sum_{0<|k|\le c\gamma/\varepsilon_\gamma}
\left(\frac{\varepsilon_\gamma |k|}{\gamma}\right)^2
\le C\gamma,
\]
the last bound \eqref{eq:lowfre} is uniform in \(\theta_\gamma>0\).
Hence the low-frequency contribution is bounded by \(C\gamma\).  

On the high-frequency region \(1-\widehat\kappa_\gamma(k)\ge c\), by Parseval's identity and Lemma~\ref{lem:weak-L2-h-from-microscopic}, we have
\begin{align}\label{eq:highfre}
    \varepsilon_\gamma \sum_{|k|>c\gamma/\varepsilon_\gamma}
\frac{\left|\widehat a_\ell(k)-\widehat\kappa_\gamma(k)\right|^2}{\theta_\gamma+1-\widehat\kappa_\gamma(k)}  \le & C\varepsilon_\gamma\sum_{|k|>c\gamma/\varepsilon_\gamma}|\widehat a_\ell(k)|^2
+C\varepsilon_\gamma\sum_{|k|>c\gamma/\varepsilon_\gamma}|\widehat\kappa_\gamma(k)|^2
+C\frac{\varepsilon_\gamma}{\gamma} \nonumber \\
\le & C \sum_{|r|\le\ell}a_\ell(r)^2 + C \sum_{z \in \Lambda_N}\kappa_\gamma(z)^2 +C\frac{\varepsilon_\gamma}{\gamma} \nonumber \\
\le & C(\ell^{-1}+\gamma+\varepsilon_\gamma/\gamma).
\end{align}
Plugging estimates \eqref{eq:lowfre} and \eqref{eq:highfre} in \eqref{eq:block-to-kac-fourier-bound-main}, we have 
\begin{equation}
    \mathbb E_{\mu_{N,\gamma,\beta}}\int_0^T \varepsilon_\gamma\sum_i \left|m_i^\ell(\sigma^N(t/\alpha_\gamma)) -h_\gamma(\sigma^N(t/\alpha_\gamma);i)\right|^2dt \le C(\ell^{-1}+\gamma+\varepsilon_\gamma/\gamma).
\end{equation}
Using same argument in Lemma~\ref{lem:weak-L2-h-from-microscopic}, we have 
\begin{align}
      &  \mathbb E_{f_{N,0}}\int_0^T
\varepsilon_\gamma\sum_i
\left|
 m_i^\ell(\sigma^N(t/\alpha_\gamma))
-h_\gamma(\sigma^N(t/\alpha_\gamma);i)
\right|^2dt \nonumber\\
& \qquad \le
C\frac{\alpha_\gamma}{\varepsilon_\gamma^{2}}\log(\varepsilon_\gamma^{-1})
+CT\left(\gamma+\frac1\ell+\frac{\varepsilon_\gamma}{\gamma}\right)
+CT \delta^2_{\gamma}.
\end{align}
Under the scaling assumptions, we have \eqref{eq:centered-block-to-kac-L2-rate-main}.

For the covariance estimate, Since
\(\kappa_\gamma\) is nonnegative and $\int_{\mathbb R} \kappa_\gamma dz = 1$, convolution by
\(\kappa_\gamma\) is a contraction on the discrete
\(L^2(\Lambda_N)\) space.  
Hence
\[
\varepsilon_\gamma\sum_j|W_{\gamma,j}^{\ell,\phi}|^2
\le
\varepsilon_\gamma\sum_i
\left|G_{\gamma,i}^{\phi}\right|^2
\left|m_i^\ell-h_{\gamma,i}\right|^2
\le
C_\phi
\varepsilon_\gamma\sum_i
\left|m_i^\ell-h_{\gamma,i}\right|^2.
\]
Combining with \eqref{eq:centered-block-to-kac-L2-rate-main} gives \eqref{eq:kac-covariance}.
\end{proof}

\begin{proposition}
\label{prop:deterministic-microscopic-to-kac-mobility-main}
Let \(J=(J_i)_{i\in\Lambda_N}\) be a deterministic bounded coefficient. 
Under Assumptions \(({\bf A})\)--\(({\bf C})\), and the block assumption \eqref{ass:clean-bg-scales}, we have
\begin{align}\label{eq:dkac-mobility}
    \mathbb E_{f_{N,0}} \left|\int_0^T \varepsilon_\gamma\sum_i
J_i\left[d_i\bigl(\sigma^N(t/\alpha_\gamma)\bigr)^2 - \Phi_{\Psi,L}^{\gamma}
\bigl(h_\gamma(\sigma^N(t/\alpha_\gamma);i)\bigr)\right]dt \right| \le C\|J\|_{\ell^\infty}\varepsilon^{1/8}_{\gamma}.
\end{align}
\end{proposition}

\begin{proof}
Insert and subtract the one-block projection:
\[
 d_i^2-
\Phi_{\Psi,L}^{\gamma}(h_{\gamma,i})
=
\left[d_i^2-\Phi_{\Psi,\ell}^{\gamma}(m_i^\ell)\right]
+
\left[
\Phi_{\Psi,\ell}^{\gamma}(m_i^\ell)
-
\Phi_{\Psi,L}^{\gamma}(h_{\gamma,i})
\right].
\]
By Theorem~\ref{thm:clean-one-block}, we have
\begin{equation}\label{eq:Kacmob1}
    \mathbb E_{f_{N,0}}\left|\int_0^T
\varepsilon_\gamma\sum_i
J_i\left[ d_i^2-
\Phi_{\Psi,\ell}^{\gamma}(m_i^\ell)\right]dt\right| \le C\left(\sqrt{T\alpha_\gamma\ell^3}
+T\gamma\ell^2\right).
\end{equation}
For the second bracket, Combining Theorem~\ref{thm:clean-one-block}, Lemma~\ref{lem:clean-explicit-dsquare}, and Lemma \ref{lem:clean-strong-kac-defect}, we have
\begin{align}\label{eq:Kacmob2}
&  \mathbb E_{f_{N,0}}\left|\int_0^T\varepsilon_\gamma\sum_i J_i\left[ \Phi_{\Psi,\ell}^{\gamma}(m_i^\ell)
-
\Phi_{\Psi,L}^{\gamma}(h_{\gamma,i})\right]dt\right| \nonumber \\
    \le & C\|J\|_{\ell^\infty}
\left[\mathbb E_{f_{N,0}}
\int_0^T \varepsilon_\gamma\sum_i|m_i^\ell-h_{\gamma,i}|dt +T\left(\frac1\ell+\frac1L+\gamma^2\ell^3+\gamma^2L^3\right)\right] \nonumber \\
\le & C\|J\|_{\ell^\infty}
\left[\mathbb E_{f_{N,0}}
\int_0^T \varepsilon_\gamma\sum_i|m_i^\ell-h_{\gamma,i}|^2dt \right]^{\frac{1}{2}}+T\left(\frac1\ell+\frac1L+\gamma^2\ell^3+\gamma^2L^3\right) \nonumber\\
\le & C \varepsilon^{1/8}_{\gamma}+T\left(\frac1\ell+\frac1L+\gamma^2\ell^3+\gamma^2L^3\right).
\end{align}
Combining with estimates \eqref{eq:Kacmob1} and \eqref{eq:Kacmob2}, we obtain \eqref{eq:dkac-mobility}.
\end{proof}

\begin{theorem}[Kac-scale covariance estimate]\label{thm:clean-kac-covariance}
Under Assumptions \(({\bf A})\)--\(({\bf C})\), and the block assumption \eqref{ass:clean-bg-scales}, for every smooth mean-zero \(\phi\), we have
\begin{align}\label{eq:caling and block}
    \mathbb E_{f_{N,0}}\int_0^T \varepsilon_\gamma \sum_j\left| \sum_i \kappa_\gamma(i-j)G_{\gamma,i}^{\phi} \left( d_i(t)^2- \Phi_{\Psi,L}^{\gamma}(h_{\gamma,i}(t)) \right)\right|^2dt  \le C_{\phi}T \varepsilon^{1/8}_{\gamma}.
\end{align}
\end{theorem}

\begin{proof}
We decompose
\[
d_i(t)^2- \Phi_{\Psi,L}^{\gamma}(h_{\gamma,i}(t))  =A_{\gamma,i}^{\ell}(t) + B_{\gamma,i}^{\ell,L}(t),
\]
where
\[
A_{\gamma,i}^{\ell}:=d_i^2-
\Phi_{\Psi,\ell}^{\gamma}(m_i^\ell),
\qquad
B_{\gamma,i}^{\ell,L}:=
\Phi_{\Psi,\ell}^{\gamma}(m_i^\ell)-
\Phi_{\Psi,L}^{\gamma}(h_{\gamma,i}).
\]
We first estimate the one-block part \(A^\ell\).  Expanding the square gives
\[
    \varepsilon_\gamma\sum_j
\left|\sum_i\kappa_\gamma(i-j)G_{\gamma,i}^{\phi}A_{\gamma,i}^{\ell}
\right|^2 =
\varepsilon_\gamma
\sum_{i,k} G_{\gamma,i}^{\phi}G_{\gamma,k}^{\phi}
A_{\gamma,i}^{\ell}A_{\gamma,k}^{\ell}
\mathcal K_\gamma^{(2)}(i,k),
\]
where
\(
\mathcal K_{\gamma}^{(2)}(i,k)
:=
\sum_j\kappa_\gamma(i-j)\kappa_\gamma(k-j),
\)
which satisfies \(\sup_{i,k}\mathcal K_{\gamma}^{(2)}(i,k)\le C\gamma\).
For each \(i \in \Lambda_N\) and some $R > 2\ell +3$, there are
at most \(2R+1\) indices \(k\) with \(|i-k|\le R\).
Since \(|A_{\gamma,i}^{\ell}|\le C\),
\(\|G_\gamma^\phi\|_{\ell^\infty}\le C_\phi\),
we have
\begin{equation}
    \mathbb E_{f_{N,0}}\int_0^T \varepsilon_\gamma
\sum_{\substack{i,k \in \Lambda_{N}\\ |i-k|\le R}}\left|G_{\gamma,i}^{\phi}G_{\gamma,k}^{\phi}A_{\gamma,i}^{\ell}A_{\gamma,k}^{\ell}\mathcal K_\gamma^{(2)}(i,k)
\right|dt \le C_\phi T\gamma R
\end{equation}
Now we estimate the far covariance is an exterior one-block estimate.
For the far part, if \(R\ge2\ell+3\), then the stochastic test function
\[
J^{s}_i(\sigma):=G_{\gamma,i}^{\phi}
\sum_{\substack{k\\ |i-k|>R}}
\mathcal K_\gamma^{(2)}(i,k)
G_{\gamma,k}^{\phi}
A_{\gamma,k}^{\ell}(\sigma).
\]
is measurable with respect to the exterior of the block \(\{i-\ell,\ldots,i+\ell\}\).
Moreover, since \(\sum_k\mathcal K_\gamma^{(2)}(i,k)=1\) and \(A^\ell\) is bounded, the stochastic test function $J^{s}_i$ is $\ell^{\infty}(\Lambda_N)$ bounded.
Then by Theorem~\ref{thm:clean-one-block}, we have
\begin{equation}\label{eq:EAL}
    \mathbb E_{f_{N,0}}
\left|
\int_0^T \varepsilon_{\gamma} \sum_{i \in \Lambda_N}
J^{s}_i(\sigma^N(t/\alpha_\gamma))
A_{\gamma,i}^{\ell}(\sigma^N(t/\alpha_\gamma))dt
\right| \le C_\phi\left(\sqrt{T\alpha_\gamma\ell^3}+T\gamma\ell^2\right).
\end{equation}
It remains to estimate \(B^{\ell,L}\).  
By Lemma~\ref{lem:clean-explicit-dsquare}, we have
\[
\Big|\Phi_{\Psi,\ell}^{\gamma}(m_i^\ell)-
\Phi_{\Psi,L}^{\gamma}(h_{\gamma,i}) \Big|
\le
C|m_i^\ell-h_{\gamma,i}| +C\left(\frac1\ell+\frac1L+\gamma^2\ell^3+\gamma^2L^3\right).
\]
Since convolution by \(\kappa_\gamma\) is an \(L^2\)-contraction and \(\|G_\gamma^\phi\|_{\ell^\infty}\le C_\phi\), then by Lemma \ref{lem:clean-strong-kac-defect},  
\begin{align}\label{eq:EBL}
    &\mathbb E_{f_{N,0}}\int_0^T \varepsilon_\gamma\sum_j
\left| \sum_i\kappa_\gamma(i-j)G_{\gamma,i}^{\phi}B_{\gamma,i}^{\ell,L}(t) \right|^2dt \nonumber \\
\le & 2 \mathbb E_{f_{N,0}}
\int_0^T \varepsilon_\gamma\sum_i
\left|\Phi_{\Psi,\ell}^{\gamma}(m_i^\ell)-
\Phi_{\Psi,L}^{\gamma}(h_{\gamma,i})
\right|^2dt  \nonumber \\
\le &  C_\phi   \mathbb E_{f_{N,0}}\int_0^T
\varepsilon_\gamma\sum_i
\left|
 m_i^\ell(\sigma^N(t/\alpha_\gamma)) -h_\gamma(\sigma^N(t/\alpha_\gamma);i)
\right|^2dt + C_\phi T\left(
\frac1\ell+\frac1L+\gamma^2\ell^3+\gamma^2L^3
\right)^2 \nonumber \\
\le & C_\phi \left[T \varepsilon^{1/4}_{\gamma} +T\left(
\frac1\ell+\frac1L+\gamma^2\ell^3+\gamma^2L^3
\right)^2\right].
\end{align}
Combining with \eqref{eq:EAL} and \eqref{eq:EBL}, we have
\begin{align*}
    &\mathbb E_{f_{N,0}}\int_0^T \varepsilon_\gamma \sum_j\left| \sum_i \kappa_\gamma(i-j)G_{\gamma,i}^{\phi} \left( d_i(t)^2- \Phi_{\Psi,L}^{\gamma}(h_{\gamma,i}(t)) \right)\right|^2dt \nonumber \\
&\qquad\le
C_\phi \left[ T\gamma R + \sqrt{T\alpha_\gamma\ell^3}
+T\gamma\ell^2 + T \varepsilon^{1/4}_{\gamma}
+T\left(\frac1\ell+\frac1L+\gamma^2\ell^3+\gamma^2L^3\right)^2
\right].
\end{align*}
With the scaling and block assumptions, we obtain \eqref{eq:caling and block}.
\end{proof}

\subsection{second-order Boltzmann--Gibbs replacement}
\label{subsec:three-step-bg-proof}

In this subsection, we show the second-order Boltzmann–Gibbs principle by a conservative multiscale replacement argument for the fast local observable $d^2_i$.
In order to avoid the divergence from $\nabla_{\varepsilon_{\gamma}} X_{\gamma}$ in the \(H^{-1}\)-tested nonlinear drift from $\mathcal D_\gamma^{\mathrm{nl}}$, we use the microscopic-bond representation \eqref{eq:clean-microscopic-bond-drift}.
The point is that the nonlinear drift does not contain an ordinary product of a local observable with a smooth macroscopic coefficient.  
Instead, the microscopic gradient \(d_j\) appears inside the Kac increment
\[
 \delta_{\gamma} \nabla_{\varepsilon_{\gamma}} X_{\gamma}(\varepsilon_{\gamma} i) = h_{\gamma,i+1}-h_{\gamma,i}
=-\sum_{j \in \Lambda_N}\kappa_\gamma(i-j)d_j .
\]
Thus the replacement has to be performed at the bond-gradient level.  

By local equilibrium principle for the local observable $d^2_i$ in Lemma \ref{lem:clean-explicit-dsquare}, we define
\[
P_{\gamma,L}(m):= 2 -2m^2.
\]
We then introduce the error term
\[
U_j(t,\sigma) := \sum_{i \in \Lambda_N}\kappa_\gamma(i-j)G_{\gamma,i}^{\phi} \left[d_i(\sigma)^2-P_{\gamma,L}(h_{\gamma,i}(\sigma))\right].
\]
in the second-order Boltzmann--Gibbs principle for
\(d_i^2\), seen from the bond \(j\) with the test coefficient \(G_{\gamma,i}^{\phi}\).
Proving that \(\sum_j d_jU_j\) is negligible is therefore the conservative form of the second-order Boltzmann--Gibbs principle needed for the nonlinear current.

Inspired by the polynomial decomposition method in \cite{G2017}, we decompose the Kac-averaged error \(U_j\) into its symmetric and antisymmetric
parts with respect to the exchange \(\tau_j\).  The symmetric part is controlled
by the Kipnis--Varadhan \(H^{-1}\) estimate. The antisymmetric part contains the genuine
exchange variation of the replacement error and is handled by explicit local cancellations and the Kac-gradient structure.
More precisely, with respect to the exchange \(\tau_j\sigma:=\sigma^{j,j+1}\), we decompose
\[
U_j(t,\sigma)=U_j^{\mathrm s}(t,\sigma)+U_j^{\mathrm a}(t,\sigma),
\]
where
\[
U_j^{\mathrm s}(t,\sigma)
:=
\frac12\left[U_j(t,\sigma)+U_j(t,\tau_j\sigma)\right],
\qquad
U_j^{\mathrm a}(t,\sigma)
:=
\frac12\left[U_j(t,\sigma)-U_j(t,\tau_j\sigma)\right].
\]
Under the exchange \(\tau_j\), the current variable changes sign, \(d_j(\tau_j\sigma)=-d_j(\sigma)\).  Hence \(d_jU_j^{\mathrm s}\) is antisymmetric with respect to \(\tau_j\) and can be estimated directly with the Kipnis--Varadhan estimate.  This gives the dynamic control of the centered fluctuation.  The remaining  antisymmetric part \(U_j^{\mathrm a}\) records how
the Boltzmann–Gibbs error itself changes under the exchange; it is not a generic error, but the part that must be handled by the local-diagonal cancellation and the
chain-rule exchange variation.

We first show that the variance of $U_j(t,\sigma)$ is negligible as $\gamma \to 0$.

\begin{lemma}\label{lem:clean-chain-rule-variance}
Under Assumption \(({\bf A})-({\bf C})\), and block assumption \eqref{ass:clean-bg-scales}, we have
\begin{equation}\label{eq:three-step-P-variance}
    \mathbb E_{f_{N,0}}
\int_0^T\varepsilon_\gamma \sum_{j \in \Lambda_N} \left| U_j(t,\sigma) \right|^2dt \le
C_\phi T \varepsilon^{1/8}_{\gamma}.
\end{equation}
\end{lemma}

\begin{proof}
We decompose
\[
d_i^2- P_{\gamma,L}(h_{\gamma,i})
=
\left[d_i^2-\Phi_{\Psi,L}^{\gamma}(h_{\gamma,i})\right]
+
\left[\Phi_{\Psi,L}^{\gamma}(h_{\gamma,i})-P_{\gamma,L}(h_{\gamma,i})\right].
\]
By Theorem~\ref{thm:clean-kac-covariance} and Lemma~\ref{lem:clean-explicit-dsquare}, we have
\begin{align}\label{eq:kace1}
   & \mathbb E_{f_{N,0}}\int_0^T\varepsilon_\gamma \sum_j \left| \sum_i
\kappa_\gamma(i-j)G_{\gamma,i}^{\phi} \left[d_i^2-\Phi_{\Psi,L}^{\gamma}(h_{\gamma,i})
+ \Phi_{\Psi,L}^{\gamma}(h_{\gamma,i})-P_{\gamma,L}(h_{\gamma,i}) \right]\right|^2dt\nonumber \\
  \le  & \mathbb E_{f_{N,0}}\int_0^T\varepsilon_\gamma \sum_j \left| \sum_i
\kappa_\gamma(i-j)G_{\gamma,i}^{\phi} \left[ d_i^2-\Phi_{\Psi,L}^{\gamma}(h_{\gamma,i}) \right]\right|^2dt \nonumber \\
 & \qquad + C\mathbb E_{f_{N,0}}\int_0^T\varepsilon_\gamma \sum_j \left| \sum_i
\kappa_\gamma(i-j)G_{\gamma,i}^{\phi} \left( L^{-1}+\gamma^2L^3 \right)\right|^2dt \nonumber \\
\le & C_\phi T \varepsilon^{1/8}_{\gamma}.
\end{align}
This proves the lemma.
\end{proof}

\begin{theorem}\label{thm:clean-microscopic-gradient}
Under Assumption \(({\bf A})-({\bf C})\), and block assumption \eqref{ass:clean-bg-scales}, we have
\begin{align}\label{eq:Esymm}
   \mathbb E_{f_{N,0}}
\left| \int_0^T
\frac{\varepsilon_\gamma^2}{\alpha_\gamma\delta_\gamma}
\sum_jd_j(t)U_j^{\mathrm s}(t,\sigma^N(t/\alpha_\gamma))dt
\right| \le C_{\phi,T}\varepsilon_\gamma^{1/16}.
\end{align}
\end{theorem}

\begin{proof}
We first show the centering of the observable 
\[
W(\sigma)=\sum_{j \in \Lambda_N} d_j(\sigma)U_j^{\mathrm{s}}(t,\sigma).
\]
For fixed relative distance \(r=i-j\), define
\[
C_{\gamma,r}(t)
:=
\int
d_j(\sigma)
\frac12
\left[
d_i(\sigma)^2- P_{\gamma,L}(h_{\gamma,i}(\sigma)) + d_i(\tau_j\sigma)^2- P_{\gamma,L}(h_{\gamma,i}(\tau_j \sigma)
\right]
d\mu_{N,\gamma,\beta}(\sigma).
\]
Notice that the canonical Gibbs measure is translation invariant, since both the Hamiltonian and the fixed-magnetization sector are invariant under translations.  
Moreover, if \(\theta_\ell\) denotes the spatial shift on the torus, then
\(d_j(\theta_\ell\sigma)=d_{j-\ell}(\sigma)\),
\(\theta_\ell\tau_j=\tau_{j-\ell}\theta_\ell\), and
\(
P_{\gamma,L}(h_{\gamma,i}(\theta_\ell\sigma))
= P_{\gamma,L}(h_{\gamma,i-\ell}(\sigma)),
\)
since the Kac field \(h_\gamma\) is
translation covariant.  Hence, for fixed \(\gamma\) and \(t\), the above
expectation is unchanged if the pair \((j,i)\) is translated simultaneously.
It therefore depends on the pair only through the relative displacement
\(r=i-j\), which justifies the notation \(C_{\gamma,r}(t)\).
Since \(G_{\gamma,i}^{\phi}\) is a periodic discrete gradient, we have
\[
\begin{aligned}
\mathbb{E}_{\mu_{N,\gamma,\beta}}W  = \int\sum_jd_j(\sigma)U_j^{\mathrm s}(t,\sigma)
d\mu_{N,\gamma,\beta}(\sigma) =
\left(\sum_r\kappa_\gamma(r)C_{\gamma,r}(t)\right)
\left(\sum_iG_{\gamma,i}^{\phi}\right)
=0 .
\end{aligned}
\]
Now we approximate the centered observable $W(\sigma)$ by a family of piecewise constant observables.
On a given partition \(0=t_0<t_1<\cdots<t_n=T\), we define the conditional expectations
\[
U_{j}^{\mathrm s,n}(t,\sigma)
:=
\frac1{|I_r|}
\int_{I_r}U_j^{\mathrm s}(s,\sigma)\,ds,
\qquad t\in I_r .
\]
Note that these approximate preserve the exchange symmetry
\(
U_j^{\mathrm s}(t,\sigma)=U_{j}^{\mathrm s,n}(\sigma), t\in[t_{r-1},t_r).
\)
Therefore, on each $I_r$, the observable
\(
W_n(\sigma)=\sum_j d_j(\sigma)U_j^{\mathrm{s,n}}(\sigma)
\) is a time-independent and centering.

Now we give the estimate for a time-independent centered observable $W_n(\sigma)$ on $I_r$.
By Lemma~\ref{lem:local-rates-elliptic}, for every test function \(F\) on the canonical sector, we have
\begin{align*}
\left| \int_{\Sigma_{N,M}} W_n(\sigma)F(\sigma)\,d\mu_{N,\gamma,\beta}(\sigma) \right| & \le C
\left(\sum_j\int |U_j^{\mathrm s}(\sigma)|^2d\mu_{N,\gamma,\beta}\right)^{1/2}
\left(D^{N,\gamma}(F)\right)^{1/2}\\
&\le  C \left( \sum_{j\in\Lambda_N} |U_j^{\mathrm s,n }(\sigma)|^2\right)^{1/2} 
\left( \sum_{j\in\Lambda_N} D^{N,\gamma}_{j}(F) \right)^{1/2} \\
&\le   D^{N,\gamma}(f) + C \sum_{j\in\Lambda_N} |U_j^{\mathrm s,n }(\sigma)|^2.
\end{align*}
Then by the variational definition of the negative norm \eqref{VCH-1}, we have
\[
  2\langle W_n,F\rangle_{\mu_{N,\gamma,\beta}}-D^{N,\gamma}(F) \le 2\bigl| \langle W,F\rangle_{\mu_{N,\gamma,\beta}} \bigr|-  D^{N,\gamma}(F) \le C\sum_{j\in\Lambda_N} |U_j^{\mathrm s,n }(\sigma)|^2.
\]
Recall the variational definition of the negative norm
\eqref{VCH-1}. After Optimizing over \(D^{N,\gamma}(F)^{1/2}\), we have
\begin{equation*}
    \|W_n\|_{\mathcal H_N^{-1}}^2
\le C\sum_j\int |U_j^{\mathrm{s,n}}(\sigma)|^2d\mu_{N,\gamma,\beta}(\sigma).
\end{equation*}
For each \(g \in L^2(\mu_{N,\gamma,\beta})\) with $\|g\|_{L^2(\mu_{N,\gamma,\beta})}=1$, since
\[
g(\sigma)^2-g(\tau_j\sigma)^2
=
\bigl(g(\sigma)-g(\tau_j\sigma)\bigr)
\bigl(g(\sigma)+g(\tau_j\sigma)\bigr),
\]
by using the reversibility of $\mu_{N,\gamma,\beta}$, we have
\[
\left|
\left\langle W_r,g^2\right\rangle_{\mu_{N,\gamma,\beta}}
\right|
\le
C
\left(\sum_j\int |U_{j,r}^{\mathrm s}(\sigma)|^2
d\mu_{N,\gamma,\beta}(\sigma)\right)^{1/2}
\left\langle g,(-\mathscr L_\gamma^N)g
\right\rangle_{\mu_{N,\gamma,\beta}}^{1/2}.
\]
Using the same argument as in the proof of Lemma~\ref{lem:entropy-KV}, by using the Markov property at the partition times and the Feynman--Kac variational formula, for one sign and any \(a>0\),
\[
\frac1a\log
\mathbb E_{\mu_{N,\gamma,\beta}}
\exp\left\{
a\sum_{r=1}^M
\int_{t_{r-1}}^{t_r}
W_n(\sigma^N(t/\alpha_\gamma))\,dt
\right\} \le aC \alpha_\gamma
\sum_{r=1}^M(t_r-t_{r-1})\|W_r\|_{\mathcal H_N^{-1}}^2,
\]
The same estimate holds for the negative sign.
Optimizing in \(a\), we have
\[
\mathbb{E}_{f_{N,0}} \left|
\int_0^T\sum_jd_j(\sigma^N(t/\alpha_\gamma))U_j^{\mathrm{s,n}}(t,\sigma^N(t/\alpha_\gamma))dt
\right|
\le
C\sqrt{\frac{\alpha_\gamma}{\varepsilon_\gamma}}
\left(
\mathbb E\int_0^T\sum_j|U_j^{\mathrm{s,n}}(t,\sigma^N(t/\alpha_\gamma))|^2dt
\right)^{1/2}.
\]
Since the state space is finite for fixed $N$ and the functions are uniformly bounded, as $n \to \infty$,
\[
U_j^{\mathrm s,n}\to U_j^{\mathrm s}
\quad\text{in }L^2\!\left([0,T]\times\mu_{N,\gamma,\beta}\right)
\]
for every \(j\).  The right-hand side above therefore converges by dominated
convergence, and the corresponding additive functionals converge in \(L^1\)
along the accelerated path. 
Finally, multiplication  by \(\varepsilon_\gamma^2/(\alpha_\gamma\delta_\gamma)\) gives the estimate 
\begin{align}\label{eq:scancellation}
    & \mathbb E_{f_{N,0}} \left|\int_0^T
\frac{\varepsilon_\gamma^2}{\alpha_\gamma\delta_\gamma}
\sum_j d_j(\sigma^N(t/\alpha_\gamma))U_j^{\mathrm{s}}(t,\sigma^N(t/\alpha_\gamma))dt
\right| \nonumber \\
\le & C \frac{\varepsilon_\gamma^{3/2}}
{\delta_\gamma\sqrt{\alpha_\gamma}}
\left( \mathbb E_{f_{N,0}}
\int_0^T \varepsilon_\gamma\sum_j
|U_j^{\mathrm{s}}(t,\sigma^N(t/\alpha_\gamma))|^2dt
\right)^{\frac{1}{2}}
\end{align}
It remains to bound the square average of \(U_j^{\mathrm{s}}\).  Since \( |U_j^{\mathrm{s}}|^2\le2|U_j|^2+2|U_j^{\mathrm{a}}|^2\),
Lemma~\ref{lem:clean-chain-rule-variance} implies
\begin{equation}
    \mathbb E_{f_{N,0}}\int_0^T \varepsilon_\gamma\sum_j |U_j(t,\sigma_t)|^2dt \le  C_{\phi} T \varepsilon^{1/8}_{\gamma}.
\end{equation}
For the second term, by Assumption \(({\bf A})\) and \(\tau_j\sigma:=\sigma^{j,j+1}\), we have
\begin{align}
    |U_j^{\mathrm a}| \le & \frac12 \left| \sum_{i \in  \{j +1,j-1\} }\kappa_\gamma(i-j)G_{\gamma,i}^{\phi}
\left[ d_i^2(\sigma)-d_i^2(\tau_j\sigma)
\right] \right| + \frac12 \left| \sum_{i \in \Lambda_N}\kappa_\gamma(i-j)G_{\gamma,i}^{\phi} \left[
P_{\gamma,L}(h_{\gamma,i}(\sigma)) - P_{\gamma,L}(h_{\gamma,i}(\tau_j \sigma)
\right] \right| \nonumber\\
\le & \frac12 C_{\phi}  \left| \sum_{i \in  \{j +1,j-1\} }\kappa_\gamma(i-j)\left[ d_i^2(\sigma)-d_i^2(\tau_j\sigma)
\right] \right|  \nonumber \\
 & \qquad  + C\sum_{i \in \Lambda_N}\kappa_\gamma(i-j)\bigl(|\kappa_\gamma(i-1-j)-\kappa_\gamma(i-j)|+ |\kappa_\gamma(i-j)-\kappa_\gamma(i+1-j)|\bigr) \nonumber \\
\le &  C_\phi\gamma,
\end{align}
uniformly in $\sigma, j$. Therefore
\begin{align}
    \mathbb E_{f_{N,0}}\int_0^T \varepsilon_\gamma\sum_j |U_j^{\mathrm s}(t,\sigma_t)|^2dt \le & 2\mathbb E_{f_{N,0}}\int_0^T \varepsilon_\gamma\sum_j |U_j^{\mathrm a}(t,\sigma_t)|^2dt + 2\mathbb E_{f_{N,0}}\int_0^T \varepsilon_\gamma\sum_j |U_j(t,\sigma_t)|^2dt \nonumber \\
    \le & C_\phi T (\gamma + \varepsilon^{1/8}_{\gamma}).
\end{align}
Plugging above estimate in \eqref{eq:scancellation} and using
\(\varepsilon_\gamma^{3/2}/(\delta_\gamma\sqrt{\alpha_\gamma})=O(1)\), we obtain \eqref{eq:Esymm}.
\end{proof}

The antisymmetric part $U_j^{\mathrm a}(t,\sigma) $ has two sources: the local variation of \(d_i^2\), handled by the signed odd and local-diagonal cancellations, and the exchange variation
of \( P_{\gamma,L}(h_{\gamma,i}(\sigma))\), handled in Lemma \ref{lem:clean-chain-rule-exchange-variation}.
Based on one block estimates, we first show the cancellation for the signed odd local observables $d_j^2\sigma_{j+2}, d_j^2\sigma_{j-1}$. It is enough to consider $d_0^2\sigma_2, d_0^2\sigma_{-1}$.

\begin{lemma}\label{lem:clean-signed-odd}
Let \(H_{\gamma,j}\) be a uniformly bounded and deterministic function.
Define
\[
\Phi_{+,\ell}^{\gamma}(m)
:=
\mathbb E_{\mu_{\gamma,\ell,m}^{\mathrm{can}}}[d_0^2\sigma_2],
\quad
\Phi_{-,\ell}^{\gamma}(m) := \mathbb E_{\mu_{\gamma,\ell,m}^{\mathrm{can}}}[d_0^2\sigma_{-1}].
\]
Then under Assumptions \(({\bf A})\)--\(({\bf C})\), and the block assumption \eqref{ass:clean-bg-scales},
\[
\lim_{\gamma \to 0}\mathbb E_{f_{N,0}}
\left|\int_0^T \varepsilon_\gamma\sum_j
H_{\gamma,j}\,\Phi_{\varsigma,\ell}^{\gamma}
\bigl(m_j^\ell(\sigma^N(t/\alpha_\gamma))\bigr)dt \right| \le C_{\phi,T}\varepsilon^{1/8}_{\gamma}  \to 0, \quad \varsigma\in\{+,-\}.
\]
\end{lemma}

\begin{proof}
Let \(\nu_{\ell,m}^{\mathrm{aux}}\) be the uniform law on the block canonical
sector.  
As in Lemma~\ref{lem:clean-explicit-dsquare}, by estimate \eqref{eq:oneOscH}, the Kac interaction restricted to a block $\Lambda_L$ has oscillation bounded by \(C\gamma^2L^3\) on each canonical sector. Thus
\[
\left\|
\mu_{\gamma,L,m}^{\mathrm{can}}-\nu_{L,m}^{\mathrm{aux}}
\right\|_{\mathrm{TV}}
\le
C_\beta\,\gamma^2L^3.
\]
uniformly in $m$.
Since $|d_0^2\sigma_2|+|d_0^2\sigma_{-1}| \le 8$, it follows that
\[
\sup_{|m|\le1}
\left|
\Phi_{\varsigma,\ell}^{\gamma}(m)
-\Phi_{\varsigma,\ell}^{\mathrm{aux}}(m)
\right|
\le C\gamma^2\ell^3,
\]
Under \(\nu_{\ell,m}^{\mathrm{aux}}\), the spins in the block canonical sector are exchangeable. Since
\[
d_0^2\sigma_2=2\sigma_2-2\sigma_0\sigma_1\sigma_2,
\qquad
d_0^2\sigma_{-1}=2\sigma_{-1}-2\sigma_0\sigma_1\sigma_{-1},
\]
each auxiliary expectation is an odd polynomial of the block magnetization
\(m\) of degree at most three.  Hence
\[
\Phi_{\varsigma,\ell}^{\mathrm{aux}}(m)
=b_{1,\ell}m+b_{3,\ell}m^3,
\qquad
\sup_\ell(|b_{1,\ell}|+|b_{3,\ell}|)<\infty .
\]
Since \(|u^p-v^p|\le C|u-v|\) for \(u,v\in[-1,1]\) and \(p=1,3\), Cauchy--Schwarz inequality and the  estimate \eqref{eq:centered-block-to-kac-L2-rate-main} give
\begin{equation}\label{eq:conloc0}
    \mathbb E_{f_{N,0}}\int_0^T \varepsilon_\gamma\sum_j
|H_{\gamma,j}|
\left| \bigl(m_j^\ell(\sigma^N(t/\alpha_\gamma))\bigr)^p -h_{\gamma,j}(t)^p
\right|dt \le C\|H_\gamma\|_{\ell^\infty}
\sqrt{T} \varepsilon^{1/8}_{\gamma}.
\end{equation}
For the weak averages of \(h_{\gamma,j}\), by Lemma~\ref{lem:weak-L2-h-from-microscopic}, we have
\begin{align}\label{eq:conloc1}
\mathbb E_{f_{N,0}}
\left|\int_0^T\varepsilon_\gamma\sum_jH_{\gamma,j}h_{\gamma,j}dt\right|
\le & \|H_\gamma\|_{\ell^\infty}
\mathbb E_{f_{N,0}}
\int_0^T \left(\varepsilon_\gamma\sum_j(h_{\gamma,j}(t)-m_N)^2\right)^{1/2}dt
+T\|H_\gamma\|_{\ell^\infty}|m_N| \nonumber \\
\le & C\left( \delta_\gamma^2\log(\varepsilon_\gamma^{-1}) \right)^{1/2} 
\rightarrow0.
\end{align}
For the cubic term, since \(|h_{\gamma,j}|\le1\) and
\(|h_{\gamma,j}-m_N|\le2\), we have
\(
|h_{\gamma,j}|^3
\le
C|h_{\gamma,j}-m_N|^2+C|m_N|.
\)
Then using Lemma~\ref{lem:weak-L2-h-from-microscopic} and estimate \eqref{eq:mass-sector-implies-mN-small} again, we have
\begin{equation}\label{eq:conloc2}
    \mathbb E_{f_{N,0}}
\int_0^T\varepsilon_\gamma\sum_j|H_{\gamma,j}||h_{\gamma,j}|^3dt
\le
C \delta_\gamma^2\log(\varepsilon_\gamma^{-1})\rightarrow0.
\end{equation}
Combining \eqref{eq:conloc0}, \eqref{eq:conloc1}, and \eqref{eq:conloc2} proves the lemma.
\end{proof}

Based on the above cancellation, we estimate the local-diagonal variation of \(d_i^2\).

\begin{lemma}\label{lem:clean-local-diagonal}
Then under Assumptions \(({\bf A})-({\bf C})\) and the block assumptions \eqref{ass:clean-bg-scales}, we have
\[
\mathbb E_{f_{N,0}}
\left|
\int_0^T \frac{\varepsilon_\gamma^2}{\alpha_\gamma\delta_\gamma}
\sum_{j\in \Lambda_N}d_j(t)\frac12
\sum_{i=j-1,j+1}
\kappa_\gamma(i-j)G_{\gamma,i}^{\phi}
\left[
d^2_i(\sigma^N(t/\alpha_\gamma))-d^2_i\bigl(\tau_j(\sigma^N(t/\alpha_\gamma))\bigr)
\right] dt
\right| \le C_{\phi,T}\varepsilon^{1/8}_{\gamma}.
\]
\end{lemma}

\begin{proof}
Note that when a local bond is changed by the exchange \(\tau_j\sigma:= \sigma^{j,j+1}\), its square is unchanged.  
Thus the only affected neighboring squares are
\begin{equation}\label{eq:diodd}
    d_{j-1}^2(\sigma)- d_{j-1}^2(\tau_j\sigma)
=-2\sigma_{j-1}d_j,
\qquad
d_{j+1}^2(\sigma)-d_{j+1}^2(\tau_j\sigma)
=2\sigma_{j+2}d_j
\end{equation}
and \(d_i^2(\sigma) - d_i^2(\tau_j\sigma) = 0 \) for $i \neq j\pm1$.
Then using \(\kappa_\gamma(1)=\kappa_\gamma(-1)\), we have
\begin{align}\label{eq:aIDE0}
     & \frac{\varepsilon_\gamma^2}{\alpha_\gamma\delta_\gamma}
\sum_{j\in \Lambda_N}d_j(t)\frac12 \sum_{i\in \{j-1,j+1\} }
\kappa_\gamma(i-j)G_{\gamma,i}^{\phi}\left[d_i(t)^2-d_i\bigl(\tau_j(\sigma^N(t/\alpha_\gamma))\bigr)^2
\right] dt \nonumber\\
= & \frac{\varepsilon_\gamma^2\kappa_\gamma(1)}{\alpha_\gamma\delta_\gamma}
\sum_{j\in \Lambda_N} d_j^2
\left[ -G_{\gamma,j-1}^{\phi}\sigma_{j-1} +G_{\gamma,j+1}^{\phi}\sigma_{j+2} \right] \nonumber\\
 = & \frac{\varepsilon_\gamma^2\kappa_\gamma(1)}{\alpha_\gamma\delta_\gamma}\sum_{j\in \Lambda_N} d_j^2
 \left[\big(G_{\gamma,j}^{\phi} -G_{\gamma,j-1}^{\phi}\big)\sigma_{j-1} + \big(G_{\gamma,j+1}^{\phi} - G_{\gamma,j}^{\phi} \big)\sigma_{j+2} \right] + \frac{\varepsilon_\gamma^2\kappa_\gamma(1)}{\alpha_\gamma\delta_\gamma}\sum_{j\in \Lambda_N}  G_{\gamma,j}^{\phi}d_j^2(\sigma_{j+2}-\sigma_{j-1}).
\end{align}
We estimate the first term by one block estimates and Lemma~\ref{lem:clean-signed-odd}. Denote
\[
H_{\gamma,j}^{+,\phi}:=
\varepsilon_\gamma^{-1}(G_{\gamma,j+1}^{\phi}-G_{\gamma,j}^{\phi}),
\qquad
H_{\gamma,j}^{-,\phi}:=
\varepsilon_\gamma^{-1}(G_{\gamma,j}^{\phi}-G_{\gamma,j-1}^{\phi}).
\]
Since \(G_\gamma^\phi\) is the discrete gradient of \(\Delta_{\varepsilon_\gamma}^{-1}\phi\), the coefficients \(H_{\gamma,j}^{\pm,\phi}\) are uniformly bounded in \(j,\gamma\). 
Under Assumption \(({\bf B})\), and  \eqref{ass:clean-bg-scales}, the prefactor \(\varepsilon_\gamma^2\kappa_\gamma(1)/(\alpha_\gamma\delta_\gamma)\) is bounded.
By Theorem~\ref{thm:clean-one-block} and Lemma~\ref{lem:clean-signed-odd} with two local observables \(d_0^2\sigma_2\) and \(d_0^2\sigma_{-1}\), the difference part 
\begin{align}\label{eq:aIDE2}
  & \frac{\varepsilon_\gamma^2\kappa_\gamma(1)}{\alpha_\gamma\delta_\gamma}
\int_0^T
\varepsilon_\gamma\sum_j
\left[
H_{\gamma,j}^{+,\phi}d_j^2\sigma_{j+2}
+H_{\gamma,j}^{-,\phi}d_j^2\sigma_{j-1}
\right]dt \nonumber \\
\le &  \frac{\varepsilon_\gamma^2\kappa_\gamma(1)}{\alpha_\gamma\delta_\gamma}
\mathbb E_{f_{N,0}}\left|\int_0^T \varepsilon_\gamma\sum_j
H^+_{\gamma,j}\,\Phi_{+,\ell}^{\gamma}\bigl(m_j^\ell(\sigma^N(t/\alpha_\gamma))\bigr)dt
+ \int_0^T \varepsilon_\gamma\sum_j
H^-_{\gamma,j}\,\Phi_{-,\ell}^{\gamma}\bigl(m_j^\ell(\sigma^N(t/\alpha_\gamma))\bigr)dt \right| \nonumber\\
& \qquad + C_\phi  \frac{\varepsilon_\gamma^2\kappa_\gamma(1)}{\alpha_\gamma\delta_\gamma} \left( \sqrt{T\alpha_\gamma\ell^3} +T\gamma\ell^2 \right). \nonumber \\
\le & C_{\phi,T}\frac{\varepsilon_\gamma^2\kappa_\gamma(1)}{\alpha_\gamma\delta_\gamma}\varepsilon^{1/8}_{\gamma} \to 0.
\end{align}
For the second term, we set the local observable
\[
B_j(\sigma):=d_j(\sigma)^2(\sigma_{j+2}-\sigma_{j-1}),
\qquad
W_G(\sigma):=\varepsilon_\gamma\sum_jG_{\gamma,j}^{\phi}B_j(\sigma).
\]
Since \(\mu_{N,\gamma,\beta}\) is translation invariant, the integral $\int B_j(\sigma)\,d\mu_{N,\gamma,\beta}(\sigma)$ is independent of \(j\).  
Since \(G_{\gamma}^{\phi}\) is a periodic discrete
gradient, \(\sum_jG_{\gamma,j}^{\phi}=0\). Therefore \(W_G\) is a centered observable.
We next prove the fixed-range \(\mathcal H^{-1}\) bound.
The local observable \(B_j\) has zero average on each such sector: reflection of the four-site block around the exchanges \(j-1\leftrightarrow j+2\) and
\(j\leftrightarrow j+1\), keeps \(d_j^2\) fixed, and changes
\(\sigma_{j+2}-\sigma_{j-1}\) to its negative.  Hence \(B_j\) lies in the
range of the finite graph divergence.  Equivalently, there are bounded local functions \(A_{j,u}\), \(u=j-1,j,j+1\), such that
\[
B_j(\sigma)
=
\sum_{u=j-1}^{j+1}
\left[
A_{j,u}(\sigma^{u,u+1})-A_{j,u}(\sigma)
\right],
\qquad
\sup_{j,u,\sigma}|A_{j,u}(\sigma)|\le C .
\]
The bound is uniform because there are only finitely many four-site sectors, and \(A_{j,u}\) are translates of the corresponding finite list of solutions.

For a test function \(F\) on the canonical sector $\sum_{n,M}$, we have
\[
\left|
\int W_G(\sigma)F(\sigma)\,d\mu_{N,\gamma,\beta}(\sigma)
\right|
\le
C\varepsilon_\gamma
\sum_j |G_{\gamma,j}^{\phi}|
\sum_{u=j-1}^{j+1}
\left(D_u^{N,\gamma}(F)\right)^{1/2}.
\]
Since each global bond is counted only a bounded number of times, by Cauchy--Schwarz inequality and \(\varepsilon_\gamma\sum_j|G_{\gamma,j}^{\phi}|^2\le C_\phi\), we have
\[
\left|
\int W_G F\,d\mu_{N,\gamma,\beta}
\right| \le C_\phi\sqrt{\varepsilon_\gamma}
\left(D^{N,\gamma}(F)\right)^{1/2},
\]
Applying the \(\mathcal H^{-1}\) estimate used in the proof of Theorem~\ref{thm:clean-microscopic-gradient}, we get
\(\|W_G\|_{\mathcal H_N^{-1}}^2\le C_\phi\varepsilon_\gamma \).
In particular, taking \(F=g^2\) with $\|g\|_{L^2(\mu_{N,\gamma,\beta})}=1$, and using
\[
g(\sigma)^2-g(\sigma^{u,u+1})^2
=
\bigl(g(\sigma)-g(\sigma^{u,u+1})\bigr)
\bigl(g(\sigma)+g(\sigma^{u,u+1})\bigr),
\]
we obtain
\[
\left|
\left\langle W_G,g^2\right\rangle_{\mu_{N,\gamma,\beta}}
\right|
\le
C_\phi\sqrt{\varepsilon_\gamma}
\left\langle g,(-\mathscr L_\gamma^N)g
\right\rangle_{\mu_{N,\gamma,\beta}}^{1/2}.
\]
Then using the entropy--Kipnis--Varadhan estimate Lemma~\ref{lem:entropy-KV}, we have
\begin{equation}\label{eq:aIDE1}
    \mathbb E_{f_{N,0}}
\left|
\int_0^T
\frac{\varepsilon_\gamma^2\kappa_\gamma(1)}{\alpha_\gamma\delta_\gamma}
\sum_j
G_{\gamma,j}^{\phi}d_j^2(t)(\sigma_{j+2}^N(t/\alpha_\gamma)-\sigma_{j-1}^N(t/\alpha_\gamma))\,dt
\right|
\le C_{\phi,T}
\frac{\varepsilon_\gamma\kappa_\gamma(1)}
{\delta_\gamma\sqrt{\alpha_\gamma}},
\end{equation}
which tends to $0$ by Assumption \(({\bf B})\) and
\(\kappa_\gamma(1)\le C\gamma\) from Assumption \(({\bf A})\).
Plugging estimates \eqref{eq:aIDE1} and \eqref{eq:aIDE2} into \eqref{eq:aIDE1}, we prove the lemma.
\end{proof}

\begin{lemma}\label{lem:clean-chain-rule-exchange-variation}
Under Assumptions \(({\bf A})-({\bf C})\)
and the block assumptions \eqref{ass:clean-bg-scales}, we have
\begin{align}
    & \mathbb E_{f_{N,0}}
\left| \frac{\varepsilon_\gamma^2}{\alpha_\gamma\delta_\gamma}
\sum_j d_j(t)\frac12
\sum_i
\kappa_\gamma(i-j)G_{\gamma,i}^{\phi}
\left[ P_{\gamma,L}(h_{\gamma,i}\Big(\bigl(\tau_j(\sigma^N(t/\alpha_\gamma))\bigr)\Big) - P_{\gamma,L}\Big(h_{\gamma,i}(\bigl(\sigma^N(t/\alpha_\gamma)\bigr)\Big)
\right]\,dt \right| \nonumber \\
\le & C_\phi \varepsilon^{1/16}_{\gamma}\left(\log(\varepsilon^{-1}_{\gamma})\right)^{1/2}.
\end{align}
\end{lemma}

\begin{proof}
By the exchange increment identity \eqref{Ex:h}, the exchange \(\tau_j\) changes the Kac field $h_{\gamma,k}$ by
\[
h_{\gamma,k}(\tau_j\sigma)-h_{\gamma,k}(\sigma)
=
d_j(\sigma)
\left[
\kappa_\gamma(k-j-1)-\kappa_\gamma(k-j)
\right].
\]
Set
\(
\eta_\gamma(z):=\kappa_\gamma(z-1)-\kappa_\gamma(z).
\)
We decompose
\begin{align*}
    P_{\gamma,L}(h_{\gamma,i}(\tau_j\sigma)) -P_{\gamma,L}(h_{\gamma,i}(\sigma)) = &  -4h_{\gamma,i}(\sigma)d_j(\sigma)\eta_\gamma(i-j) -2d_j(\sigma)^2\eta_\gamma(i-j)^2  \\
    = & d_j(\sigma)B_{\gamma,i,j}(\sigma) + E_{\gamma,i,j}(\sigma),
\end{align*}
where
\[
 B_{\gamma,i,j}(\sigma) := 4h_{\gamma,i}(\sigma)\eta_\gamma(i-j), \qquad E_{\gamma,i,j}(\sigma) := -2d_j(\sigma)^2\eta_\gamma(i-j)^2.
\]
For the linear term we split $P_{\gamma,L}(h_{\gamma,j}) = \left(d_j^2- P_{\gamma,L}(h_{\gamma,j})\right) + P_{\gamma,L}(h_{\gamma,j}) $.
We first control the centered part
\[
\sum_j \left( d_j^2-P_{\gamma,L}(h_{\gamma,j}) \right)
\sum_i\kappa_\gamma(i-j)G_{\gamma,i}^{\phi}
h_{\gamma,i}\eta_\gamma(i-j)
=
\sum_iG_{\gamma,i}^{\phi}h_{\gamma,i}
\sum_jK_\gamma(i-j) \left( d_j^2-P_{\gamma,L}(h_{\gamma,j}) \right).
\]
where $K_\gamma(r):=\kappa_\gamma(r)\eta_\gamma(r)$.
Then by Assumption \(({\bf A})\)
\[
\sum_{r \in \Lambda_N}| K_\gamma(r)|\le C\gamma^2,\qquad
\|K_\gamma\|_{\ell^\infty}\le C\gamma^3,\qquad
\sup_{a,b \in \Lambda_N}\sum_{i \in \Lambda_N}| K_\gamma(i-a) K_\gamma(i-b)|\le C\gamma^3 .
\]
Repeating the covariance expansion of Theorem~\ref{thm:clean-kac-covariance} with \(K_\gamma\) gives
\[
\mathbb E_{f_{N,0}}
\int_0^T
\varepsilon_\gamma\sum_i
\left|
\sum_j \gamma^{-2}K_\gamma(i-j)
\bigl[d_j(t)^2-\Phi_{\Psi,L}^{\gamma}(h_{\gamma,j}(t))\bigr]
\right|^2dt
\le C_{\phi,T}\varepsilon_\gamma^{1/8}.
\]
Moreover, by Lemma \ref{lem:clean-explicit-dsquare}, we have
\[
\sup_m|\Phi_{\Psi,L}^{\gamma}(m)-P_{\gamma,L}(m)|
\le C(L^{-1}+\gamma^2L^3),
\]
Thus
\begin{equation}\label{eq:differentiated-centered-kac-bound}
\mathbb E_{f_{N,0}}
\int_0^T
\varepsilon_\gamma\sum_i
\left|
\sum_jK_\gamma(i-j)\bigl(d_j(t)^2-P_{\gamma,L}(h_{\gamma,j}(t))\bigr)
\right|^2dt
\le C_{\phi,T}\gamma^4\varepsilon_\gamma^{1/8}.
\end{equation}
Therefore, by Lemma~\ref{lem:mean-square-kac-gradient} and Lemma~\ref{lem:weak-L2-h-from-microscopic},
\begin{align}\label{eq:eakerf0}
&
\mathbb E_{f_{N,0}}
\left|
\int_0^T
\frac{\varepsilon_\gamma^2}{\alpha_\gamma\delta_\gamma}
\sum_j\bigl(d_j(t)^2-P_{\gamma,L}(h_{\gamma,j}(t))\bigr)
\sum_i\kappa_\gamma(i-j)G_{\gamma,i}^{\phi}
h_{\gamma,i}(t)\eta_\gamma(i-j)\,dt
\right| \nonumber \\
&\qquad\le
C_\phi
\frac{\varepsilon_\gamma}{\alpha_\gamma\delta_\gamma}
\left(
\mathbb E_{f_{N,0}}
\int_0^T\varepsilon_\gamma\sum_i h_{\gamma,i}(t)^2dt
\right)^{1/2}
\left(
\mathbb E_{f_{N,0}}
\int_0^T\varepsilon_\gamma\sum_i
\left|\sum_jK_\gamma(i-j)B^P_{\gamma,j}(t)\right|^2dt
\right)^{1/2} \nonumber \\
&\qquad\le
C_{\phi,T}
\frac{\varepsilon_\gamma\gamma^2\varepsilon_\gamma^{1/16}}
{\alpha_\gamma}
\left(\log(\varepsilon_\gamma^{-1})\right)^{1/2}.
\end{align}
Now we estimate the smooth part 
\[
4\sum_{j \in \Lambda_N}  P_{\gamma,L}(h_{\gamma,j})\sum_{i \in \Lambda_N} \kappa_\gamma(i-j)G_{\gamma,i}^{\phi}h_{\gamma,i}(\sigma)\eta_\gamma(i-j) = 4\sum_{j \in \Lambda_N}  P_{\gamma,L}(h_{\gamma,j})\sum_{i \in \Lambda_N}G_{\gamma,i}^{\phi}h_{\gamma,i}(\sigma) K_\gamma(i-j).
\]
Since $\sum_{r \in \Lambda_N}| K_\gamma(r)|\le C\gamma^2$ and \(\|G_\gamma^\phi\|_{\ell^\infty}\le C_\phi\), we have
\[
\left| \varepsilon_\gamma\sum_j
\sum_i \kappa_\gamma(i-j) \mathfrak A_{\gamma,j}(\sigma) G_{\gamma,i}^{\phi}B_{\gamma,i,j}(\sigma)
\right| 
\le C_\phi\gamma^3 + C_\phi\gamma^2\,
\varepsilon_\gamma\sum_j
\bigl(h_{\gamma,j}^2+h_{\gamma,j+1}^2\bigr).
\]
Using Lemma~\ref{lem:weak-L2-h-from-microscopic} and $|h_{\gamma,j}| \le 1$, we obtain
\begin{equation}\label{eq:eakerf}
     \mathbb E_{f_{N,0}} \left| \int_0^T \frac{\varepsilon_\gamma^2}{2\alpha_\gamma\delta_\gamma}
\sum_{j \in \Lambda_N}  P_{\gamma,L}(h_{\gamma,j})\sum_{i \in \Lambda_N} \kappa_\gamma(i-j)G_{\gamma,i}^{\phi}\eta_\gamma(i-j)h_{\gamma,i}(\sigma^N(t/\alpha_\gamma))\,dt \right|
 \le C_{\phi,T} \varepsilon_\gamma^{1/8}\left( \log(\varepsilon_\gamma^{-1}) \right)^{\frac{1}{2}}.
\end{equation}
By \( \sum_i
\kappa_\gamma(i-j)
\left(\eta_\gamma(i-j)
\right)^2
\le C\gamma^4 \), we have
\begin{align}\label{eq:assE0}
    \int_0^T \frac{\varepsilon_\gamma^2}{\alpha_\gamma\delta_\gamma}
\sum_j \sum_i \kappa_\gamma(i-j)|G_{\gamma,i}^{\phi}|
|E_{\gamma,i,j}|\,dt \le & \int_0^T \frac{\varepsilon_\gamma^2}{\alpha_\gamma\delta_\gamma}
\sum_j \sum_i \kappa_\gamma(i-j)|G_{\gamma,i}^{\phi}|
C\left(\eta_\gamma(i-j)\right)^2 \,dt \nonumber \\
\le &
C_{\phi,T}
\frac{\varepsilon_\gamma\gamma^4}{\alpha_\gamma\delta_\gamma}.
\end{align}
Combining the three estimates \eqref{eq:eakerf}, \eqref{eq:eakerf0}, and \eqref{eq:assE0} proves the lemma.
\end{proof}

Based on the microscopic-gradient-cancellation in Theorem~\ref{thm:clean-microscopic-gradient},  Lemma~\ref{lem:clean-local-diagonal}, and Lemma~\ref{lem:clean-chain-rule-exchange-variation}, now we show the second-order Boltzmann–Gibbs principle for the microscopic-gradient.

\begin{theorem}\label{cor:clean-nonlinear-drift}
Under Assumptions \(({\bf A})\)--\(({\bf C})\), and the block assumption \eqref{ass:clean-bg-scales}, we have
\begin{equation}
    \int_0^T
\mathcal D_\gamma^{\mathrm{nl}}(\phi;\sigma^N(t/\alpha_\gamma))dt =  -\frac{\beta }{2}
\frac{\varepsilon_\gamma^2}{\alpha_\gamma}
\int_0^T
\left\langle
X_\gamma(t),
\mathcal K_\gamma^\varepsilon\phi
\right\rangle_\gamma dt 
+ \frac{\beta}{6}
\frac{\varepsilon_\gamma^2\delta_\gamma^2}{\alpha_\gamma}
\int_0^T
\left\langle
X_\gamma(t)^3,
\mathcal K_\gamma^\varepsilon\phi
\right\rangle_\gamma dt + r_{\gamma,\mathrm{BG}}^\phi(T),
\end{equation}
where the remainder satisfies
\(
\mathbb E_{f_{N,0}}
\bigl|r_{\gamma,\mathrm{BG}}^\phi(T)\bigr|
\longrightarrow 0\) as \(\gamma\downarrow0.\)
\end{theorem}

\begin{proof}
Since $U_j(t,\sigma) = U_j^{\mathrm s}(t,\sigma)+U_j^{\mathrm a}(t,\sigma)$, the target integral is the sum of the symmetric contribution with
\(U_j^{\mathrm s}\) and the antisymmetric contribution with \(U_j^{\mathrm a}\).

By Theorem~\ref{thm:clean-microscopic-gradient}, the symmetric part satisfies
\begin{align}
   \mathbb E_{f_{N,0}}
\left| \int_0^T \frac{\varepsilon_\gamma^2}{\alpha_\gamma\delta_\gamma}
\sum_jd_j(t)U_j^{\mathrm s}(t,\sigma^N(t/\alpha_\gamma))dt
\right| \le C_{\phi,T}\varepsilon_\gamma^{1/8}.
\end{align}
From the definition of \(U_j^{\mathrm a}\), the antisymmetric part
\[
\begin{aligned}
d_jU_j^{\mathrm a}
={}&
d_j\frac12\sum_i\kappa_\gamma(i-j)G_{\gamma,i}^{\phi}
\left[d_i^2(\sigma)-d_i^2(\tau_j\sigma)\right]
+
d_j\frac12\sum_i\kappa_\gamma(i-j)G_{\gamma,i}^{\phi}
\left[
\mathfrak A_{\gamma,i}(\tau_j\sigma)-\mathfrak A_{\gamma,i}(\sigma)
\right].
\end{aligned}
\]
The first term is the antisymmetric part coming from \(d_i^2\).
Note that only bonds whose square changes under \(\tau_j\) are
\(i=j-1\) and \(i=j+1\)). Therefore
\begin{equation}
    \frac{\varepsilon_\gamma^2}{\alpha_\gamma\delta_\gamma}
\sum_jd_j(t)
\frac12 \sum_i\kappa_\gamma(i-j)G_{\gamma,i}^{\phi}
\left[
d_i(t)^2-d_i\bigl(\tau_j(\sigma^N(t/\alpha_\gamma))\bigr)^2
\right] = \mathscr A_{\gamma}^{\mathrm{ld},\phi}(t) .
\end{equation}
The remaining antisymmetric term is $\mathscr L_{\gamma}^{\mathrm{cr},\phi}(t)$.
Thus by Lemma~\ref{lem:clean-local-diagonal} and  Lemma~\ref{lem:clean-chain-rule-exchange-variation}, 
\begin{align}
     \mathbb E_{f_{N,0}} \left| \int_0^T \frac{\varepsilon_\gamma^2}{\alpha_\gamma\delta_\gamma} \sum_jd_j(t)U^{\mathrm a}_j(t,\sigma^N(t/\alpha_\gamma))dt \right| \le  C_\phi \varepsilon^{1/16}_{\gamma}\left(\log(\varepsilon^{-1}_{\gamma})\right)^{1/2}
\end{align}
Combining the symmetric estimate and antisymmetric estimate, we obtain 
\[
\mathbb E_{f_{N,0}}
\left|
\int_0^T
\frac{\varepsilon_\gamma^2}{\alpha_\gamma\delta_\gamma}
\sum_j d_j(t)U_j(t,\sigma^N(t/\alpha_\gamma))\,dt
\right| \le C_{\phi,T} \varepsilon^{1/16}_{\gamma}\left(\log(\varepsilon^{-1}_{\gamma})\right)^{1/2}.
\]
Consider the discrete chain-rule coefficient
\[
\mathfrak A_{\gamma,i}(t)
:=
\begin{cases}
\dfrac{
Q_{\gamma,L}(h_{\gamma,i+1}(t))
-
Q_{\gamma,L}(h_{\gamma,i}(t))
}{
h_{\gamma,i+1}(t)-h_{\gamma,i}(t)
},
& h_{\gamma,i+1}(t)\neq h_{\gamma,i}(t),
\\[1.2em]
P(h_{\gamma,i}(t)),
& h_{\gamma,i+1}(t)=h_{\gamma,i}(t),
\end{cases}
\]
so that
\[
\mathfrak A_{\gamma,i}(t)
\bigl(h_{\gamma,i+1}(t)-h_{\gamma,i}(t)\bigr)
=
Q_{\gamma,L}(h_{\gamma,i+1}(t))
-
Q_{\gamma,L}(h_{\gamma,i}(t)).
\]
Since \(\|h_{\gamma,i}\|_{\ell^{\infty}}\le1\), and 
\(
P_{\gamma,L}(h_{\gamma,i})-\mathfrak A_{\gamma,i}
= -\frac23
\bigl(h_{\gamma,i+1}-h_{\gamma,i}\bigr)
\bigl(h_{\gamma,i+1}+2h_{\gamma,i}\bigr),
\)
by Lemma \ref{lem:mean-square-kac-gradient} we have
\begin{equation}
    \mathbb E_{f_{N,0}}
\left|
\int_0^T
\frac{\varepsilon_\gamma^2}{\alpha_\gamma\delta_\gamma}
\sum_iG_{\gamma,i}^{\phi}
\bigl(P_{\gamma,L}(h_{\gamma,i})-\mathfrak A_{\gamma,i}\bigr)
g_{\gamma,i}\,dt
\right|  \le C_{\phi,T} \varepsilon^{1/16}_{\gamma}\left(\log(\varepsilon^{-1}_{\gamma})\right)^{1/2}.
\end{equation}
Therefore
\[
\int_0^T
\mathcal D_\gamma^{\mathrm{nl}}(\phi;\sigma_t)\,dt
=
\frac{\beta}{4}
\int_0^T
\frac{\varepsilon_\gamma^2}{\alpha_\gamma\delta_\gamma}
\sum_iG_{\gamma,i}^{\phi}
\mathfrak A_{\gamma,i}(t)g_{\gamma,i}(t)\,dt
+o(1).
\]
Using discrete summation by parts on the torus, for every smooth mean-zero \(\phi\), we have
\[
\begin{aligned}
 \int_0^T \mathcal D_\gamma^{\mathrm{nl}}(\phi;\sigma^N(t/\alpha_\gamma))dt = &
\frac{\beta}{4}
\int_0^T
\frac{\varepsilon_\gamma^2}{\alpha_\gamma\delta_\gamma}
\sum_i
G_{\gamma,i}^{\phi}
\left[
Q_{\gamma,L}(h_{\gamma,i+1}(t)) - Q_{\gamma,L}(h_{\gamma,i}(t))
\right]dt + r_{\gamma,\mathrm{BG}}^\phi(T)\\
= & -\frac{\beta}{4}
\int_0^T
\frac{\varepsilon_\gamma^2}{\alpha_\gamma\delta_\gamma}
\left\langle
Q_{\gamma,L}(h_\gamma(t)),
\mathcal K_\gamma^\varepsilon\phi
\right\rangle_\gamma dt
+ r_{\gamma,\mathrm{BG}}^\phi(T) \\
= & -\frac{\beta}{4}
\int_0^T \frac{\varepsilon_\gamma^2}{\alpha_\gamma\delta_\gamma}
\left\langle
2 \delta_\gamma X_\gamma
 -\frac{2}{3}\delta_\gamma^3X_\gamma^3,
\mathcal K_\gamma^\varepsilon\phi
\right\rangle_\gamma dt
+ r_{\gamma,\mathrm{BG}}^\phi(T). \\
\end{aligned}
\]
The proof is completed.
\end{proof}

\begin{corollary}[Discrete drift form of the approximate SPDE]\label{cor:discrete-spde-drift-form}
Under Assumptions \(({\bf A})\)--\(({\bf C})\), $X_{\gamma}$ satisfies the approximate discrete SPDE 
\begin{equation}
\partial_t X_\gamma = -\nu_\gamma\Delta_{\varepsilon_\gamma}^2X_\gamma
-A_\gamma\Delta_{\varepsilon_\gamma}X_\gamma
+\chi_\gamma\Delta_{\varepsilon_\gamma} X_\gamma^3
+\dot M_\gamma+R_\gamma,
\label{eq:formal-discrete-spde-drift}
\end{equation}
where
\[
\nu_\gamma=
\frac{\mathfrak m_{2,\gamma}}{4}
\frac{\varepsilon_\gamma^4}{\alpha_\gamma\gamma^2},
\quad
\chi_\gamma=
\frac{\beta}{6}
\frac{\varepsilon_\gamma^2\delta_\gamma^2}{\alpha_\gamma}, \quad A_\gamma := \frac{\varepsilon_\gamma^2}{\alpha_\gamma} \left( \frac12-\beta\kappa_\gamma(1) - \frac{\beta}{2} \right),
\]
for every fixed \(T>0\) and \(t\in[0,T]\), the remainder
\(
R_{\gamma,-1}^{\phi}(t)\rightarrow0 \quad\text{in }L^1(\Omega).
\) as $\gamma \to 0$, and \(\dot M_\gamma\) denotes the remaining Dynkin martingale noise.
Moreover, for every mean-zero test function \(\phi\), the extension \(\widetilde R_\gamma:=\mathrm{Ext}_\gamma R_\gamma\) satisfies
\[
R_{\gamma,-1}^{\phi}(t)
:=
\int_0^t
\bigl\langle
\widetilde R_\gamma(s),(-\Delta)^{-1}\phi
\bigr\rangle\,ds \to 0
\]
in probability, uniformly for \(t\in[0,T]\).  
\end{corollary}

\begin{proof}
Let \(\phi\in C^\infty(\mathbb T)\) be mean-zero and set \(\psi_\gamma=\Delta_{\varepsilon_\gamma}^{-1}\phi\). 
Using Theorem~\ref{thm:remainder-expansion}, Theorem~\ref{thm:expand linear part}, Theorem \ref{cor:clean-nonlinear-drift}, and Assumption \(({\bf B})\), we have the Dynkin decomposition
\begin{align}
\langle X_\gamma(t),\phi\rangle_{-1,\gamma} - \langle X_\gamma(0),\phi\rangle_{-1,\gamma}
&= \int_0^t
\Bigl[ -A_\gamma\langle X_\gamma(s),\phi\rangle_\gamma
-\nu_\gamma \langle X_\gamma(s),\Delta_{\varepsilon_\gamma}\phi\rangle_\gamma +\chi_\gamma \langle X_\gamma(s)^3,\phi\rangle_\gamma \Bigr]ds \nonumber\\
&\qquad  + M_{\gamma,-1}^{\phi}(t)  + R_{\gamma,-1}^{\phi}(t),
\end{align}
Here, the martingale is given by $M_{\gamma,-1}^{\phi}(t):=M_t^\gamma(\psi_\gamma)$, and the difference between \(\langle X_\gamma^3,\mathcal K_\gamma^\varepsilon\phi\rangle_\gamma\) and \(\langle X_\gamma^3,\phi\rangle_\gamma\) is absorbed into \(R_{\gamma,-1}^{\phi}\), since \(\mathcal K_\gamma^\varepsilon\phi-\phi\to0\) for smooth \(\phi\).
The formal strong form \eqref{eq:formal-discrete-spde-drift} follows by discrete summation by parts in the \(H^{-1}\)-testing convention.
\end{proof}

\section{Convergence of the noise term}\label{sec:noise}

In this section we identify the martingale part in the Dynkin decomposition and show that,
after the conservative replacement of the local mobility, it converges to the Gaussian noise
term in the limiting stochastic Cahn--Hilliard equation. 
Unlike the \(H^{-1}\)-test \(\psi_\gamma=\Delta_{\varepsilon_\gamma}^{-1}\phi\) in Section~\ref{sec:current}, we test \( X_\gamma\) with smooth test functions in the \(L^2\) inner product as a simplification for the martingale part.
More precisely, for every smooth test function \(\phi\in C^\infty(\mathbb T)\), we study the Dynkin martingale
\[
M_t^\gamma(\phi)
:=
\langle X_\gamma(t),\phi\rangle_\gamma
-
\langle X_\gamma(0),\phi\rangle_\gamma
-
\int_0^t
\Bigl(\alpha_\gamma^{-1}\mathscr L_\gamma^N\langle X_\gamma,\phi\rangle_\gamma\Bigr)
\Bigl(\sigma^N(\alpha_\gamma^{-1}s)\Bigr)\,ds.
\]
Then by using the convergence of semigroup, we will show that the associated discrete stochastic convolution \(Z_\gamma\) converges to the stochastic convolution term
\(
Z(t):= \sigma_{*}\int_0^t S(t-r) \nabla \cdot\xi(r)dr.
\)

\subsection{Convergence of the discrete martingale}

The goal of this subsection is twofold:

\begin{enumerate}
\item[(i)] to compute the predictable quadratic variation of \(M^\gamma(\phi)\) and identify
its limit as
\(
\sigma_*^2\,t\int_{\mathbb T}|\nabla\phi(x)|^2\,dx;
\)
\item[(ii)] to deduce, by a martingale central limit theorem, that the family
\(\{M^\gamma(\phi)\}_\gamma\) converges to the centered Gaussian martingale associated with the
conservative noise \(\sigma_*\nabla\xi\).
\end{enumerate}

The structure in the Dynkin martingale $M_t^\gamma(\phi) $ is parallel to the multiscale analysis of the drift part, but simpler: only the background mobility coefficient survives in the quadratic variation.

\begin{theorem}\label{thm:martingale-bracket-corrected-rate}
Assume, in addition to Assumptions \(({\bf A})\)--\(({\bf C})\).
For every \(\phi,\psi\in C^\infty(\mathbb T)\), we have
\[
\langle M^\gamma(\phi),M^\gamma(\psi)\rangle_t
\longrightarrow
\sigma_*^2\,t\int_{\mathbb T}\partial_x\phi(x)\,\partial_x\psi(x)\,dx
\]
in \(L^1(\mathbb P_{f_{N,0}})\), uniformly for \(t\in[0,T]\).
Equivalently, the martingale process $M_{\gamma}$ converges to the conservative Gaussian martingale $\sigma_*\nabla\cdot W(t,x)$ associated with the divergence-type noise $\sigma_*\partial_x\xi$.
\end{theorem}

\begin{proof}
We first compute the predictable quadratic variation. 
At a Kawasaki jump across the bond
\((i,i+1)\), the configuration changes from \(\sigma\) to \(\sigma^{i,i+1}\), and
\eqref{Ex:h} yields
\[
h_\gamma(\sigma^{i,i+1};k)-h_\gamma(\sigma;k)
=
d_i(\sigma)\Bigl[\kappa_\gamma(k-i-1)-\kappa_\gamma(k-i)\Bigr].
\]
Therefore
\begin{align*}
\delta_{i}\langle X_\gamma,\phi\rangle_\gamma(\sigma)
&:=
\langle X_\gamma(\sigma^{i,i+1}),\phi\rangle_\gamma
-
\langle X_\gamma(\sigma),\phi\rangle_\gamma \\
&=
\frac{\varepsilon_\gamma}{\delta_\gamma}
\sum_{k\in\Lambda_N}
\Bigl(h_\gamma(\sigma^{i,i+1};k)-h_\gamma(\sigma;k)\Bigr)\phi(\varepsilon_\gamma k) \\
&=
-\frac{\varepsilon_\gamma^2}{\delta_\gamma}\,
d_i(\sigma)\,
\nabla_{\varepsilon_\gamma}\!\bigl(\mathcal K_\gamma^\varepsilon\phi\bigr)
(\varepsilon_\gamma i).
\end{align*}
Hence, by the standard bracket formula for pure-jump Dynkin martingales,
\begin{equation}\label{eq:qv-exact}
\langle M^\gamma(\phi)\rangle_t =  \frac{\varepsilon_\gamma^4}{\alpha_\gamma\delta_\gamma^2}
\int_0^t \sum_{i \in\Lambda_N} d^2_i\big( \sigma^N(\alpha_{\gamma}^{-1}s) \big)c_i\Big(\sigma^N\big( \alpha_{\gamma}^{-1}s \big)\Big)
\Bigl[
\nabla_{\varepsilon_\gamma}(\mathcal K_\gamma^\varepsilon\phi)(\varepsilon_\gamma i)
\Bigr]^2
\,ds.
\end{equation}
We set
\(
q_i(\sigma):=d_i(\sigma)^2\,F\!\bigl(\beta\,\Delta_iH_\gamma(\sigma)\bigr),
\)
where $F(z):=\frac{1}{1+e^z}$.
The quantity \(q_i\) is the local mobility observable entering the predictable quadratic variation of the martingale.
In parallel with the second-order Boltzmann–Gibbs principle for $d^2_i$, we now consider the second-order Boltzmann–Gibbs principle for \(
q_i(\sigma).
\)

We next separate the mobility into a leading symmetric part and a small remainder.
A direct computation shows that
\[
\Delta_i H_\gamma(\sigma)
=
d_i(\sigma)\sum_{j\in\Lambda_N}
\bigl(\kappa_\gamma(i-j)-\kappa_\gamma(i+1-j)\bigr)\sigma_j
+\kappa_\gamma(1)d_i(\sigma)^2.
\]
Hence, since \(|d_i(\sigma)|\le 2\) and \(|\sigma_j|\le1\),
\[
|\Delta_i H_\gamma(\sigma)|
\le
2\sum_{j\in\Lambda_N}
\bigl|\kappa_\gamma(i-j)-\kappa_\gamma(i+1-j)\bigr|
+4\kappa_\gamma(1).
\]
Since
\(
\kappa_\gamma(z)=\gamma\,\mathfrak K(\gamma z),
\mathfrak K\in W^{1,1}(\mathbb R)\cap W^{1,\infty}(\mathbb R),
\)
the discrete \(L^1\)-difference estimate, the mean-value theorem, and
\(\mathfrak K(0)=0\) imply
\[
\sup_{i,\sigma}|\beta\Delta_i H_\gamma(\sigma)|\le C\gamma.
\]
Since the function $F(z):=\frac{1}{1+e^{z}}$ satisfies
\[
F(0)=\frac12, \quad
|F'(z)|=\frac{e^{z}}{(1+e^{z})^2}\le \frac{1}{4}
\quad\text{for all }z\in\mathbb R,
\]
the logistic map is globally Lipschitz. 

The Taylor expansion of \(F\) at $z =0$ gives $F(\beta\Delta_iH_\gamma(\sigma))
=
\frac12+r_i^\gamma(\sigma)$, where \( \sup_{i,\sigma}|r_i^\gamma(\sigma)|\le C\gamma. \)
Hence
\[
d_i(\sigma)^2c_{i,i+1}(\sigma)
=
\frac12\,d_i(\sigma)^2+\widetilde r_i^\gamma(\sigma),
\qquad
\sup_{i,\sigma}|\widetilde r_i^\gamma(\sigma)|\le C\gamma.
\]
Note that \(\widetilde r_i^\gamma\) is not a strictly local observable; it can be treated as a small uniform error term.

Using the previous decomposition,
\begin{align}\label{eq:qv-decomposition-1}
    \langle M^\gamma(\phi)\rangle_t
= &
\frac12\,\frac{\varepsilon_\gamma^4}{\alpha_\gamma\delta_\gamma^2}
\int_0^t
\sum_{i\in\Lambda_N}
d_i^2\!\left(\frac{s}{\alpha_\gamma}\right)\,
\Big|
\nabla_{\varepsilon_\gamma}(\mathcal K_\gamma^\varepsilon\phi)(\varepsilon_\gamma i)\Big|^2\,ds \nonumber  \\
& \quad + \frac{\varepsilon_\gamma^4}{\alpha_\gamma\delta_\gamma^2}
\int_0^t
\sum_{i\in\Lambda_N}
\widetilde r_i^\gamma\!\left(\sigma^N \left(\frac{s}{\alpha_\gamma}\right)\right)\,
\Big|
\nabla_{\varepsilon_\gamma}(\mathcal K_\gamma^\varepsilon\phi)(\varepsilon_\gamma i)
\Big|^2\,ds \nonumber \\
= & \frac{1}{2}\mathcal Q_\gamma(\phi;t) + \mathcal R_\gamma^{\mathrm{mob}}(\phi;t).
\end{align}
We first show remainder term $\mathcal R_\gamma^{\mathrm{mob}}(\phi;t)$ is negligible.
Using \(\sup_{i,\sigma}|\widetilde r_i^\gamma(\sigma)|\le C\gamma\),
\(|\Lambda_N|\simeq\varepsilon_\gamma^{-1}\), and
\(\sup_i|\nabla_{\varepsilon_\gamma}(\mathcal K_\gamma^\varepsilon\phi)
(\varepsilon_\gamma i)|\le C_\phi\), as
\(\gamma\,\frac{\varepsilon_\gamma^3}{\alpha_\gamma\delta_\gamma^2}\to0\) we have
\[
 \big|\mathcal R_\gamma^{\mathrm{mob}}(\phi;t)\big|=\Big| \frac{\varepsilon_\gamma^4}{\alpha_\gamma\delta_\gamma^2}
\int_0^t
\sum_{i\in\Lambda_N}
\widetilde r_i^\gamma\!\left(\sigma^N \left(\frac{s}{\alpha_\gamma}\right)\right)\,
\Big| \nabla_{\varepsilon_\gamma}(\mathcal K_\gamma^\varepsilon\phi)(\varepsilon_\gamma i)\Big|^2\,ds \Big|
\le
C_\phi\,t\,
\gamma\,\frac{\varepsilon_\gamma^3}{\alpha_\gamma\delta_\gamma^2} \to 0.
\]
Note that coefficient \(J_\gamma^\phi\) in $\mathcal Q_\gamma$ is deterministic and \(\|J_\gamma^\phi\|_{\ell^\infty}\le C_\phi\).  
Hence Proposition~\ref{prop:deterministic-microscopic-to-kac-mobility-main} implies that for every fixed
interval \([0,t]\subset[0,T]\),
\begin{equation}
    \mathbb E_{f_{N,0}}
\left| \int_0^t \varepsilon_\gamma\sum_i
J_{\gamma,i}^{\phi} \left[ d_i\bigl(\sigma^N(s/\alpha_\gamma)\bigr)^2 - \Phi_{\Psi,L}^{\gamma}(m_N)
\right]ds \right| \le C_{\phi,T} \varepsilon^{1/8}_{\gamma} \to 0,
\end{equation}
where
\[
\Phi_{\Psi,L}^{\gamma}(m)
= 2+\frac1L - \left( 2+\frac1L \right)m^2+ \rho_{\gamma,L}(m).
\]
Hence, uniformly for \(t\in[0,T]\),
\[
\mathcal Q_\gamma(\phi;t)
= \Phi_{\Psi,L}^{\gamma}(m_N)\frac{\varepsilon_\gamma^4}{\alpha_\gamma\delta_\gamma^2}
\int_0^t
\sum_{i\in\Lambda_N}
\left|
\nabla_{\varepsilon_\gamma}
(\mathcal K_\gamma^\varepsilon\phi)(\varepsilon_\gamma i)
\right|^2ds
+o(1).
\]
Since \(m_N := \varepsilon_{\gamma} M(N) \to 0\) and $1/L \to 0$ as $\gamma \to 0$, 

\[
\sigma_{\gamma,L}^2
:=
\frac12\,\Phi_{\Psi,L}^{\gamma}(m_N)
\frac{\varepsilon_\gamma^3}{\alpha_\gamma\delta_\gamma^2}
\longrightarrow \sigma_*^2 .
\]

The discrete Kac convolution
approximates the identity at the level of gradients:
\[
\varepsilon_\gamma\sum_{i\in\Lambda_N}
\left|
\nabla_{\varepsilon_\gamma}
(\mathcal K_\gamma^\varepsilon\phi)(\varepsilon_\gamma i)
\right|^2
\longrightarrow
\int_{\mathbb T}|\partial_x\phi(x)|^2\,dx .
\]
Combining the estimates above, we obtain
\[
\langle M^\gamma(\phi)\rangle_t
\rightarrow
 \sigma_*^2\,t\int_{\mathbb T}|\partial_x\phi(x)|^2\,dx \quad \text{in} \quad L^1(\mathbb P_{f_{N,0}})
\]
uniformly for \(t\in[0,T]\).

Finally, at each Kawasaki jump across the bond \((i,i+1)\), we have
\begin{equation}
    \bigl|\Delta M^\gamma(\phi)\bigr|
=
\left|
\delta_{i,i+1}^\gamma\langle X_\gamma,\phi\rangle_\gamma
\right|
\le
C\,\frac{\varepsilon_\gamma^2}{\delta_\gamma}\,
\|\nabla_{\varepsilon_\gamma}(\mathcal K_\gamma^\varepsilon\phi)\|_{\ell^\infty}
\le
C_\phi\,\frac{\varepsilon_\gamma^2}{\delta_\gamma}.
\end{equation}
Since \(\varepsilon_\gamma^2/\delta_\gamma\to0\), the jumps vanish.
Martingale central limit theorem yields
\(
M^\gamma(\phi)\rightarrow M(\phi)
\)
in \(D([0,T],\mathbb R)\), where \(M(\phi)\) is a continuous centered Gaussian martingale with quadratic variation
\[
\langle M(\phi)\rangle_t =  \sigma_*^2\,t\int_{\mathbb T}|\partial_x\phi(x)|^2\,dx.
\]
The cross-variation follows by polarization:
\[
\langle M^\gamma(\phi),M^\gamma(\psi)\rangle_t
=
\frac14\Bigl(
\langle M^\gamma(\phi+\psi)\rangle_t
-
\langle M^\gamma(\phi-\psi)\rangle_t
\Bigr) \rightarrow
 \sigma_*^2\,t\int_{\mathbb T}\partial_x\phi\,\partial_x\psi\,dx.
\]
This completes the proof.
\end{proof}

\subsection{Convergence of the stochastic convolution in \(\mathcal C^\alpha\)}

Recall that
\(
Z_\gamma(t):=\int_0^t S_\gamma(t-r)\,dM_\gamma(r),
\)
where \(M_\gamma\) is the lattice martingale field constructed above.
We set
\(
\widetilde Z_\gamma(t):=\mathrm{Ext}_\gamma Z_\gamma(t).
\)
For \(k\in\mathcal B_\gamma\), define the Fourier mode of martingale
\[
\widehat M_\gamma(t,k):=M_t^\gamma(e_{-k}).
\]
We also denote
\[
D_\gamma(k):=\frac{e^{2\pi i k\varepsilon_\gamma}-1}{\varepsilon_\gamma},
\qquad
\vartheta_\gamma(k):=
\sum_{z\in\Lambda_N}\kappa_\gamma(z)e^{-2\pi i k\varepsilon_\gamma z}.
\]
Now we show the convergence of the Fourier-mode bracket.

\begin{lemma}\label{lem:Fourier-mode-bracket-final}
Under the assumptions of Theorem~\ref{thm:martingale-bracket-corrected-rate},
for every fixed \(K<\infty\), uniformly in
\(|k|,|\ell|\le K\),
\[
\left\langle \widehat M_\gamma(k),\overline{\widehat M_\gamma(\ell)}\right\rangle_t
=
\delta_{k,\ell}\,q_\gamma(k)\,t
+
o_\gamma(1)
\]
in \(L^1(\mathbb P_{f_{N,0}})\), uniformly for \(t\in[0,T]\), where
\[
q_\gamma(k)
:=
\sigma_{\gamma,L}^2\,|D_\gamma(k)|^2\,|\vartheta_\gamma(k)|^2,
\qquad
|q_\gamma(k)|\le C|k|^2
\quad\text{for all }k\in\mathcal B_\gamma,
\]
and
\(
q_\gamma(k) \rightarrow \sigma_*^2(2\pi k)^2
\)
locally uniformly in \(k\in\mathbb Z\).\\
Moreover, the exact brackets satisfy the uniform high-frequency bound
\begin{equation}\label{eq:fourier-mode-bracket-uniform-bound}
\left\langle \widehat M_\gamma(k)\right\rangle_T
\le C_T |D_\gamma(k)|^2|\vartheta_\gamma(k)|^2
\le C_T |k|^2,
\quad k\in\mathcal B_\gamma.
\end{equation}
\end{lemma}

\begin{proof}
Using the bracket computation in the proof of
Theorem~\ref{thm:martingale-bracket-corrected-rate} with
\(
\phi=e_{-k},
\psi=e_{\ell},
\)
we obtain
\[
\left\langle \widehat M_\gamma(k),\overline{\widehat M_\gamma(\ell)}\right\rangle_t
=
\sigma_{\gamma,L}^2\,t\,
\varepsilon_\gamma\sum_{i\in\Lambda_N}
\nabla_{\varepsilon_\gamma}(\mathcal K_\gamma^\varepsilon e_{-k})(\varepsilon_\gamma i)\,
\overline{\nabla_{\varepsilon_\gamma}(\mathcal K_\gamma^\varepsilon e_{-\ell})(\varepsilon_\gamma i)}
+
o_\gamma(1),
\]
where the error converges to zero in \(L^1(\mathbb P_{f_{N,0}})\), uniformly in \(t\in[0,T]\).

After changing of variables \(z=i-j\), we have
\begin{align*}
    (\mathcal K_\gamma^\varepsilon e_k)(\varepsilon_\gamma i) = \sum_{j\in\Lambda_N}\kappa_\gamma(i-j)e^{2\pi i k\varepsilon_\gamma j} = 
\left(
\sum_{z\in\Lambda_N}\kappa_\gamma(z)e^{-2\pi i k\varepsilon_\gamma z}
\right)
e^{2\pi i k\varepsilon_\gamma i}.
\end{align*}
Applying the discrete gradient,
\[
\nabla_{\varepsilon_\gamma}(\mathcal K_\gamma^\varepsilon e_k)(\varepsilon_\gamma i)
=
\vartheta_\gamma(k)\,
\frac{e_k(\varepsilon_\gamma(i+1))-e_k(\varepsilon_\gamma i)}{\varepsilon_\gamma}
=
\vartheta_\gamma(k)\,D_\gamma(k)\,e_k(\varepsilon_\gamma i),
\]
where the notation agrees with the definitions above, since
\(\kappa_\gamma(z)=\gamma\,\mathfrak K(\gamma z)\).

Now we show that \(\vartheta_\gamma(k)\to1\) locally uniformly.
Set $\xi_{\gamma,k}:=2\pi k\frac{\varepsilon_\gamma}{\gamma}$ so that
\[
\vartheta_\gamma(k)
=
\sum_{m\in\Lambda_N}\gamma\,\mathfrak K(\gamma m)e^{-i\xi_{\gamma,k}(\gamma m)}.
\]
Since \(\mathfrak K\in L^1(\mathbb R)\) and $\int_{\mathbb R}\mathfrak K(u)\,du=1$, its Fourier transform is continuous, and
\[
\widehat{\mathfrak K}(k) := \int_{\mathbb R}\mathfrak K(u)e^{-i\xi_{\gamma,k}u}\,du \rightarrow1 \quad \text{as} \quad \gamma \to 0
\]
locally uniformly for \(|k|\le K\). 
Applying the standard Riemann-sum estimate for the \(W^{1,1}\)-function
\(
u\mapsto \mathfrak K(u)e^{-i\xi_{\gamma,k}u},
\)
we obtain
\[
\left|
\vartheta_\gamma(k)-\int_{\mathbb R}\mathfrak K(u)e^{-i\xi_{\gamma,k}u}\,du
\right|
\leq \gamma(\|\mathfrak K' \|_{L^1} + |\xi_{\gamma,k}|\|\mathfrak K\|_{L^1}) \le
C\gamma\bigl(1+2\pi k \frac{\varepsilon_\gamma}{\gamma} \bigr).
\]
Therefore
\[
\sup_{|k|\le K}
\left|
\vartheta_\gamma(k)- 1
\right|
\rightarrow 0 \quad \text{as} \quad \frac{\varepsilon_{\gamma}}{\gamma} \to 0.
\]
Then 
\begin{equation}\label{E:jumpc}
    \sup_{i,\gamma}|\nabla_{\varepsilon_\gamma}(\mathcal K_\gamma^\varepsilon e_k)(\varepsilon_\gamma i)| =  |D_\gamma(k)||\vartheta_\gamma(k)| \leq C_k,
\end{equation}
and 
\begin{align*}
\left\langle \widehat M_\gamma(k),\overline{\widehat M_\gamma(\ell)}\right\rangle_t
&=
\sigma_{\gamma,L}^2\,t\,
\varepsilon_\gamma\sum_{i\in\Lambda_N}
\nabla_{\varepsilon_\gamma}(\mathcal K_\gamma^\varepsilon e_{-k})(\varepsilon_\gamma i)\,
\overline{\nabla_{\varepsilon_\gamma}(\mathcal K_\gamma^\varepsilon e_{-\ell})(\varepsilon_\gamma i)}
+
o_\gamma(1) \\
&=
\sigma_{\gamma,L}^2\,t\,
\vartheta_\gamma(-k)D_\gamma(-k)\,
\overline{\vartheta_\gamma(-\ell)D_\gamma(-\ell)}
\,
\varepsilon_\gamma\sum_{i\in\Lambda_N}e_{-k}(\varepsilon_\gamma i)e_{\ell}(\varepsilon_\gamma i)
+
o_\gamma(1).
\end{align*}
By the orthogonality of discrete Fourier model,
\(
\varepsilon_\gamma\sum_{i\in\Lambda_N}e_{-k}(\varepsilon_\gamma i)e_{\ell}(\varepsilon_\gamma i)
=
\delta_{k,\ell},
\)
and therefore
\[
\left\langle \widehat M_\gamma(k),\overline{\widehat M_\gamma(\ell)}\right\rangle_t
=
\delta_{k,\ell}\,
\sigma_{\gamma,L}^2\,|D_\gamma(k)|^2\,|\vartheta_\gamma(k)|^2\,t
+
o_\gamma(1).
\]

Next, note that
\[
|D_\gamma(k)|
=
\left|\frac{e^{2\pi i k\varepsilon_\gamma}-1}{\varepsilon_\gamma}\right|
=
\frac{2|\sin(\pi k\varepsilon_\gamma)|}{\varepsilon_\gamma}
\le 2\pi |k|,
\]
and for each fixed \(k\),
\(
D_\gamma(k)\longrightarrow 2\pi i k, 
|D_\gamma(k)|^2\longrightarrow (2\pi k)^2.
\)
Therefore
\[
q_\gamma(k)
=
\sigma_{\gamma,L}^2\,|D_\gamma(k)|^2\,|\vartheta_\gamma(k)|^2
\longrightarrow
\sigma_*^2(2\pi k)^2
\]
locally uniformly in \(k\in\mathbb Z\).

It remains to show the estimate \eqref{eq:fourier-mode-bracket-uniform-bound} for high-frequency tightness.
From the exact quadratic variation formula, \(0\le q_i=d_i^2c_{i,i+1}\le4\),
Assumption \(({\bf B})\), and the Fourier identity above,
\[
\begin{aligned}
\left\langle \widehat M_\gamma(k)\right\rangle_T
&\le
C\frac{\varepsilon_\gamma^4}{\alpha_\gamma\delta_\gamma^2}
\int_0^T\sum_i
\left|
\nabla_{\varepsilon_\gamma}
(\mathcal K_\gamma^\varepsilon e_{-k})(\varepsilon_\gamma i)
\right|^2ds
\\
&\le
C_T\frac{\varepsilon_\gamma^3}{\alpha_\gamma\delta_\gamma^2}
|D_\gamma(k)|^2|\vartheta_\gamma(k)|^2
\le
C_T|k|^2 .
\end{aligned}
\]
Here \(|D_\gamma(k)|\le2\pi|k|\), and
\(|\vartheta_\gamma(k)|\le C\) follows from the normalization and
non-negativity of the Kac kernel.  This completes the proof.

\end{proof}

\begin{lemma}\label{lemma:ZgammaE}
Under scale assumption, the family stochastic convolution term \(\{\widetilde Z_\gamma\}_\gamma\) has uniform bound
\begin{align}\label{E:UZ_gamma}
    \sup_{\gamma >0 } \mathbb E
\Bigl[
\sup_{t\le T}\|\widetilde Z_\gamma(t)\|_{C^\alpha(\mathbb T)}
\Bigr]  <\infty, \quad \alpha <\frac{1}{2},
\end{align}
Moreover, \(\{\widetilde Z_\gamma\}_\gamma\) is tight in
\(
C([0,T],\mathcal C^\alpha(\mathbb T))
\)
for every \(\alpha<\frac12\).
\end{lemma}

\begin{proof}
By Fourier expansion, the Littlewood--Paley block of $\widetilde Z_\gamma$ is given by
\[
\Delta_j\widetilde Z_\gamma(t,x)
=
\sum_{k\in\mathcal A_j\cap\mathcal B_\gamma}
\varrho_q(k)
\int_0^t e^{-(t-r)\lambda_\gamma(k)}\,d\widehat M_\gamma(r,k)\,e_k(x).
\]
Fix \(p\ge2\). 
Using the factorization method for stochastic convolutions combined with the Burkholder-Davis-Gundy inequality on the underlying martingales, we have
\[
\mathbb E
\Bigl[
\sup_{t\le T}\|\Delta_j\widetilde Z_\gamma(t)\|_{L^p}^p
\Bigr]
\lesssim_p
\mathbb E
\Bigl[
\Bigl(
\sup_{x\in\mathbb T}
\langle \Delta_j\widetilde Z_\gamma(\cdot,x)\rangle_T
\Bigr)^{p/2}
\Bigr].
\]
The high-frequency estimate uses the exact bracket, not the fixed-mode
convergence.  Since \(0\le q_i=d_i^2c_{i,i+1}\le4\), Assumption \(({\bf B})\),
and the discrete Fourier orthogonality give, uniformly in \(x\),
\[
\begin{aligned}
\langle \Delta_j\widetilde Z_\gamma(\cdot,x)\rangle_T
&\lesssim
\sum_{k\in\mathcal A_j\cap\mathcal B_\gamma}
\int_0^T
e^{-2(T-r)\lambda_\gamma(k)}
|D_\gamma(k)|^2|\vartheta_\gamma(k)|^2\,dr .
\end{aligned}
\]
Together with \eqref{eq:fourier-mode-bracket-uniform-bound}, the semigroup
lower bound \(\lambda_\gamma(k)\gtrsim |k|^4\), and
\(|\mathcal A_j|\sim2^q\), \(|k|\sim2^q\) on \(\mathcal A_j\), this yields
\[
\langle \Delta_j\widetilde Z_\gamma(\cdot,x)\rangle_T
\lesssim
\sum_{k\in\mathcal A_j\cap\mathcal B_\gamma}|k|^{-2}
\lesssim 2^{-q}.
\]
Therefore
\[
\sup_{\gamma >0 }
\mathbb E
\Bigl[
\sup_{t\le T}\|\Delta_j\widetilde Z_\gamma(t)\|_{L^p}^p
\Bigr]
\lesssim_p 2^{-qp/2}.
\]
Let \(\alpha'\in(0,\frac12)\). Then
\begin{align}
    \sup_{\gamma >0 } \mathbb E
\Bigl[
\sup_{t\le T}\|\widetilde Z_\gamma(t)\|_{B^{\alpha'}_{p,p}}^p
\Bigr] = & \sup_{\gamma >0 } \mathbb E
\Bigl[ \sup_{t\le T} 2^{\alpha' pj} \sum_{j\ge-1}\|\Delta_j\widetilde Z_\gamma(t)\|_{L^p}^p \Bigr] \nonumber \\
 \leq & \sum_{j\ge-1}2^{\alpha' pj}2^{-j p/2}<\infty,
\end{align}
Since the Besov embedding \(B^{\alpha'}_{p,p}(\mathbb T)\hookrightarrow \mathcal C^\alpha(\mathbb T)\) holds whenever \(\alpha<\alpha'-\frac1p\), for \(p>1\) large enough, we obtain the bound \eqref{E:UZ_gamma}.

Now we show the time increments and tightness of \(\{\widetilde Z_\gamma\}_\gamma\). For \(0\le s<t\le T\), write
\[
\Delta_j\widetilde Z_\gamma(t)-\Delta_j\widetilde Z_\gamma(s)
=
I_{\gamma,q}^{(1)}(s,t)+I_{\gamma,q}^{(2)}(s,t),
\]
where
\[
I_{\gamma,q}^{(1)}(s,t)
:=
\sum_{k\in\mathcal A_j\cap\mathcal B_\gamma}
\varrho_q(k)\int_s^t e^{-(t-r)\lambda_\gamma(k)}\,d\widehat M_\gamma(r,k)\,e_k,
\]
and
\[
I_{\gamma,q}^{(2)}(s,t)
:=
\sum_{k\in\mathcal A_j\cap\mathcal B_\gamma}
\varrho_q(k)\int_0^s
\bigl(e^{-(t-r)\lambda_\gamma(k)}-e^{-(s-r)\lambda_\gamma(k)}\bigr)\,
d\widehat M_\gamma(r,k)\,e_k.
\]
Applying Burkholder--Davis--Gundy inequality together with the same uniform
bracket estimate and semigroup estimates \ref{ass:semigroup}, we obtain
\[
\sup_\gamma
\mathbb E
\Bigl[
\|\Delta_j(\widetilde Z_\gamma(t)-\widetilde Z_\gamma(s))\|_{L^p}^p
\Bigr]
\lesssim_p
|t-s|^{\vartheta p/2}\,2^{-q(1-4\vartheta)p/2}.
\]
Now we choose \(\alpha'\) with
\(
0<\alpha'<\frac12-2\vartheta,
\)
so that
\(
\sum_{q\ge-1}2^{\alpha' pq}2^{-q(1-4\vartheta)p/2}<\infty,
\)
Then
\[
\sup_\gamma
\mathbb E
\Bigl[
\|\widetilde Z_\gamma(t)-\widetilde Z_\gamma(s)\|_{B^{\alpha'}_{p,p}}^p
\Bigr]
\lesssim
|t-s|^{\vartheta p/2}.
\]
The Kolmogorov's criterion yields tightness of \(\{\widetilde Z_\gamma\}_\gamma\) in
\(
C([0,T],B^{\alpha'}_{p,p}(\mathbb T)).
\)
Since \(B^{\alpha'}_{p,p}(\mathbb T)\hookrightarrow \mathcal C^\alpha(\mathbb T)\) for every \(\alpha<\alpha'-\frac1p\), the family \(\{\widetilde Z_\gamma\}_\gamma\) is tight in
\(
C([0,T],\mathcal C^\alpha(\mathbb T))
\)
for every \(\alpha<\frac12\).
\end{proof}

Now we show the convergence of the stochastic convolution.

\begin{theorem}
\label{thm:Zgamma-final}
For every \(\alpha<\frac12\), the discrete stochastic convolution
\(\widetilde Z_\gamma\rightarrow Z\) in \(C([0,T],\mathcal C^\alpha(\mathbb T)).\)
\end{theorem}

\begin{proof}
Lemma \ref{lemma:ZgammaE} yields that, passing to a further subsequence if necessary, we have
\(
\widetilde Z_{\gamma_n}\rightarrow \widetilde Z \)
in $C([0,T],B^\beta_{p,p}(\mathbb T))$.
To identify the limit, it is enough to show that every convergent subsequence has limit equal to \(Z\).
We now prove that \(\widetilde Z=Z\) in law.

Fix \(Q > 0\), and let
\(
\mathcal K_Q:=\{k\in\mathbb Z:\ |k|\le C\,2^Q\}.
\)
For \(n\) sufficiently large, \(\mathcal K_Q\subset \mathcal B_{\gamma_n}\).
Note that for each fixed \(k\in\mathbb Z\), the estimate \eqref{E:jumpc} implies that $M_\gamma(\cdot,k)$ satisfies the following jump condition
\[
\sup_{t\le T}|\Delta \widehat M_\gamma(t,k)| \le 2\frac{\varepsilon_\gamma^2}{\delta_\gamma} \sup_{i,\gamma}|\nabla_{\varepsilon_\gamma}(\mathcal K_\gamma^\varepsilon e_k)(\varepsilon_\gamma i)|
\le
C_k\,\frac{\varepsilon_\gamma^2}{\delta_\gamma} \to 0.
\]
Then by Lemma~\ref{lem:Fourier-mode-bracket-final} and Skorokhod's representation theorem, after replacing the processes by copies on a common probability space,
we assume that
\[
\sup_{t\le T}
\max_{k\in\mathcal K_Q}
\bigl|
\widehat M_{\gamma_n}(t,k)-\widehat M(t,k)
\bigr|
\longrightarrow0
\qquad\text{a.s.}
\]
For \(k\in\mathcal K_Q\), we define
\[
\widehat{\widetilde Z}_{\gamma_n}(t,k)
=
\int_0^t e^{-(t-r)\lambda_{\gamma_n}(k)}\,d\widehat M_{\gamma_n}(r,k),
\qquad
\widehat Z(t,k)
=
\int_0^t e^{-(t-r)(2\pi k)^4}\,d\widehat M(r,k).
\]
We claim that, for each fixed \(k\in\mathcal K_Q\),
\[
\sup_{t\le T}
\bigl|
\widehat{\widetilde Z}_{\gamma_n}(t,k)-\widehat Z(t,k)
\bigr|
\rightarrow0
\qquad\text{a.s.}
\]
We decompose
\(
\widehat{\widetilde Z}_{\gamma_n}(t,k)-\widehat Z(t,k)
=
A_n(t,k)+B_n(t,k),
\)
where
\[
A_n(t,k)
:=
\int_0^t
\Bigl(
e^{-(t-r)\lambda_{\gamma_n}(k)}
-
e^{-(t-r)(2\pi k)^4}
\Bigr)\,d\widehat M_{\gamma_n}(r,k),
\]
and
\[
B_n(t,k)
:=
\int_0^t e^{-(t-r)(2\pi k)^4}\,d\bigl(\widehat M_{\gamma_n}-\widehat M\bigr)(r,k).
\]

We first estimate \(A_n\). After integration by parts, we have
\[
\sup_{t\le T}|A_n(t,k)|
\le
\sup_{r\le T}|\widehat M_{\gamma_n}(r,k)|
\int_0^T
\left|
\lambda_{\gamma_n}(k)e^{-u\lambda_{\gamma_n}(k)}
-(2\pi k)^4e^{-u(2\pi k)^4}
\right|\,du.
\]
Since \(k\) is fixed and \(
\lambda_{\gamma_n}(k)\longrightarrow (2\pi k)^4,
\)
the integral tends to \(0\). Moreover, since \(\widehat M_{\gamma_n}(\cdot,k)\to \widehat M(\cdot,k)\) uniformly,
\(
\sup_{r\le T}|\widehat M_{\gamma_n}(r,k)|
\)
is almost surely bounded for large \(n\).
Hence
\[
\sup_{t\le T}|A_n(t,k)|\longrightarrow0
\qquad\text{a.s.}
\]

We next estimate \(B_n\). Using integration by parts, we have
\[
B_n(t,k)
=
\bigl(\widehat M_{\gamma_n}(t,k)-\widehat M(t,k)\bigr)
-
(2\pi k)^4
\int_0^t
e^{-(t-r)(2\pi k)^4}
\bigl(\widehat M_{\gamma_n}(r,k)-\widehat M(r,k)\bigr)\,dr.
\]
Hence by the Skorokhod representation, as $\gamma \to 0$,
\[
\sup_{t\le T}|B_n(t,k)|
\le
\left(1+(2\pi k)^4T\right)
\sup_{r\le T}
\bigl|
\widehat M_{\gamma_n}(r,k)-\widehat M(r,k)
\bigr| \to 0, \quad \text{a.s.} \quad 
\]
Combining with convergences of $A_n(t,k)$ and $B_n(t,k)$, we conclude that for every \(p\in[1,\infty]\) and \(j \le Q\),
\[
\sup_{t\le T}
\|\Delta_j(\widetilde Z_{\gamma_n}(t)-Z(t))\|_{L^p(\mathbb T)}
\le
C_j
\max_{k\in\mathcal A_j}
\sup_{t\le T}
\bigl|
\widehat{\widetilde Z}_{\gamma_n}(t,k)-\widehat Z(t,k)
\bigr|
\longrightarrow0
\qquad\text{a.s.}
\]
Summing over \(j \le Q\), we obtain
\[
\sup_{t\le T}
\left\|
\sum_{j \le Q}\Delta_j(\widetilde Z_{\gamma_n}(t)-Z(t))
\right\|_{B^{\alpha'}_{p,p}}
\rightarrow0
\qquad\text{a.s.}
\]
It remains to control the high-frequency tail.
By Lemma \ref{lemma:ZgammaE},
\[
\sup_n
\mathbb E
\Bigl[
\sup_{t\le T}
\sum_{j>Q}2^{\alpha' pq}\|\Delta_j\widetilde Z_{\gamma_n}(t)\|_{L^p}^p
\Bigr]
\lesssim
\sum_{j>Q}2^{\alpha' pq}2^{-qp/2}\longrightarrow0
\]
as \(Q\to\infty\).
The same estimate holds for \(Z\), since its Fourier modes satisfy
\(
\mathbb E\bigl|\widehat Z(t,k)\bigr|^2
\lesssim |k|^{-2},
\)
and the same Littlewood--Paley argument applies.
Hence, for every \(\varepsilon>0\), we first choose \(Q\) large enough, so that
\[
\mathbb P\!\left(
\sup_{t\le T} \big\|\sum_{j>Q}\Delta_j\widetilde Z_{\gamma_n}(t)
\big\|_{B^{\alpha'}_{p,p}} + \sup_{t\le T}
\big\|\sum_{j>Q}\Delta_j Z(t) \big\|_{B^{\alpha'}_{p,p}} > \varepsilon
\right)
\]
is arbitrarily small, uniformly in \(n\). Then, for this fixed \(Q\), the low-frequency part converges almost surely to zero.
Therefore
\[
\sup_{t\le T}\|\widetilde Z_{\gamma_n}(t)-Z(t)\|_{B^{ \alpha'}_{p,p}}
\rightarrow0
\qquad\text{in probability.}
\]
This shows that every subsequential limit \(\widetilde Z\) coincides with \(Z\).
Hence
\(
\widetilde Z_\gamma\rightarrow Z\) in \(C([0,T],B^{\alpha'}_{p,p}(\mathbb T)).
\)
Finally, since
\(
B^{\alpha'}_{p,p}(\mathbb T)\hookrightarrow \mathcal C^\alpha(\mathbb T)
\)
for some $\alpha'>\alpha+\frac1p$,
we conclude
$\widetilde Z_\gamma \rightarrow Z$ in $C([0,T],\mathcal C^\alpha(\mathbb T))$
for every \(\alpha<\frac12\).
\end{proof}

\section{Energy estimates for $Y_{\gamma}$}
\label{subsec:energy-estimates}

In this section, we show the uniform \(H^{-1}\)-energy estimate for \(Y_\gamma\).
Let
\(
\widetilde Y_\gamma:=\mathrm{Ext}_\gamma Y_\gamma,
\widetilde Z_\gamma:=\mathrm{Ext}_\gamma Z_\gamma,
\)
where \(Z_\gamma\) is the linear stochastic convolution and
\(Y_\gamma\) is the remainder field.
Note that the extension operator is uniformly bounded on the discrete Sobolev scales used below, namely
\[
\|\widetilde f_\gamma\|_{H^{-1}(\mathbb T)}
\lesssim
\|f_\gamma\|_{-1,\gamma},
\qquad
\|\widetilde f_\gamma\|_{H^{1}(\mathbb T)}
\lesssim
\|f_\gamma\|_{1,\gamma},
\qquad
\|\widetilde f_\gamma\|_{L^p(\mathbb T)}
\lesssim
\|f_\gamma\|_{L^p_\gamma},
\]
uniformly in \(\gamma\). 
Corollary \ref{cor:discrete-spde-drift-form} yields that \(\widetilde Y_\gamma\) satisfies
\begin{equation}\label{E:disCH}
    \partial_t \widetilde Y_\gamma =
-\nu_{\gamma }\Delta^2 \widetilde Y_\gamma
- A_\gamma\Delta(\widetilde Y_\gamma+\widetilde Z_\gamma)
+ \chi_{\gamma}\Delta\bigl((\widetilde Y_\gamma+\widetilde Z_\gamma)^3\bigr) + \widetilde R_\gamma,
\end{equation}
in weak form, where the discrete operators have been transferred to \(\mathbb T\) through the extension, and the remainders satisfy
 \begin{equation}\label{E:energy-remainder-control}
\|\widetilde R_\gamma\|_{L^2(0,T;H^{-3}(\mathbb T))}
\longrightarrow 0
\quad\text{in probability}.
\end{equation}

The weaker scalar convergence of the tested remainders is not sufficient for the
energy estimate below, because the \(H^{-1}\)-testing produces the term
\(\int_0^T\|\widetilde R_\gamma(s)\|_{H^{-3}}^2\,ds\).

\begin{proposition}\label{prop:Ygamma-Hminus1}
There exist constants \(C_T >0\), independent of \(\gamma\), such that for every \(T>0\),
\begin{align}\label{E:YH-1}
    \sup_{t\in[0,T]}\|\widetilde Y_\gamma(t)\|_{H^{-1}(\mathbb T)}^2 + &  \int_0^T\|\widetilde Y_\gamma(s)\|_{H^1(\mathbb T)}^2\,ds + \int_0^T\|\widetilde Y_\gamma(s)\|_{L^4(\mathbb T)}^4\,ds \nonumber \\
& \le C_T \left( 1+ \|\widetilde X_\gamma(0)\|_{H^{-1}(\mathbb T)}^2
+ \int_0^T\|\widetilde Z_\gamma(s)\|_{L^4 (\mathbb T)}^4\,ds
+ \int_0^T\|\widetilde R_\gamma(s)\|_{H^{-3}(\mathbb T)}^2\,ds \right).
\end{align}

In particular, by \eqref{E:energy-remainder-control}, the last term is
\(o_{\mathbb P}(1)\) and the estimate is uniform in probability.
We also have the uniform time regularity 
\begin{equation}\label{E:Ygamma-TimeR}
    \sup_{\gamma>0} \mathbb E \Bigl[ \|\partial_t \widetilde Y_\gamma\|_{L^{4/3}(0,T;H^{-3})}^{4/3} \Bigr] <\infty,
\end{equation}

provided, in addition, that
\(\sup_\gamma\mathbb E\|\widetilde R_\gamma\|_{L^{4/3}(0,T;H^{-3})}^{4/3}<\infty\).

Under the same condition, for \(0\le s<t\le T\),
\[
\|\widetilde Y_\gamma(t)-\widetilde Y_\gamma(s)\|_{H^{-3}}
\le
\int_s^t \|\partial_t\widetilde Y_\gamma(r)\|_{H^{-3}}\,dr
\le
|t-s|^{1/4}\,
\|\partial_t\widetilde Y_\gamma\|_{L^{4/3}(0,T;H^{-3})}.
\]
\end{proposition}

\begin{proof}
Testing the equation \eqref{E:disCH} against \( (-\Delta)^{-1} \widetilde Y_\gamma\), we obtain
\[
\frac12 \frac{d}{dt}\|\widetilde Y_\gamma\|_{H^{-1}(\mathbb T)}^2
+
\nu_{\gamma}\| \widetilde Y_\gamma\|_{\dot{H}^{1}(\mathbb T)}^2
=
A_\gamma\langle \widetilde Y_\gamma+ \widetilde Z_\gamma, \widetilde Y_\gamma\rangle
- \chi_{\gamma}\langle (\widetilde Y_\gamma+ \widetilde Z_\gamma)^3, \widetilde Y_\gamma\rangle
+
\langle \widetilde R_\gamma, (-\Delta)^{-1} \widetilde Y_\gamma\rangle.
\]
For the linear term, by the Poincar\'e inequality and Young's inequality, for any $\eta > 0$:
\begin{equation}\label{E:energy1}
    |A_\gamma\langle \widetilde Y_\gamma+ \widetilde Z_\gamma, \widetilde Y_\gamma\rangle|
\le
\eta\|\widetilde Y_\gamma\|_{H^{1}(\mathbb T)}^2
+
C_\eta\|\widetilde Y_\gamma\|_{H^{-1}(\mathbb T)}^2
+
C_\eta\|\widetilde Z_\gamma\|_{L^2(\mathbb T)}^2.
\end{equation}
For the nonlinear term, using the Young's inequality, one obtains the pointwise lower bound \((\widetilde Y_\gamma+\widetilde Z_\gamma)^3 \widetilde Y_\gamma \ge c_0|\widetilde Y_\gamma|^4-C|\widetilde Z_\gamma|^4\). Hence,
\begin{equation}\label{E:energy2}
    -\chi_{\gamma}\langle (\widetilde Y_\gamma+ \widetilde Z_\gamma)^3, \widetilde Y_\gamma\rangle
\le
-c_1\|\widetilde Y_\gamma\|_{L^4(\mathbb T)}^4 + C\|\widetilde Z_\gamma\|_{L^4(\mathbb T)}^4.
\end{equation}
For the remainder term, we use the \(H^{-3}\times H^3\) duality and the estimate
\(\|(-\Delta)^{-1}\widetilde Y_\gamma\|_{H^3}
\lesssim \|\widetilde Y_\gamma\|_{H^1}\):
\begin{equation}\label{E:energy3}
    |\langle \widetilde R_\gamma,(-\Delta)^{-1}\widetilde Y_\gamma\rangle|
\le
\|\widetilde R_\gamma\|_{H^{-3}(\mathbb T)} \|(-\Delta)^{-1}\widetilde Y_\gamma\|_{H^3(\mathbb T)}
\le
\eta \|\widetilde Y_\gamma\|_{H^{1}(\mathbb T)}^2
+ C_\eta \|\widetilde R_\gamma\|_{H^{-3}(\mathbb T)}^2.
\end{equation}
Collecting the estimates \eqref{E:energy1}-\eqref{E:energy3}, and choosing \(\eta>0\) sufficiently small to be absorbed by the coercive term \(\nu_\gamma\|\widetilde Y_\gamma\|_{\dot{H}^{1}}^2\), we get
\[
\frac{d}{dt}\|\widetilde Y_\gamma\|_{H^{-1}(\mathbb T)}^2
+
c\|\widetilde Y_\gamma\|_{H^1(\mathbb T)}^2
+
c\|\widetilde Y_\gamma\|_{L^4(\mathbb T)}^4
\le
C\Bigl(
1+\|\widetilde Z_\gamma\|_{L^4(\mathbb T)}^4+\|\widetilde Y_\gamma\|_{H^{-1}(\mathbb T)}^2 + \|\widetilde R_\gamma\|_{H^{-3}(\mathbb T)}^2
\Bigr).
\]
Applying Gronwall's inequality yields \eqref{E:YH-1}; the control
\eqref{E:energy-remainder-control} then makes the additional remainder contribution
negligible in probability.

To see the time regularity \eqref{E:Ygamma-TimeR}, let \(\varphi\in H^3(\mathbb T)\). Then
\[
\langle \partial_t \widetilde Y_\gamma,\varphi\rangle
=
- \nu_{\gamma}\langle \widetilde Y_\gamma,\Delta^2\varphi\rangle
-
A_\gamma\langle \widetilde Y_\gamma+\widetilde Z_\gamma,\Delta\varphi\rangle
-
\chi_{\gamma}\langle (\widetilde Y_\gamma+\widetilde Z_\gamma)^3,\Delta\varphi\rangle
+
\langle \widetilde R_\gamma,\varphi\rangle.
\]
For the bi-Laplacian term,
\[
|\langle \widetilde Y_\gamma, \Delta_{\varepsilon_\gamma}^2\varphi\rangle|
=
|\langle \nabla_{\varepsilon_\gamma} \widetilde Y_\gamma,\nabla_{\varepsilon_\gamma}\Delta_{\varepsilon_\gamma}\varphi\rangle|
\le
C\|\widetilde Y_\gamma\|_{H^1}\|\varphi\|_{H^3}.
\]
For the linear lower-order term,
\[
|\langle \widetilde Y_\gamma+\widetilde Z_\gamma,\Delta\varphi\rangle|
\le
\bigl(\|\widetilde Y_\gamma\|_{L^2}+\|\widetilde Z_\gamma\|_{L^2}\bigr)\|\varphi\|_{H^2}
\le
C\bigl(\|\widetilde Y_\gamma\|_{H^1}+\|\widetilde Z_\gamma\|_{L^\infty}\bigr)\|\varphi\|_{H^3}.
\]
For the cubic term, since in one dimension \(H^3(\mathbb T)\hookrightarrow W^{2,\infty}(\mathbb T)\),
\[
|\langle (\widetilde Y_\gamma+\widetilde Z_\gamma)^3,\Delta\varphi\rangle|
\le
\|(\widetilde Y_\gamma+\widetilde Z_\gamma)^3\|_{L^1}\,\|\Delta\varphi\|_{L^\infty}
\le
C\|\widetilde Y_\gamma+\widetilde Z_\gamma\|_{L^3}^3\,\|\varphi\|_{H^3}.
\]
By Hölder inequality and the \(L^4\)-bound on \(\widetilde Y_\gamma\),
\[
\|\widetilde Y_\gamma\|_{L^3}^3
\le
C\bigl(1+\|\widetilde Y_\gamma\|_{L^4}^3\bigr),
\]
Then combining with $\|\widetilde Z_\gamma\|_{L^\infty}
\lesssim
\|\widetilde Z_\gamma\|_{\mathcal C^\alpha}$ for every \(\alpha \in (0, \frac{1}{2})\),
we obtain
\[
\|\partial_t \widetilde Y_\gamma\|_{H^{-3}}
\le
C\Bigl(
\|\widetilde Y_\gamma\|_{H^1}
+
\|\widetilde Y_\gamma\|_{L^4}^3
+
1
+
\|\widetilde Z_\gamma\|_{\mathcal C^\alpha}^3
+
\|\widetilde R_\gamma\|_{H^{-3}}
\Bigr).
\]
Since
\(
\widetilde Y_\gamma\in L^2(0,T;H^1)\cap L^4(0,T;L^4),
\)
we have
\[
\|\widetilde Y_\gamma\|_{H^1}\in L^2(0,T),
\qquad
\|\widetilde Y_\gamma\|_{L^4}^3\in L^{4/3}(0,T).
\]
Together with the stated \(L^{4/3}(0,T;H^{-3})\)-bound on
\(\widetilde R_\gamma\), this yields the uniform
\(L^{4/3}(0,T;H^{-3})\)-bound.

Consequently, for \(0\le s<t\le T\),
\[
\|\widetilde Y_\gamma(t)-\widetilde Y_\gamma(s)\|_{H^{-3}}
\le
\int_s^t \|\partial_t\widetilde Y_\gamma(r)\|_{H^{-3}}\,dr
\le
|t-s|^{1/4}\,
\|\partial_t\widetilde Y_\gamma\|_{L^{4/3}(0,T;H^{-3})}.
\]
This gives a uniform \(1/4\)-H\"older-type modulus in time with values in \(H^{-3}\), in expectation.
\end{proof}

\section{Convergence to the stochastic Cahn--Hilliard equation}
\label{sec:limit}

In this section, we prove the convergence of \(X_\gamma\) to the stochastic Cahn--Hilliard limit. Then we show that the induced canonical Gibbs measure \(
\mu_\gamma:=(\widetilde X_\gamma)_\#\mu_{N,\gamma,\beta}
\) converges to the canonical $\phi^4_1$ measure on the conserved-mass hyperplane $V_M$ in distribution.
We denote 
\(
\widetilde X_\gamma=\widetilde Y_\gamma+\widetilde Z_\gamma,
\)
where \(\widetilde Z_\gamma\) is the Fourier extension of the linear stochastic convolution and
\(\widetilde Y_\gamma\) is the Fourier extension of the remainder field.

\begin{theorem}[\textbf{Theorem 1}]
\label{thm:Xgamma-final-convergence}
Assume that the initial data converge:
\(
\widetilde X_\gamma(0)\rightarrow X_0, \text{in }H^{-1}(\mathbb T);
\)
Then as $\gamma \to 0$, 
\(
\widetilde X_\gamma\) converges in law to \(X
\)
in
\(
L^2(0,T;L^2(\mathbb T))\cap C([0,T],H^{-3}(\mathbb T)),
\)
where \(X = Y+Z \in V_M \) is the unique weak solution of the stochastic Cahn--Hilliard equation on the conserved-mass hyperplane \( V_M := \{ \phi \in L^2: \langle 1,\phi \rangle = M \}\).
\end{theorem}

\begin{proof}
We first show the tightness and convergence of $Y_{\gamma}$ via the classical compactness argument.

Since for every $\kappa \in (0,1)$,
\(
H^1(\mathbb T)\hookrightarrow H^{1-\kappa}(\mathbb T)
\hookrightarrow H^{-3}(\mathbb T),
\)
the Aubin--Lions--Simon compactness argument yields relative compactness of
\(\{\widetilde Y_\gamma\}_\gamma\) in \(L^2(0,T;L^2(\mathbb T))\) from the uniform bounds
\[
\widetilde Y_\gamma
\ \text{bounded in }L^2(0,T;H^1),
\qquad
\partial_t\widetilde Y_\gamma
\ \text{bounded in }L^{4/3}(0,T;H^{-3}).
\]
On the other hand, the time Hölder estimate \eqref{E:Ygamma-TimeR} and the uniform
\(L^\infty(0,T;H^{-1})\)-bound imply, by the Arzel\`a--Ascoli criterion in the weak topology together with the compact embedding
\(
H^{-1}(\mathbb T) \hookrightarrow H^{-3}(\mathbb T),
\)
that the laws of \(\{\widetilde Y_\gamma\}_\gamma\) are tight in
\(
C([0,T],H^{-3}(\mathbb T)).
\)
Combining the two compactness statements proves the tightness of \(\{\widetilde Y_\gamma\}_\gamma\) in
\[
L^2(0,T;H^{1-\kappa}(\mathbb T))\cap C([0,T],H^{-3}(\mathbb T)).
\]
By Skorokhod's representation theorem, there exists a probability space $(\Omega, \mathcal{F}, \mathbb P)$ and a limit point
\[
Y \in L^\infty(0,T;H^{-1})
\cap
L^2(0,T;H^1)
\cap
L^4((0,T)\times\mathbb T),
\]
so that along a subsequence,
\[
\widetilde Y_\gamma\to Y
\qquad\text{a.s. in }L^2(0,T;L^2(\mathbb T))\cap C([0,T],H^{-3}(\mathbb T)),
\]
and $\widetilde Z_\gamma\to Z$ a.s. in $C([0,T],\mathcal C^\alpha(\mathbb T)$.
Using the interpolation inequality, the uniform \(L^4\)-bound and strong \(L^2\)-convergence imply that
\[
\widetilde X_{\gamma} = \widetilde Y_\gamma+\widetilde Z_\gamma\to Y+Z
\qquad\text{strongly in }L^3((0,T)\times\mathbb T),
\]
Therefore
\[
(\widetilde Y_\gamma+\widetilde Z_\gamma)^3
\to
(Y+Z)^3
\qquad\text{strongly in }L^1((0,T)\times\mathbb T).
\]
Now test the weak form of the \(\widetilde Y_\gamma\)-equation against any \(\varphi\in C^\infty(\mathbb T)\), integrate in time from \(0\) to \(t\), and pass to the limit.
The linear terms pass to the limit by the strong \(L^2\)-convergence of \(\widetilde Y_\gamma\) and the convergence of \(\widetilde Z_\gamma\) to \(Z\).
The cubic term passes to the limit by the \(L^1\)-convergence above.
Combining with
\(\|\widetilde R_\gamma\|_{L^1(0,T;H^{-3})}
\le T^{1/2}\|\widetilde R_\gamma\|_{L^2(0,T;H^{-3})}\to0\)
in probability, and with the initial condition,
we conclude that the limit \(Y\) satisfies the shifted equation
\[
\partial_t Y
=
- \nu\Delta^2Y - A\Delta(Y+Z)+ \chi\Delta((Y+Z)^3), \quad Y(0)=X_0,
\]
in weak sense.
By the uniqueness of the shifted equation, then all subsequences have the same limit, and therefore the full sequence converges:
\(
\widetilde Y_\gamma\rightarrow Y
\)
in
\(
L^2(0,T; H^{1-\kappa}(\mathbb T))\cap C([0,T],H^{-3}(\mathbb T)), \forall \kappa \in (0,1).
\)
Therefore,
\[
\widetilde X_\gamma=\widetilde Y_\gamma+\widetilde Z_\gamma
\rightarrow
Y+Z=:X \qquad \text{a.s. in } L^2(0,T; H^{1-\kappa}(\mathbb T))\cap C([0,T],H^{-3}(\mathbb T)).
\]
Combining with the weak formulations for \(Y\) and \(Z\), we obtain that
\[
\partial_tX = -\nu \Delta^2X- A\Delta X+ \chi \Delta(X^3)
 + \sigma_* \nabla \xi,
\qquad
X(0)=X_0.
\]
Moreover, the uniqueness of stochastic Cahn--Hilliard equation yields that every subsequence of \(\{\widetilde X_\gamma\}_\gamma\) has the same limit \(X\).
This completes the proof.
\end{proof}

Let \(
\mu_\gamma:=(\widetilde X_\gamma)_\#\mu_{N,\gamma,\beta}
\)
be the equilibrium measure of the Kac coarse-grained field $\widetilde X_\gamma$ induced by the microscopic canonical Gibbs measure $\mu_{N,\gamma,\beta}$ on the fixed total-magnetization sector. By construction of \(\widetilde X_{\gamma}\), \(\mu_\gamma\) is supported on the same conserved-mass hyperplane as \(\widetilde X_\gamma\).
By above convergence result, we also have the convergence of canonical Gibbs measures.

\begin{proposition}\label{prop:induced-gibbs-measure-Xgamma}
Define the induced equilibrium measure of \(X_\gamma\) by
\(
\mu_\gamma
:=
(\widetilde X_\gamma)_\#\mu_{N,\gamma,\beta}.
\)
Then $\mu_\gamma\rightarrow \mu$ in distribution
as $\gamma \to 0$, where \(\mu\) is the $\phi^4_1$ measure on the conserved-mass hyperplane \( V_M \).
\end{proposition}

\begin{proof}
Since the Kawasaki dynamics is reversible, and invariant, with respect to the canonical Gibbs measure
\(\mu_{N,\gamma,\beta}\), if the initial law  \(\mathrm{Law}\bigl(\sigma^N(0)\bigr) \stackrel{\mathrm{d}}{=} \mu_{N,\gamma,\beta}\), then
\[
\mathrm{Law}\bigl(\sigma^N(t)\bigr) \stackrel{\mathrm{d}}{=} \mu_{N,\gamma,\beta}
\qquad\text{for every }t\ge0.
\]
Applying the measurable map \(\sigma^N\mapsto \widetilde X_\gamma(\sigma)\), we obtain the stationary measure of \(\widetilde X_\gamma\), which is given by
\[
\mathrm{Law}\bigl(\widetilde X_\gamma(t)\bigr)
=
(\widetilde X_\gamma)_\#\mathrm{Law}\bigl(\sigma_\gamma(t)\bigr)
=
(\widetilde X_\gamma)_\#\mu_{N,\gamma,\beta}
= \mu_\gamma.
\]
Now we show the tightness of \(\{\mu_\gamma\}\).
By the uniform bound of $\widetilde Z_\gamma$ in $C^\alpha$, and the convergence of $\widetilde Z_{\gamma}$, for every $\alpha \in (0,1/2)$ and $\gamma>0$, we have
\begin{align}\label{tight-1}
\frac1T \int_0^T \mathbb{E}[\|\widetilde Z_{\gamma}(t)\|_{C^{\alpha}}] dt
\le C_1
\end{align}
for some $C_1>0$ uniform in $\gamma$ and $T$. 
Then using the uniform $H^{-1}$-energy estimate \eqref{E:YH-1} and Sobolev embedding $H^{1}(\mathbb T) \hookrightarrow C^{\frac{1}{2}}(\mathbb T)$, we have 
\begin{align}\label{tight-2}
 \frac{1}{T}\int_{0}^{T} \mathbb{E}\left[ \|\widetilde Y_{\gamma}(t)\|_{C^{\frac12}}^2\right] dt\le \frac{C}{T} \int_{0}^{T} \mathbb{E}[\|\widetilde Y_{\gamma}(t)\|_{H^1}^2]dt
\le C_2 + \frac{C_2}{T} \mathbb{E}[\|\widetilde X_{\gamma}(0)\|_{H^{-1}}^{2}],
\end{align}
for some $C_2>0$ uniform in $\gamma$ and $T$. 

Combining with \eqref{tight-1} and \eqref{tight-2}, the stochastic field $\widetilde X_{\gamma}(t)=\widetilde Y_{\gamma}(t)+ \widetilde Z_{\gamma}(t)$ satisfies
\begin{align*}
\frac{1}{T} \int_{0}^{T} \mathbb{E}[\|\widetilde X_{\gamma}(t)\|_{C^{\alpha}}^2] dt
\le C_3 + \frac{C_3}{T} \mathbb{E}[\|\widetilde X_{\gamma}(0)\|_{H^{-1}}^2],
\end{align*}
for every $\alpha \in (0,\frac{1}{2})$, and some $C_3>0$ which is independent with $\gamma$ and $T$. 
Since $\mathrm{Law}\bigl(\widetilde X_\gamma(t)\bigr) \stackrel{\mathrm{d}}{=} \mu_{\gamma}$, we can let $T\to\infty$ to obtain that \( \mathbb{E}^{\mu_{\gamma}}[\| \widetilde X_{\gamma}\|_{C^{\alpha}}^2] \le C_4\),
which is uniform in $\gamma$.
Note that the embedding $C^{\alpha} \hookrightarrow C^{\alpha'}$
is compact if $\alpha > \alpha'$.
Therefore, $\{\mu_{\gamma}\}_\gamma$ is tight on $C^{\alpha}$
for every $\alpha \in (0,1/2)$.

Now we fix a subsequence \(\gamma_n\downarrow0\) such that
\(
\mu_{\gamma_n}\rightarrow \mu_*.
\) in distribution.
By Skorokhod's representation theorem, there are $C^{\alpha}$ value random variables $\{\widetilde X_{\gamma_n}(0)\}_{n \geq 1}$ and $X(0)$, so that $\mathrm{Law}(\widetilde X_{\gamma_{n}}(0)) = \mu_{\gamma_{n}}$, $\mathrm{Law}(\widetilde X(0)) = \mu_*$, and $\widetilde X_{\gamma_{n}}(0) \to X(0)$ almost surely in $C^{\alpha}$ as $\gamma_{n} \to 0$.

Since each \(\widetilde X_{\gamma_n}\) is stationary and $\widetilde X_{\gamma} \to X$ in distribution, for every bounded continuous \(F: L^2(\mathbb T)\to\mathbb R\) and every \(t\in[0,T]\), we pass to the limit as $\gamma_n \to 0$ and obtain
\[
 \lim_{\gamma_n \to 0}\mathbb E\bigl[F(\widetilde X_{\gamma_n}(t))\bigr] = \lim_{\gamma_n \to 0}\int_{C^{\alpha}(\mathbb T)} F(\phi)\,\mu_{\gamma_n}(d\phi) =
\int_{C^{\alpha}(\mathbb T)} F(\phi)\,\mu_*(d\phi) = \mathbb E\bigl[F(X(t))\bigr]
\]
for every $t\in[0,T]$.
Thus \(\mathrm{Law}(X(t)) = \mu_*\) for every \(t\), i.e. the limit process \(X\) is stationary.
Therefore \(\mu_*\) is an invariant measure for the limiting stochastic Cahn--Hilliard equation.
By the conservation-law, the limit measure \(\mu_*\) is supported on the conserved-mass hyperplane \( V_M := \{ \phi \in L^2: \langle 1,\phi \rangle = M \}\).
Since the limiting conservative stochastic Cahn--Hilliard equation admits a unique $\phi^4_1$ Gibbs measure $\mu$ on the conserved-mass hyperplane \(V_M\), every subsequence of \(\{\mu_\gamma\}\) has the same limit measure \(\mu\). Therefore the full sequence \(\mu_\gamma\) converges weakly to the conserved-mass \(\phi^4_1\) measure.
This completes the proof.
\end{proof}

\bibliographystyle{abbrv}
\bibliography{references}

\end{document}